\documentclass[fleqn]{amsart} 

\usepackage{graphics}
\usepackage{graphicx} 
\usepackage{hyperref}
\usepackage{amsmath}
\usepackage{xcolor}
\usepackage{geometry}
\usepackage{amssymb,latexsym}

\usepackage[utf8]{inputenc}
\usepackage{calrsfs}
\usepackage{enumerate}

 \usepackage{amsfonts}
\usepackage{amsthm}

\newcommand{\cL}{{\mathcal L}}
\newcommand{\cM}{{\mathcal M}}
\newcommand{\cZ}{{\mathcal Z}}

\newcommand{\PG}{\mathrm{PG}}
 \newcommand{\FF}{\mathbb{F}}
 \newcommand{\KK}{\mathbb{K}}
 \newcommand{\ZZ}{\mathbb{Z}}
\newcommand{\hee}{\widehat{e}}
\newcommand{\wtV}{\widehat{V}}
\newcommand{\hpi}{\widehat{\pi}}
\newcommand{\ovm}{\overline{m}}
\newcommand{\ove}{\bar{e}}
\newcommand{\ovu}{\bar{u}}
\newcommand{\ovx}{\bar{x}}
\newcommand{\Tr}{\mathrm{Tr}}
\newcommand{\ch}{\mathrm{char}}
\newcommand{\Der}{\mathrm{Der}}
\newcommand{\Aut}{\mathrm{Aut}}
\newcommand{\Ker}{\mathrm{ker}}
\newcommand{\drk}{\mathrm{drank}}

\newcommand{\AnK}{A_{n,\{1,n\}}(\FF)}
\newcommand{\slnK}{\mathfrak{sl}(n+1,\FF)}
\newcommand{\SlnK}{\mathrm{SL}(n+1,\FF)}

\newcommand{\onK}{\mathfrak{o}(2n+1,\FF)}
\newcommand{\OnK}{\Omega(2n+1,\FF)}
\newcommand{\ohK}{\mathfrak{o}(2n,\FF)}
\newcommand{\OhK}{\Omega(2n,\FF)}
\newcommand{\BnK}{B_{n,2}(\FF)} 
\newcommand{\DnK}{D_{n,2}(\FF)}

\newcommand{\adj}{\mathrm{adj}}

\newcommand{\ba}{{\bf a}}
\newcommand{\bfb}{{\bf b}} 
 
\newcommand{\be}{{\bf e}} 

\newtheorem{theorem}{Theorem}[section]
\newtheorem{claim}[theorem]{Claim}

\newtheorem{corollary}[theorem]{Corollary}
\newtheorem{lemma}[theorem]{Lemma}
\newtheorem{result}[theorem]{Result}

\newtheorem{proposition}[theorem]{Proposition}
\newtheorem{remark}{Remark}

\newtheorem{definition}[theorem]{Definition} 

\begin{document}

\title[Absolute embeddings of long-root geometries of orthogonal type]{Absolutely universal projective embeddings of long-root geometries of orthogonal type}
\author{A. Pasini}
\address[A. Pasini]{Dep. Information Engineering and Mathematics, University of Siena, I-53100 Siena, Italy}
\email{pasini@unisi.it}


\begin{abstract}
Let $G$ be a split Chevalley group of type $B_n$ ($n \geq 3$) or $D_n$ ($n \geq 4$), let $\cL$ be its Lie algebra and $\Gamma$ its long-root geometry, namely $\Gamma$ is the line-grassmannian of the polar space associated to $G$. It is well known that $\cL$ hosts a projective embedding of $\Gamma$, called the adjoint embedding, which realizes $\Gamma$ as a projective variety in the projective geometry of the underlying vector space of $\cL$. It is also known that $\Gamma$ admits the absolutely universal projective embedding, which is such that all projective embeddings of $\Gamma$ can be obtained from it as projections. When the field $\FF$ over which $G$ is defined is finite, the universal embedding of $\Gamma$ is just the same as the adjoint embedding. The same is not necessarily true when $\FF$ is infinite. Indeed, as we shall prove in this paper, the universal projective embedding of $\Gamma$ lives in a direct sum $K\oplus \cL$ for an $\FF$-vector space $K$, the dual of which is isomorphic to the derivation space $\Der(\FF)$ of $\FF$, which in general is non-trivial. Accordingly, the adjoint embedding of $\Gamma$ is universal if and only if $\Der(\FF)$ is trivial, namely $\FF$ is either perfect of positive characteristic or algebraic over its minimum subfield. 
\end{abstract} 

\maketitle   

\section{The objects to be considered in this paper} 

In this paper we shall study some problems concerning projective embeddings of Lie geometries. First of all we must explain what we are talking about, what embeddings and Lie geometries are.  

\subsection{Projective embeddings}

In this subsection we recall some basics on projective embeddings of point-line geometries, referring the reader to Shult \cite{S93} and \cite[Chp. 4]{S} and Ronan \cite{R} for more on this topic. 

Let $\Gamma = (P, L)$ be a point-line geometry, with $P$ a the point-set and  $L$ as the set of lines. We assume that $\Gamma$ is connected, namely its collinearity graph is connected, and no two lines of $\Gamma$ are incident with the same set of points; thus we can regard the lines of $\Gamma$ as subsets of $P$. 

A {\em projective embedding} of $\Gamma$ (also {\em embedding} for short)  is an  injective mapping $e$ from $P$ to the point-set of the projective geometry $\PG(V)$ of a vector space $V$, such that the image $e(P)$ of $P$ spans $\PG(V)$ and every line of $\Gamma$, regarded as a subset of $P$, is mapped by $e$ onto a line of $\PG(V)$. If $\KK$ is the underlying division ring of $V$, we say that $e$ is {\em defined over} $\KK$. 

Henceforth, although $e$ maps $\Gamma$ onto a subgeometry of $\PG(V)$, we write $e:\Gamma\rightarrow V$ to mean that $e$ embeds $\Gamma$ in $\PG(V)$. We will also freely use expressions as ``$e$ embeds $\Gamma$ in $V$", ``$e$ lives in $V$" or ``$V$ hosts $e$" which, although incorrect, streamline the exposure a bit. Accordingly, we call $\dim(V)$ the {\em dimension} of $e$ and we denote it by $\dim(e)$.    

We say that a geometry is {\em embeddable} if it admits an embedding as defined above. In other words, a point-line geometry is embeddable if it can be realized as a projective variety in such a way that its lines are some (possibily not all) of the projective lines contained in that variety. Obviously, a geometry is embeddable only if no two of its lines share more than one point and all of its lines have the same size (which, in the finite case, must be equal to a prime power augmented by 1).   

\subsubsection{Covers and projections of embeddings}

With $\Gamma$ as above, for $i = 1, 2$ let $e_i:\Gamma\rightarrow V_i$ be an embedding of $\Gamma$. We say that $e_1$ {\em covers} $e_2$ (also that $e_2$ is a {\em projection} of $e_1$) if $e_1$ and $e_2$ are defined over the same division ring and there exists a semi-linear mapping $f:V_1\rightarrow V_2$ such that $e_2 =  \mathrm{pg}(f)\circ e_1$, 
where $\circ$ stands for composition and $\mathrm{pg}(f):\PG(V_1)\rightarrow\PG(V_2)$ is the morphism of projective geometries associated to $f$. 

The semi-linear mapping $f$ such that $e_2 =  \mathrm{pg}(f)\circ e_1$, if it exists, is uniquely determined by $e_1$ and $e_2$ up to scalars, namely the morphism $\mathrm{pg}(f)$ is unique. (This uniqueness property follows from the connectedness of $\Gamma$.) Note also that $f$ is surjective, as it follows from the fact that $\mathrm{pg}(f)$ maps $e_1(P)$ bijectively onto $e_2(P)$ and $e_2(P)$ spans $\PG(V_2)$. We call $\mathrm{pg}(f)$ the {\em projection} of $e_1$ onto $e_2$ and the kernel of $f$ is the {\em kernel} of this projection. Clearly, the kernel of $\mathrm{pg}(f)$ is skew with every line of $\PG(V_1)$ joining two distinct points of $e_1(P)$. Conversely, every subspace of $V_1$ with this property is the kernel of suitable projections of $e_1$.    

If $e_1$ covers $e_2$ we write $e_1\geq e_2$. If $f$ (equivalently, $\mathrm{pg}(f)$) is bijective then we say that $e_1$ and $e_2$ are {\em isomorphic} and we write $e_1\cong e_2$. If $e_1 \geq e_2\not\cong e_1$ then we say that $e_1$ {\em properly covers} $e_2$.  Clearly, $e_1\geq e_2\geq e_1$ iff $e_1\cong e_2$.  

\subsubsection{Absolutely and relatively universal embeddings}

An embedding of $\Gamma$ {\em absolutely universal} ({\em absolute} for short) if it covers all embeddings of $\Gamma$. Clearly, the absolute embedding, if it exists, is unique up to isomorphisms. It is also clear that a geometry $\Gamma$ admits the absolute embedding only if all of its embeddings are defined over the same division ring. 

An embedding $\hee:\Gamma\rightarrow\wtV$ is {\em relatively universal} if it admits no proper cover. Equivalently, for every projection $e$ of $\hee$, the embedding $\hee$ also covers all embeddings which cover $e$ (Ronan \cite{R}). If $e$ is an embedding of $\Gamma$ and $\hee \geq e$ is relatively universal then, up to isomorphism, $\hee$ is the unique relatively universal embedding that covers $e$. We call it the {\em universal cover} of $e$.    

The absolute embedding, when it exists, might be regarded as the best one. If no absolute embedding exists for a given embeddable geometry, then we must be happy with relatively universal embeddings; every embeddable geometry admits them, as stated in the next proposition, due to Ronan \cite{R}. The proof amounts to an elementary free contruction; we will recall it in Appendix \ref{AppEmb}.  

\begin{proposition}
Every embedding admits the universal cover.
\end{proposition} 

Clearly absolute embeddings are relatively universal. An absolute embedding of a given geometry $\Gamma$ is indeed the universal cover of all embeddings of $\Gamma$. So, $\Gamma$ admits the absolute embedding if and only if all of its embeddings admit the same universal cover. However this characterization is of little use to check if a given embeddable geometry does admit the absolute embedding. Sufficient conditions for the existence of the absolute embedding have been found by Kasikova and Shult \cite{KS}, but I am not going to recall them here. 

\subsubsection{Homogeneity}\label{Hom}  

Given a projective embedding $e:\Gamma\rightarrow V$ and an automorhism $g$ of $\Gamma$, the composite mapping $e\circ g$ is still an embedding of $\Gamma$ in $V$, with the same image as $e$. If $e\circ g \cong e$ then we say that $g$ {\em lifts to} $\PG(V)$ {\em through} $e$ and the unique isomorphism $e(g)$ from $e$ to $e\circ g$ is called the {\em lifting} of $g$. Clearly $e(g)$, as a collineation of $\PG(V)$, stabilizes the image $e(\Gamma)$ of $e$. Also, every automorphism of $\Gamma$ admits at most one lifting through $e$ to $\PG(V)$. If all elements of a subgroup $G$ of $\Aut(\Gamma)$ lift to $\PG(V)$ through $e$ then we say that $e$ is $G$-{\em homogeneous} (also {\em homogeneous} for short when $G = \Aut(\Gamma)$). If $e$ is $G$-homogeneous then the liftings of the elements of $G$ form an isomorphic copy $e(G)$ of $G$ in the group of all collineations of $\PG(V)$. 

Cearly, absolutely universal embeddings are homogeneous and, if an embedding $e:\Gamma\rightarrow V$ is $G$-homogeneous for some subgroup $G$ of $\Aut(\Gamma)$, then its universal cover is $G$-homogeneous as well.   

So far we have assumed that $G$ is a subgroup of $\Aut(\Gamma)$, namely $G$ acts faithfully on $\Gamma$. Suppose now that the action of $G$ on $\Gamma$ is possibly unfaithful. We say that $G$ {\em lifts to} $V$ {\em through} $e$ and that $e$ is {\em linearly} $G$-{\em homogeneous} if $G$ also acts (possibly unfaithfully) on $V$ as a group of linear transformations (hence on $\PG(V)$ too as a group of linear collineations), the kernel of the action of $G$ on $\Gamma$ is equal to the kernel of the action of $G$ on $\PG(V)$ and the group induced by $G$ on $\PG(V)$ is the lifting through $e$ of the subgroup of $\Aut(\Gamma)$ induced by $G$ on $\Gamma$. 

The next proposiion easily follows from Ronan's construction of the universal cover of an embedding (Appendix \ref{AppEmb}).  

\begin{proposition}
Let $\hee:\Gamma\rightarrow\wtV$ be the universal cover of an embedding $e:\Gamma\rightarrow V$ of a geometry $\Gamma$. Suppose that $e$ is linearly $G$-omogeneous for some group $G$ acting on $\Gamma$. Then $\hee$ is also linearly $G$-homogeneous, the lifting of $G$ to $\wtV$ stabilizes the kernel $K = \ker(\pi)$ of the projection $\pi$ of $\hee$ onto $e$ and the action of $G$ on the quotient space $\wtV/K$ is mapped by $\pi$ onto the action of $G$ on $V$.     
\end{proposition}  

\subsection{Lie geometries} 

Lie geometries arise from spherical buildings. At least two settings are possible for the theory of buildings. In this paper we follow the one originally chosen by Tits in \cite{T}. In this setting, Lie geometries are defined as follows. 

Given a building $\Delta$ of spherical type and rank at least 2, let $J$ be a nonempty subset of the set of types of $\Delta$. Following Tits \cite[Chp. 12]{T}, we  define a point-line geometry $\Delta_J$ as follows: the points of $\Delta_J$ are the flags of $\Delta$ of type $J$ while the lines are the flags of type $j^\sim\cup(J\setminus\{j\})$, where $j\in J$ and $j^\sim$ is the adjacency of $j$ in the Coxeter diagram of $\Delta$, a point $X$ and a line $Y$ of $\Delta_J$ being incident in $\Delta_J$ if they are incident as flags of $\Delta$. The geometries constructed in this way are called {\em geometries of Lie type} in the literature, also {\em Lie geometries} for short. They are connected and no two of their lines have more than one point in common.  

\subsubsection{A few examples}

1. Let $\Delta$ be a building of Coxeter type $C_n$, $n \geq 3$. Take the integers $1,2,..., n$ as types as in the picture below. If we put $J = \{1\}$ then we obtain the polar space $\Delta_1$ associated to $\Delta$ while $\Delta_n$ is a {\em dual polar space}, dual of $\Delta_1$ (here and henceforth, when $J$ is a singleton, say $J = \{j\}$, we write $\Delta_j$ for short instead of $\Delta_{\{j\}}$).  Put $J = \{2\}$; the points of $\Delta_2$ are the lines of the polar space $\Delta_1$ and the lines of $\Delta_2$ are the point-plane flags of $\Delta_1$. The geometry $\Delta_2$ is called the {\em line grassmannian} of the polar space $\Delta_1$. 

\begin{picture}(310,68)(0,0)
\put(0,30){$\bullet$}
\put(3,33){\line(1,0){47}}
\put(50,30){$\bullet$}
\put(53,33){\line(1,0){47}}
\put(100,30){$\bullet$}
\put(103,33){\line(1,0){12}}
\put(118,32){$\dots..$}
\put(146,33){\line(1,0){12}}
\put(158,30){$\bullet$}
\put(161,33){\line(1,0){47}}
\put(208,30){$\bullet$}
\put(211,32){\line(1,0){47}}
\put(211,34){\line(1,0){47}}
\put(258,30){$\bullet$}
\put(0,40){1}
\put(50,40){2}
\put(100,40){3} 
\put(148,40){$n-2$}
\put(198,40){$n-1$}
\put(258,40){$n$}
\put(47,30){\textcolor{red}{$\bigcirc$}}
\put(52,30){\textcolor{blue}{\oval(100, 20)[b]}}
\put(38,53){\textcolor{red}{points}}
\put(40,5){\textcolor{blue}{lines}}
\end{picture}

\noindent
2. Let $\Delta$ be a building of type $A_n$, namely an $n$-dimensional projective geometry. With $1, 2,..., n$ as types as in the following pictures, put $J = \{1,n\}$. The points of $\Delta_J = \Delta_{\{1,n\}}$ are the point-hyperplane flags of $\Delta$ and the lines of $\Delta_{\{1,n\}}$ are the line-hyperplanes flags and the point-subhyperplane flags of $\Delta$. 

\begin{picture}(330,48)(0,0)
\put(0,20){$\bullet$}
\put(3,23){\line(1,0){47}}
\put(50,20){$\bullet$}
\put(53,23){\line(1,0){47}}
\put(100,20){$\bullet$}
\put(103,23){\line(1,0){12}}
\put(118,22){$\dots..$}
\put(146,23){\line(1,0){12}}
\put(158,20){$\bullet$}
\put(161,23){\line(1,0){47}}
\put(208,20){$\bullet$}
\put(211,23){\line(1,0){47}}
\put(258,20){$\bullet$}
\put(0,30){1}
\put(50,30){2}
\put(100,30){3} 
\put(148,30){$n-2$}
\put(198,30){$n-1$}
\put(258,30){$n$}
\put(132,20){\textcolor{red}{\oval(260, 20)[b]}}
\put(300,20){\textcolor{red}{points}}
\end{picture}

\begin{picture}(330,56)(0,0)
\put(0,20){$\bullet$}
\put(3,23){\line(1,0){47}}
\put(50,20){$\bullet$}
\put(53,23){\line(1,0){47}}
\put(100,20){$\bullet$}
\put(103,23){\line(1,0){12}}
\put(118,22){$\dots..$}
\put(146,23){\line(1,0){12}}
\put(158,20){$\bullet$}
\put(161,23){\line(1,0){47}}
\put(208,20){$\bullet$}
\put(211,23){\line(1,0){47}}
\put(258,20){$\bullet$}
\put(0,30){1}
\put(50,30){2}
\put(100,30){3} 
\put(148,30){$n-2$}
\put(198,30){$n-1$}
\put(258,30){$n$}
\put(107,20){\textcolor{blue}{\oval(210, 20)[b]}}
\put(157,40){\textcolor{blue}{\oval(210, 20)[t]}}
\put(300,20){\textcolor{blue}{lines}}
\end{picture}

\noindent
In particular, when $n = 2$ we obtain the {\em flag-geometry} $\Delta_{\{1,2\}}$ of the projective plane $\Delta$, the chambers and the elements of $\Delta$ being respectively the points and the lines of $\Delta_{\{1,2\}}$.   

\begin{picture}(310,46)(0,0)
\put(5,20){$\bullet$}
\put(8,23){\line(1,0){47}}
\put(55,20){$\bullet$}
\put(5,30){1}
\put(55,30){2}
\put(33,24){\textcolor{red}{\oval(70, 24)}}
\put(3,20){\textcolor{blue}{$\bigcirc$}}
\put(53,20){\textcolor{blue}{$\bigcirc$}}
\end{picture}

\noindent
3. With $\Delta$ of type $D_n$ and types $1, 2,..., n$ as in the following picture, $\Delta_1$ is a polar space (of hyperbolic type) while $\Delta_{n-1}$ and $\Delta_n$ are the two {\em half-spin geometries} of $\Delta$. Put $J = \{n-1, n\}$. The points of $\Delta_{n-1,n}$ are the $\{n-1, n\}$-flags while the lines are the flags of type $\{n-2, n-1\}$ and $\{n-2, n\}$.  

\begin{picture}(310,72)(0,0)
\put(0,31){$\bullet$}
\put(3,34){\line(1,0){47}}
\put(50,31){$\bullet$}
\put(53,34){\line(1,0){12}}
\put(68,34){$\dots..$}
\put(96,34){\line(1,0){12}}
\put(108,31){$\bullet$}
\put(111,34){\line(1,0){47}}
\put(158,31){$\bullet$}
\put(160,33){\line(3,1){42}}
\put(160,33){\line(3,-1){42}}
\put(201,46){$\bullet$}
\put(201,18){$\bullet$}
\put(0,41){1}
\put(50,41){2}
\put(98,41){$n-3$}
\put(144,41){$n-2$}
\put(211,18){$n-1$}
\put(211,46){$n$}
\put(203,35){\textcolor{red}{\oval(24,40)}}
\put(240,32){\textcolor{red}{points}}
\put(204,31){\textcolor{blue}{\oval(98,32)[bl]}}
\put(128,14){\textcolor{blue}{lines}}
\put(204,40){\textcolor{blue}{\oval(98,30)[tl]}}
\put(130,57){\textcolor{blue}{lines}}
\end{picture}

\noindent
The points of $\Delta_{\{n-1,n\}}$ can be identified with the sub-generators of the polar space $\Delta_1$ while the lines of $\Delta_{\{n-1,n\}}$ consist of a generator of $\Delta_1$ and a sub-sub-generator contained in it. In short, $\Delta_{\{n-1, n\}} = \Delta'_{n-1}$, where $\Delta'$ is the (non-thick) building the elements of which are the points, the lines and the singular subspaces of $\Delta_1$. 

\subsubsection{Chevalley geometries, long-root geometries and their natural embeddings}\label{Weyl Sec}  

The following addition to the current terminology will be helpful to streamline our exposition. Henceforth, when the building $\Delta$ is associated with a split Chevalley group we say that $\Delta_J$ is a {\em Chevalley geometry}, of {\em simple}  or {\em multiple} type according to whether $J$ is a singleton or not. 

Every Chevalley geometry is embeddable. Explicitly, let $\Delta$ be associated with a Chevalley group $G$ of Dynkin type $X_n$ of rank $n$, the nodes of $X_n$ being numbered by $1, 2,..., n$ as usual. These numbers also serve as types for $\Delta$. Let $\lambda_1, \lambda_2,...., \lambda_n$ be the fundamental dominant weights of the root system associated to the Lie algebra $L(G)$ of $G$. For $\emptyset \neq J \subseteq \{1, 2,..., n\}$, put $\lambda_J := \sum_{j\in J}\lambda_j$ and let $W_J$ be the Weyl module associated with $\lambda_J$ (see Humphreys \cite[Chp. 3]{H-mod}; Humphreys \cite{H-mod} uses the name `Weyl module' only in positive characteristic, but it is convenient for us to extend it to the null characteristic case: in null characteristic the `Weyl module' $W_J$ is just the unique irreducible cyclic $L(G)$-module with maximal vector of weight $\lambda_J$). 

The geometry $\Delta_J$ admits an embedding $e_J:\Delta_J\rightarrow W_J$. Explicitly, given a point $F_0$ of $\Delta_J$, namely a $J$-flag of $\Delta$, let $G_{F_0}$ be the stabilizer of $F_0$ in $G$, with $G$ acting on $\Delta$. In the action of $G$ on $W_J$, the subgroup $G_{F_0}$ stabilizes a unique $1$-dimensional subspace $V_0$ of $W_J$ and every generator of $V_0$ can be chosen as the maximal vector of $W_J$ up to a suitable choice of the Chevalley basis of $L(G)$. Every $J$-flag $F$ of $\Delta$ is the image $F = g(F_0)$ of $F_0$ by an element $g\in G$. We put $e_J(F) = g(V_0)$. If $g_1(F_0) = g_2(F_0)$ then $g_2^{-1}g_1\in G_{F_0}$, hence $g_2^{-1}g_1(V_0) = V_0$, namely $g_1(V_0) = g_2(V_0)$. So, $e_J(F)$ is well defined for every $J$-flag $F$ of $\Delta$. The mapping $e_J$ defined in this way is indeed a projectve embedding of $\Delta_J$ in $\PG(W_J)$. We call it the {\em natural embedding} (also {\em Weyl embedding}) of $\Delta_J$. 

It is clear from the above that $e_J$ is $G$-homogeneous. Note also that if $|J| > 1$ then $W_J$ is a (possibly improper) subspace of the tensor product $\otimes_{j\in J}W_j$. Accordingly, $e_J$ is a `tensor product' of the embeddings $e_j$ for $j\in J$. 

Before going any further, we need to fix some notation. Since here we are dealing with Dynkin diagrams instead of Coxeter diagrms, we distinguish between $B_n$ (parabolic orthogonal case) and $C_n$ (symplectic case). When the diagram is irreducible, we number its nodes in the usual way (see e.g. Humphreys \cite[11.4]{H-Lie}). 

Given a Dynkin diagram $X_n$ of rank $n$ and a field $\FF$ we denote by $X_n(\FF)$ the building of type $X_n$ defined over $\FF$. Given a set $J$ of types of $X_n(\FF)$, we put $X_{n,J}(\FF) := \Delta_J$ where $\Delta = X_n(\FF)$. For instance, let $\Delta = B_n(\FF)$, the building associated with the orthogonal group $G = \mathrm{O}(2n+1,\FF)$. Then $B_{n,1}(\FF)$ is the orthogonal polar space associated to $G$ while $B_{n,2}(\FF)$ is the line-grassmanian of $B_{n,1}(\FF)$. Let $\Delta = A_n(\FF)$ ($= \PG(n,\FF)$). Then $A_{n,\{1, n\}}(\FF)$ is the geometry with the point-hyperplane flags of $\PG(n,\FF)$ as points and the line-hyperplane and point-subhyperplane flags of $\PG(n,\FF)$ as lines.   

In particular $A_{n,\{1,n\}}(\FF)$, $C_{n,1}(\FF)$, $B_{n,2}(\FF)$, $D_{n, 2}(\FF)$, $F_{4,1}(\FF)$, $E_{6,2}(\FF)$, $E_{7,1}(\FF)$, $E_{8,8}(\FF)$ and $G_{2,2}(\FF)$ are the so-called {\em long-root geometries}, this name being motivated by the fact that, for each of these geometries except $C_{n,1}(\FF)$ and $B_{2,2}(\FF)$ ($\cong C_{2,1}(\FF)$), the natural embedding is hosted by the Weyl module associated with the weight which corresponds to the highest positive long root of the appropriate root system (with the convention that when the diagram is simply laced all roots are long). This module is just the {\em adjoint module} of the group associated to the building, namely the Lie algebra of this group. Accordingly, the natural embedding of a long-root geometry other than $C_{n,1}(\FF)$ is also called the {\em adjoint embedding} of that geometry. 

The case of $C_{n,1}(\FF)$ is a bit different. The geometry $C_{n,1}(\FF)$ is a polar space and its natural (projective) embedding as defined in the previous paragraph is just its usual realization in $\PG(2n-1,\FF)$ as a symplectic polar space, but this is not the adjoint embedding of $C_{n,1}(\FF)$. The vector space $V = V(2n,\FF)$ is the same as the Weyl module $W_{\lambda_1}$ associated to first fundamental dominant weight $\lambda_1$. However, in the root system of type $C_n$ the highest positive long root corresponds to $2\lambda_1$. Accordingly, the Lie algebra corresponds to the Weyl module $W_{2\lambda_1}$, which can also be described as a symmetrized tensor product of two copies of $V = W_{\lambda_1}$. An injective mapping of the point-set of $C_{n,1}(\FF)$ into the point-set of  $\PG(W_{2\lambda_1})$ can be defined in the same way as natural embeddings have been defined for Chevalley geometries in a previous paragraph of this subsection. The mapping obtained in this way is called the {\em adjoint embedding} of $C_{n,1}(\FF)$, but it is not a projective embedding. Indeed it maps lines of $C_{n,1}(\FF)$ onto non-degenerate conics of $\PG(W_{2\lambda_1})$. Mappings like these are called {\em veronesean embeddings}. Veronesean embeddings are very interesting but we prefer not to get involved with them in this paper.  

\section{The problem we deal with in this paper} 

It is tempting to believe that nearly all embeddable Lie geometries admit the absolute embedding and, for most of them, that absolute embedding is one of those that we already know. In particular, we might believe that nearly all Chevalley geometries admit the absolute embedding and in most cases this is the same as the natural one. 

The vague phrasing of this belief allows for a few exceptions. Indeed we have known some of them for a long time. For instance, embeddable grids with lines of size at least 5 admit no absolute embedding. Not even the generalized quadrangles described by Tits in \cite[8.6, II(a)]{T} admit absolute embedding (they admit two distinct relatively universal embeddings). The long-root geometry $A_{2,\{1,2\}}(\FF)$ with $\FF$ finite but not prime admits at least as many relatively universal embeddings as the automorphisms of $\FF$, hence no absolute embedding (Thas and Van Maldeghem \cite{TVM}). 

Geometries admitting a natural embedding which however is not universal are also known. For instance, the adjont embedding of $A_{2,\{1,2\}}(\FF)$ is relatively uiversal if and only if $\FF$ is either perfect of positive characteristic or algebraic over its minimum subfield (Smith and V\"{o}lklein \cite{SV}). Dual polar spaces with lines of size 3 are also couterexamples: they admit the absolute embedding, but the latter is a proper cover of the natural one (Li \cite{Li1}, \cite {Li2}). 

Chevalley geometries of type $C_{n,k}(\FF)$ with $\ch(\FF) = 2$ also deserve to be mentioned. When $\ch(\FF) = 2$ and $k < n$ the Weyl embedding of $C_{n,k}(\FF)$ is not relatively universal, ultimately because the Weyl embedding of $\Gamma = C_{n,1}(\FF)$, which embeds $\Gamma$ in $\PG(2n-1,\FF)$ as a symplectic polar space, is a proper projection of an (absolutely universal) embedding which maps $\Gamma$ as a quadric in $\PG(d+2n-1,\FF)$, where $d$ is the degree of $\FF$ over $\FF^2 = \{t^2\}_{t\in \FF}$. If moreover $\FF$ is perfect then $C_n(\FF) = B_n(\FF)$. In this case we might state that  the natural embeddings of the Lie geometries obtained from the building $\Delta = B_n(\FF) = C_n(\FF)$ are the Weyl embeddings as defined according to the $B_n$-setting. However this doesn't remove all incoveniences. Indeed the Weyl embedding of the dual polar space $\Delta_n$ in the $B_n$-setting (namely the so-called {\em spin embedding}) is a proper projection of the Weyl embedding of $\Delta_n$ in the $C_n$-setting (which is absolutely universal, provided that $|\FF| > 2$).  

More counterexamples certainly exist but, according to the above optimistic belief, they cannot be too many. However this is definitely wrong.

For instance, suppose that the field $\FF$ admits a non-trivial automorphism. Then, as proved in \cite{Pas}, for every $n \geq 2$, no absolute embedding exists for the long-root geometry $A_{n,\{1,n\}}(\FF)$. Most likely, the same conclusion holds true for every Chevalley geometry $\Delta_J$ of multiple type defined over $\FF$. It certainly holds true when $J$ contains two types which are not joined by a multiple stroke in the Coxeter diagram (Giuzzi and Pasini \cite{GP}). 

In contrast, it follows from the main result of Kasikova and Shult \cite{KS} that every Chevalley geometry of simple type admits the absolute embedding. This result naturally leads to the next question: is the natural embedding of a Chevalley geometry relatively universal (hence absolutely universal if the absolute embedding exists)? 

This is indeed the case for many Chevalley geometries (see e.g. \cite{CGP1}), but not all of them. For instance, generalizing a result of Smith and V\"{o}lkein \cite{SV} quoted above, Cardinali, Giuzzi and Pasini  \cite{CGP2} have proved that, for any $n \geq 2$, the adjoint embedding of the long-root geometry $A_{n,\{1,n\}}(\FF)$ is relatively universal if and only if the field $\FF$ is either perfect of positive characteristic or algebraic over it minumum subfied. They also provide an explicit description of the universal cover of the adjoint embedding of $A_{n,\{1,n\}}(\FF)$. We shall recall that description in Section  \ref{A}. Here we only mention that, if $\hee$ is the universal cover of the adjoint embedding $e$ of $\AnK$, then  
\begin{equation}\label{univ-agg}
\dim(\hee) ~ = ~ \drk(\FF)+\dim(e)
\end{equation}
where $\drk(\FF)$ is the derivation rank of $\FF$ as defined in Appendix \ref{AppDer}. Hence the adjoint embedding of $\AnK$ is relatively universal if and only $\drk(\FF) = 0$, namely $\FF$ is either perfect of positive characteristic or algebraic over its minumum subfied, as mentioned above.

In this paper we shall prove that a similar result holds for long-root geometries of orthogonal type $B_{n,2}(\FF)$ with $n > 2$ or $D_{n,2}(\FF)$. Explicitly, let $\Gamma$ be the long-root geometry of type $B_{n,2}(\FF)$ with $n > 2$ or $D_{n,2}(\FF)$ and let $\hee$ be the universal cover of the adjoint embedding $e$ of $\Gamma$. Then $\dim(\hee)$ and $\dim(e)$ satisfy equation (\ref{univ-agg}). In particular, $\hee = e$ iff $\drk(\FF) = 0$. 

Recall that, in view of \cite{KS}, both $B_{n,2}(\FF)$ with $n > 2$ and $D_{n,2}(\FF)$ admit the absolutely universal embedding. Hence $\hee$, to be described in Section \ref{univ BD}, is indeed the absolute embedding of $\Gamma$. 

The similarity between the above result and the main result of \cite{CGP2} is very deep; it goes beyond the mere statements and also affects the proofs. Indeed the proofs of these two results follow the same general pattern, to be presented in Section \ref{GenConstr}. 

The rank 2 case $B_{2,2}(\FF)$ is not considered in this paper because $B_{2,2}(\FF) = C_{2,1}(\FF)$ and, as explained at the end of Subsection \ref{Weyl Sec}, the adjoint embedding of $C_{n,1}(\FF)$ is veronesean rather than projective. As was said in that subsection, I prefer to leave veronesean embeddings of $C_{n,1}(\FF)$ out of the scope of this paper. I'll keep them for future work. 

An investigation of the absolute embeddings of the long-root geometries $E_{6,2}(\FF)$, $E_{7,1}(\FF)$ and $E_{8,8}(\FF)$ has been undertaken by H. Van Maldeghem and S. Busch. From the little I know about the progress of their work, equality (\ref{univ-agg}) seems to hold true for these geometries too. Anyway, as implicit in the main result of V\"{o}lkein \cite{V2}, when $\drk(\FF) = 0$ formula (\ref{univ-agg}) (which in this case boils down to $\hee = e$) holds true for every long-root geometry defined over $\FF$ and different from $C_{n,1}(\FF)$, which is not considered in \cite{V2}. 

I don't know if there are Chevalley geometries $\Gamma$ other than long-root geometries where something like (\ref{univ-agg}) holds. On the other hand, a number of Chevalley geometries are known that are not long-root geometries and whose natural embeddings are absolutely universal. This is the case for $A_{n,k}(\FF)$ for any $k$, $B_{n,1}(\FF)$, $B_{n,n}(\FF)$ with $\ch(\FF) \neq 2$, $D_{n,k}(\FF)$ with $k\in \{1, n-1, n\}$, $C_{n, k}(\FF)$ with $k < n$ and $\ch(\FF)\neq 2$,  $C_{n,n}(\FF)$ with $|\FF| > 2$, $E_{6,k}(\FF)$ with $k\in \{1, 6\}$, $E_{7,7}(\FF)$ and $\FF_{4,4}(\FF)$; see e.g. Cooperstein \cite{Coop3} (for a survey updated at 2003, which covers all cases but $C_{n,k}$ with $k < n$ and $F_{4,4}$), Blok \cite{B} (for $C_{n,k}$) and De Schepper, Schillewaert and Van Maldeghem \cite{DSSVM}) (for $F_{4,4}$). In all of these cases the conclusion is obtained by showing that the geometry under consideration can be generated (by iteratively stretching lines through pairs of collinear points) starting from a finite set of points of size equal to the dimension of the natural embedding.      

\section{Derivations, $1$-cocycles and embeddings}\label{GenConstr}  

First of all we recall some notation from Appendix \ref{AppDer}. Given a field $\FF$, the space of the derivations of $\FF$ is called $\Der(\FF)$ and, given a derivation basis $\Omega$ of $\FF$, for every $\omega\in \Omega$ we denote by $d_\omega$ the derivation that maps $\omega$ onto $1$ and $\Omega\setminus\{\omega\}$ onto $\{0\}$. The set $\{d_\omega\}_{\omega\in\Omega}$ is an independent subset of $\Der(\FF)$; the subspace of $\Der(\FF)$ spanned by this set is called $\Der_\Omega(\FF)$. As we know from Appendix \ref{AppDer}, the space $\Der(\FF)$ is naturally isomorphic to the dual $\Der_\Omega(\FF)^*$ of $\Der_\Omega(\FF)$. In view of this isomorphism, every derivation $d$ can we represented as a sum $d = \sum_{\omega\in\Omega}\lambda_\omega d_\omega$ where $\lambda_\omega = d(\omega)$ for every $\omega\in\Omega$; conversely every sum $\sum_{\omega\in\Omega}\lambda_\omega d_\omega$ where $\lambda_\omega\in\FF$ for every $\omega\in\Omega$ represents a unique derivation of $\FF$. Note that a sum as $\sum_{\omega\in\Omega}\lambda_\omega d_\omega$ makes sense even when $\Omega$ is infinite. Indeed for every $t\in\FF$ we have $d_\omega(t) = 0$ for all but at most finitely many choices of $\omega\in\Omega$. Hence $\sum_{\omega\in\Omega}\lambda_\omega d_\omega(t)$ is always a finite sum, for every $t\in\FF$.     

Having recalled the above, we can get to the topic of this section. Throughout this section $\FF$ is a given field, $V$ an $\FF$-vector space, $\Gamma = (P, L)$ a given point-line geometry, $G$ a group acting possibly unfaithfully on $\Gamma$ as a subgroup of $\Aut(\Gamma)$ and $e:\Gamma\rightarrow V$ a linearly $G$-homogeneous projective embedding of $\Gamma$ in $V$. So, $G$ also acts (possibly unfaithfully) on $V$, thus equipping $V$ with a $G$-module structure. To fix ideas, we assume that $G$ acts on $V$ at right. Accordingly, for $v\in V$ and $g\in G$, we denote by $v\cdot g$ the image of $v$ by $g$. So, if $v\neq 0$ and the point $\langle v\rangle$ of $\PG(V)$ spanned by $v$ belongs to $e(\Gamma)$, say $\langle v\rangle = e(p)$ for some point $p \in P$, then $\langle v\cdot g\rangle = e(g(p))$. 
 
As $V$ is a right $G$-module, its dual $V^*$ is a right $G$-module as well. So, we can consider the $\FF$-space $Z^1(V^*,G)$ of the $1$-cocycles of $G$ over $V^*$ and its subspace $B^1(V^*,G)$ formed by the $1$-coboundaries. Suppose there exists an injective $\FF$-linear mapping $\gamma$ from $\Der(\FF)$ to $Z^1(V^*,G)$ which satisfies both the following: 
\begin{itemize}
\item[$(\gamma1)$] the $\gamma$-image $\gamma(\Der(\FF))$ of $\Der(\FF)$ is a complement of $B^1(V^*,G)$ in $Z^1(V^*,G)$ (hence $\gamma$ induces an isomorphism from $\Der(\FF)$ to the first cohomology group $H^1(V^*,G)$ of $G$ over $V^*$).  
\item[$(\gamma2)$] Given a derivation basis $\Omega$ of $\FF$, for every $\omega\in \Omega$ let $\gamma_\omega$ be the $\gamma$-image of $d_\omega$. Then for every $g\in G$ we have $\gamma_\omega(g) = 0 \in V^*$ for all but possibly finitely many choices of $\omega\in \Omega$.
\end{itemize}
With $V_\Omega := \Der_\Omega(\FF)\oplus V$ we can define an action $\cdot_\gamma$ of $G$ on $V_\Omega$ as follows: 
\begin{equation}\label{Intro1}
g\in G~:~ d\oplus v  ~ \longrightarrow ~ (d\oplus v)\cdot_\gamma g ~ := ~ (d + \sum_{\omega\in\Omega}\langle \gamma_\omega(g^{-1}),  v) d_\omega\rangle\oplus v\cdot g 
\end{equation}
where $v\in V$, $d\in \Der_\Omega(\FF)$,  $v\cdot g$ is the image of $v$ by $g$ in the $G$-module $V$ and by $(\langle\gamma_\omega(g^{-1}), v\rangle$ we mean the value taken by the linear functional $\gamma_\omega(g^{-1})$ at $v$.  Note that the sum $\sum_{\omega\in\Omega}\langle \gamma_\omega(g^{-1}), v\rangle d_\omega$ is well defined even when $\Omega$ is infinite, thanks to property $(\gamma2)$ of $\gamma$. 

For every $g\in G$, the mapping $g:V_\Omega\rightarrow V_\Omega$ that maps every $x\in V_\Omega$ onto $x\cdot_\gamma g$ is obviously linear. Note also  that $\gamma_\omega(g_2^{-1}g_1^{-1}) = \gamma_\omega(g_2^{-1})\cdot g_1^{-1} + \gamma_\omega(g_1^{-1})$ by the cocycle condition on $\gamma_\omega\in Z^1(V^*,G)$. Therefore
\begin{equation}\label{intro3}
\langle \gamma_\omega(g_2^{-1}g_1^{-1}), v\rangle ~=~ \langle \gamma_\omega(g_2^{-1}), v\cdot g_1\rangle +  \langle \gamma_\omega(g_2^{-1}), v\rangle.
\end{equation}  
Hence $(d\oplus v)\cdot_\gamma g_2g_1 = ((d\oplus v)\cdot_\gamma g_2)\cdot_\gamma g_1$. So, the operation $\cdot_\gamma: V_\Omega\times G \rightarrow V_\Omega$ defined by (\ref{Intro1}) yields indeed a right group action of $G$ on $V_\Omega$, thus equipping $V_\Omega$ with a right $G$-module structure. We shall denote this structure by $V_{\Omega,\gamma}$ in order to recall that in general it depends on the particular choice of $\gamma$. 

Note that $G$ acts trivially on the subspace $\Der_\Omega(\FF)$ of $V_\Omega$. So,  $V_{\Omega,\gamma}$ is a central extension of the $G$-module $V$ by the trivial $G$-module $\Der_\Omega(\FF)$. 

\begin{lemma}\label{nuovo1}
The set $X_\gamma := \{\sum_{\omega\in\Omega}\langle \gamma_\omega(g^{-1}), v\rangle d_\omega\}_{g\in G, v\in V}$ spans $\Der_\Omega(\FF)$. Moreover, for every proper subspace $K$ of $\Der_\Omega(\FF)$, the quotient $V_{\Omega, \gamma}/K$ is a central non-split extension of the $G$-module $V$.  
\end{lemma} 
{\em Proof.} By way of contradiction, suppose that $X_\gamma$ is contained in a hyperplane $H$ of $\Der_\Omega(\FF)$. Let $\lambda$ be a linear functional of $\Der_\Omega(\FF)$ with $\Ker(\lambda)  = H$ and, with $\lambda_\omega := \lambda(d_\omega)$ for every $\omega\in\Omega$, let $d = \sum_{\omega\in\Omega}\lambda_\omega d_\omega$ be the derivation corresponding to $\lambda$ in the natural isomorphism $\Der(\FF)\cong \Der_\Omega(\FF)^*$. In view of hypothesis $(\gamma1)$, the derivation $d$ corresponds to a $1$-cocycle $\gamma_d\not \in B^1(V^*,G)$ and, by Theorem \ref{thm link} of Appendix \ref{teoria}, the 1-cocycle $\gamma_d$ defines on $\FF\times V$ a non-split central extension $V_{d,\gamma}$ of the $G$-module $V$ where $G$ acts as follows:
\begin{equation}\label{Intro4}
g\in G~ ;~ (t, v)\in \FF\times V ~ \rightarrow ~ (t+\langle \gamma_d(g^{-1}), v\rangle, ~ v\cdot g).
\end{equation} 
However $\gamma$ is linear and $d = \sum_{\omega\in\Omega}\lambda_{\omega}d_\omega$. Hence
\[\begin{array}{rcccl}
\langle \gamma_d(g^{-1}), v\rangle &= & \sum_{\omega\in\Omega}\langle\gamma_\omega(g^{-1}), v\rangle\lambda_\omega & = & \\
 & = & \sum_{\omega\in\Omega}\langle\gamma_\omega(g^{-1}), v\rangle\lambda(d_\omega) & = & \lambda(\sum_{\omega\in\Omega}\langle\gamma_\omega(g^{-1}), v\rangle d_\omega) 
\end{array}\]
and $\lambda(\sum_{\omega\in\Omega}\langle\gamma_\omega(g^{-1}, v\rangle d_\omega) = 0$ because $\sum_{\omega\in\Omega}\langle\gamma_\omega(g^{-1}, v\rangle d_\omega$ belongs to $X_\gamma$, which is contained in $H = \Ker(\lambda)$. So, $\{0\}\times V$ is $G$ invariant in $V_{d, \gamma}$ while, as said above, $V_{d, \gamma}$ is a non-split extension of $V$.  We have reached a contradiction. The first part of the lemma is proved.

Turning to the second part, as $G$ centralizes $\Der_\Omega(\FF)$, for every proper subspace $K$ of $\Der_\Omega(\FF)$ the quotient $V_{\Omega, \gamma}/K$ is a $G$-module, indeed a central extension of the $G$-module $V$. By way of contradiction, suppose that the extension $V_{\Omega,\gamma}/K$ splits for some proper subspace $K$ of $ \Der_\Omega(\FF)$. Then the extension $V_{\Omega,\gamma}/H$ splits as well, for every hyperplane $H$ of $\Der_\Omega(\FF)$ containing $W$. This conclusion contradicts Theorem \ref{thm link}, as we can see by an argument similar to the one used in the first part of this proof.  \hfill $\Box$ 

\begin{definition}\label{semi-trans}
\em
We say that $G$ acts {\em semi-flag-transitively} on $\Gamma$ if acts transitively on the point-set $P$ of $\Gamma$ and, for every line $\ell$ of $\Gamma$, the stabilizer $G_\ell$ of $\ell$ in $G$ acts transitively on the set of points of $\ell$. 
\end{definition}

Clearly, $G$ is flag-transitive on $\Gamma$ if and only if it is semi-flag-transitive and admits just one orbit on the line-set $L$ of $\Gamma$. Throughout the rest of this subsection we assume that $G$ is semi-flag-transitive on $\Gamma$ and that $\gamma$ also satisfies the following:

\begin{itemize}
\item[$(\gamma3)$] There is a point $p_0\in\Gamma$ such that, if $g \in G$ fixes $p_0$ and $v_0$ is a representative vector of $e(p_0)$, then $\langle \gamma_\omega(g), v_0\rangle = 0$ for every $\omega\in\Omega$. 
\end{itemize}
With $p_0$ as above, let $e_{\gamma, p_0}$ be the mapping from the point-set $P$ of $\Gamma$ to the set of points of $\PG(V_\Omega)$ defined as follows:  for every $p \in P$, let $g\in G$ map $p_0$ onto $p$ (such an element always exist since, by assumption, $G$ is transitive on $P$); then, by definition, $e_{\gamma, p_0}(p)$ is the point of $\PG(V_\Omega)$ represented by the following vector of $V_\Omega$, where $v_0\in V$ is a representative vector of $e(p_0)$:  
\begin{equation}\label{Intro2}
(\sum_{\omega\in\Omega}\langle \gamma_\omega(g^{-1}), v_0\rangle d_\omega)\oplus (v_0\cdot g) ~ = ~ v_0\cdot_\gamma g. 
\end{equation}

\begin{lemma}\label{nuovo2}
The mapping $e_{\gamma, p_0}$ is well defined and injective. Morever, with $v_0$ as above, the set $X_{\gamma, v_0} := \{\sum_{\omega\in\Omega}\langle \gamma_\omega(g^{-1}), v_0\rangle d_\omega\}_{g\in G}$ spans $\Der_\Omega(\FF)$ and the $e_{\gamma, p_0}$-image $e_{\gamma, p_0}(P)$ of the point-set $P$ of $\Gamma$ spans $\PG(V_\Omega)$. 
\end{lemma} 
{\em Proof.} For a given $p\in P$, let $g_1, g_2\in G$ be such that $g_1(p_0) = g_2(p_0) = p$. Then $g_1 = g_0g_2$ for some $g_0$ in the stabilizer $G_{p_0}$ of $p_0$ in $G$. (Note that, since $G$ acts at right on $V$, the product $g_0g_2$ is the composition $g_2\circ g_0$, where $g_0$ is applied first and $g_2$ second.) Accordingly, $\sum_{\omega\in\Omega}\langle \gamma_\omega(g_1^{-1}), v_0\rangle d_\omega = \sum_{\omega\in\Omega}\langle \gamma_\omega(g_2^{-1}g_0^{-1}), v_0\rangle d_\omega$. However 
\[\langle \gamma_\omega(g_2^{-1}g_0^{-1}), v_0\rangle ~=~ \langle \gamma_\omega(g_2^{-1}), v_0\cdot g_0\rangle +  \langle \gamma_\omega(g_0^{-1}), v_0\rangle\]
by the cocyle condition (compare (\ref{intro3})). Moreover, $v_0\cdot g_0 = tv_0$ for some scalar $t\in \FF$ and $\langle \gamma_\omega(g_0^{-1}), v_0\rangle = 0$ by propery $(\gamma3)$. Hence $\sum_{\omega\in\Omega}\langle \gamma_\omega(g_1^{-1}), v_0\rangle d_\omega = t\cdot \sum_{\omega\in\Omega}\langle \gamma_\omega(g_2^{-1}), v_0\rangle d_\omega$, namely $v_0\cdot_\gamma g_1 =  tv_0\cdot_\gamma g_2$. Thus, we have proved that the mapping $e_{\gamma, p_0}$ is well defined. 

The injectivity of $e_{\gamma, p_0}$ follows from the injectivity of $e$. In order to prove that $X_{\gamma, v_0}$ spans $\Der_\Omega(\FF)$, note that  by (\ref{intro3}) the sum $\sum_{\omega\in\Omega}\langle \gamma_\omega(g_2^{-1}), v_0\cdot g_1\rangle d_\omega$ belongs to the span of $X_{\gamma, v_0}$ for any choice of $g_1, g_2\in G$. The conclusion follows from Lemma \ref{nuovo1}, the transitivity of $G$ on $P$ and the fact that, since $e$ embeds $\Gamma$ in $\PG(V)$, the set $e(P)$ spans $\PG(V)$. It remains to prove that $e_{\gamma, p_0}(P)$ spans $\PG(V_\Omega)$, namely the set 
\[Y_{\gamma, v_0} := \{v_0\cdot_\gamma g\}_{g\in G} ~ = ~ \{(\sum_{\omega\in\Omega}\langle \gamma_\omega(g^{-1}), v_0\rangle d_\omega)\oplus(v_0\cdot g)\}_{g\in G}\] 
spans $V_\Omega$.  Let $V_{\gamma, p_0}$ be the subspace of $V_\Omega$ spanned by $Y_{\gamma, v_0}$. Obviously $G$ stabilizes $V_{\gamma, p_0}$, which thus inherits a $G$-module structure from $V_{\Omega, \gamma}$. 

The intersection $K := \Der_\Omega(\FF)\cap V_{\gamma, p_0}$ is the kernel of the mapping induced on $V_{\gamma, p_0}$ by the projection $\pi$ of $V_\Omega$ onto $V$ along $\Der_\Omega(\FF)$. Clearly, $\pi$ maps $V_{\gamma, p_0}$ onto $V$. We shall prove that $K = \Der_\Omega(\FF)$, namely $\Der_\Omega(\FF) \subseteq V_{\gamma, p_0}$. By way of contradiction, suppose that $\Der_\Omega(\FF) \not\subseteq V_{\gamma, p_0}$ and let $H$ be a hyperlane of $\Der_\Omega(\FF)$ containing $K$. We have $(H+V_{\gamma, p_0})\cap\Der_\Omega(\FF) = H$, hence $W := H+V_{\gamma, p_0}$ is a $G$-invariant proper subspace of $V_\Omega$. Accordingly, $W/H$ is a $G$-invariant subspace of $V_\Omega/H$. The $G$-module $V_{\Omega,\gamma}$ induces a $G$-module structure $V_{\Omega,\gamma}/H$ on $V_\Omega/H$. The latter is a central extension of the $G$-module $(V+H)/H\cong V$ by the $1$-dimensional trivial module $\Der_\Omega(\FF)/H$. The projection of $V_\Omega/H$ onto $(V+H)/H$ along $\Der_\Omega(\FF)/H$ maps the $G$-invariant proper subspace $W/H$ of $V_\Omega/H$ onto $(V+H)/H$. On the other hand, if $d$ is the derivation corresponding to $H$ then, as a $G$-module, $V_\gamma/H$ is isomorhic to the central non-split extension $V_{d,\gamma}$ of $V$ considered in the fist part of the proof of Lemma \ref{nuovo1}. Recall that $V_{d,\gamma}$ is defined over $\FF\oplus V$ via the $1$-cocycle $\gamma_d$ associated to $d$ as explained in (\ref{Intro4}). (Here, for ease of exposition, I prefer a direct sum setting to the direct product setting chosen in (\ref{Intro4}).) 

The isomorphism of $G$-modules from $V_{\Omega,\gamma}/H$ to $V_{d,\gamma}$ maps $\Der_\Omega(\FF)/H$ onto $\FF$ and $(V+H)/H$ onto $V$. In particular, it maps $v_0+H$ onto $v_0\in V$. The quotient $W/H$ is a $G$-invariant proper subspace of $V_{\Omega,\gamma}/H$ and meets $\Der_\Omega(\FF)/H$ trivally. It follows that $v_0$ is contained in a $G$-invariant proper subspace $X$ of $V_{d,\gamma}$, namely in the image of $W/H$ via the isomorphism from $V_{\Omega,\gamma}/H$ to $V_{d,\gamma}$. The set $\{v_0\cdot g\}_{g\in G}$ spans $V$. Therefore the projection of $V_{d,\gamma}$ onto $V$ along $\FF$ maps $X$ onto $V$. So, $X$ is a $G$-invariant complement of $\FF$ in $V_{d, \gamma}$.  Thus the $G$-module $V_{d, \gamma}$ splits as the direct sum of the trivial submodule $\FF$ and the submodule $X\cong V$. This conclusion contradicts the statement of Theorem \ref{thm link}

Therefore $V_{\gamma, p_0}$ contains $\Der_\Omega(\FF)$. Accordingly, the projection of $V_{\Omega}$ onto $V$ along $\Der_\Omega(\FF)$ induces an isomorphism from $V_{\gamma, p_0}/\Der_\Omega(\FF)$ to a subspace $X$ of $V$ containing $v_0$. However, the $G$-orbit $\{v_0\cdot g\}_{g\in G}$ of $v_0$ spans $V$. Consequently $X = V$, namely $V_{\gamma, p_0} = V_\Omega$.    \hfill $\Box$. 

\begin{lemma}\label{nuovo3}
Suppose that $\gamma$ also satisfies the following besides $(\gamma 1)$, $(\gamma 2)$ and $(\gamma 3)$:

\begin{itemize} 
\item[$(\gamma4)$] With $p_0$ and $v_0$ as in $(\gamma 3)$, for every orbit $O$ of $G$ on the line-set $L$ of $\Gamma$ there exists a line $\ell\in O$ through $p_0$ such that, if $g\in G$ stabilizes $\ell$ then $\langle \gamma_\omega(g^{-1}), v_0\rangle = 0$ for every $\omega\in \Omega$.
\end{itemize}
Then $e_{\gamma, p_0}$ is a linearly $G$-homogeneous embedding of $\Gamma$ in $V_\Omega$, the $G$-module defined on $V_\Omega$ by the lifting of $G$ to $V_\Omega$ through $e_{\gamma, p_0}$ being the same as $V_{\Omega,\gamma}$. Moreover $e_{\gamma, p_0}$ covers $e$ and the  projection of $V_\Omega$ onto $V$ along $\Der_\Omega(\FF)$ projects $e_{\gamma, p_0}$ onto $e$.  
\end{lemma}
{\em Proof.} Let $\{O_i\}_{i\in I}$ be the set of orbits of $G$ on $L$. For every  $i\in I$ let $\ell_{i}$ be a line through $p_0$ as required in $(\gamma 4)$. In view of the transitivity hypotheses made on $G$, the stabilizer $G_{\ell_i}$ of $\ell_{i}$ in $G$ acts transitively on $\ell_{i}$. Moreover, the embedding $e$ is linearly $G$-homogeneous by assumption. Consequently, the set $V_{i} := \{tv_0\cdot g\}_{t\in \FF, g\in G_{\ell_i}}$ is a $2$-dimensional subspace of $V$ and corresponds to the line $e(\ell_{i})$ of $\PG(V)$. As $V$ is a subspace of the vector space $V_\Omega$, the set $V_{i}$ is also a $2$-dimensional subspace of $V_\Omega$ and by $(\gamma4)$ it is stabilized by $G_{\ell_i}$ in its action on $V_\Omega$. In view of the definition of $e_{\gamma, p_0}$, the line of $\PG(V_\Omega)$ corresponding to $V_{i}$ is the $e_{\gamma, p_0}$-image $e_{\gamma, p_0}(\ell_{i})$ of $\ell_{i}$. For every $g\in G$, the $e_{\gamma, p_0}$-image of $\ell := g(\ell_{i})$ is the line of $\PG(V_\Omega)$ corresponding to the $g$-image $V_{i}\cdot _\gamma g := \{v\cdot_\gamma g\}_{v\in V_{i}}$ of $V_{i}$ with $g$ acting on $V_\Omega$ according the action $\cdot_\gamma$. So, $e_{\gamma, p_0}$ maps lines of $\Gamma$ onto lines of $\PG(V_\Omega)$. Moreover $e_{\gamma, p_0}$ is injective and $e_{\gamma, p_0}(P)$ spans $\PG(V_\Omega)$, as stated in Lemma \ref{nuovo2}. Therefore $e_{\gamma, p_0}$ is a projective embedding of $\Gamma$ in $V_\Omega$. Clearly, $e_{\gamma, p_0}$ is linearly $G$-homogeneous and the projection of $V_\Omega$ onto $V$ along $\Der_\Omega(\FF)$ maps $e_{\gamma, p_0}$ onto $e$.   \hfill $\Box$ 

\begin{definition}
\em
Under the hypotheses of Lemma \ref{nuovo3}, we call $e_{\gamma, p_0}$ the {\em embedding of the pair $(\Gamma, G)$ in $V_{\Omega, \gamma}$  based on $p_0$}. 
\end{definition}  

\begin{remark}
\em
In order to construct an embedding like $e_{\gamma, p_0}$ we only need a vector $x = d\oplus v \in V_\Omega$ satisfying both the following. \\
(1) We have $\langle v\rangle = e(p)$ for a point $p$ of $\Gamma$ and if $g\in G$ fixes $p$ then $x\cdot_\gamma g$ is proportional to $x$.  \\
(2) For every orbit $O$ of $G$ on $L$, there is a line $X$ of $\PG(V_\Omega)$ containing the point $\langle x\rangle$ and satisfying the following properties: the projection of $V_\Omega$ onto $V$ along $\Der_\Omega(\FF)$ maps $X$ onto the $e$-image $e(\ell)$ of a line $\ell\in O$ and if $g\in G$ stabilizes $\ell$ then $X\cdot_\gamma g = X$. 

Properties $(\gamma3)$ and $(\gamma4)$ provide sufficient conditions for such a vector to exist. Clearly, if $x$ satisfies (1) and (2) and $e_{\gamma, x}$ is the embedding constructed starting from it then, if $y\in V_\Omega$ represents a point of $e_{\gamma, x}(P)$ then $y$ also satisfies (1) and (2) and $e_{\gamma, y} = e_{\gamma, x}$. If $y$ satisfies both  (1) and (2) but $\langle y\rangle \not\in e_{\gamma,x}(P)$ then in general $e_{\gamma, y}\cong e_{\gamma, x}$.
\end{remark}    

\begin{definition}\label{rigid}
\em
We say that the $G$-module $V$ is {\em rigid} if the only linear mappings $f: V\rightarrow V$ such that $f(v\cdot g) = f(v)\cdot g$  for every $v\in V$ and every $g\in G$ are the scalar transformations. 
\end{definition}

We are now ready to state the main result of this section. 

\begin{theorem}\label{theo nuovo}
Given a group $G$ acting semi-flag-transitively on a geometry $\Gamma$ and a linearly $G$-homogeneous projective embedding $e:\Gamma\rightarrow V$ of $\Gamma$ defined over a given field $\FF$, suppose that there exists an injective linear mapping $\gamma:\Der(\FF)\rightarrow Z^1(V^*,G)$ which satisfies properties $(\gamma1)$, $(\gamma2)$, $(\gamma3)$ and $(\gamma4)$. With $p_0$ as in $(\gamma3)$, let $e_{\gamma, p_0}:\Gamma\rightarrow V_\Omega= \Der_\Omega(\FF)\oplus V$ be the embedding of the pair $(\Gamma,G)$ in $V_{\Omega, \gamma}$ based on $p_0$. Suppose moreover that both the followiing hold. 
\begin{itemize}
\item[$(G1)$] The $G$-module $V$ is rigid.
\item[$(G2)$]  If $e':\Gamma\rightarrow V'$ is a linearly $G$-homogeneous cover of $e$ then $G$, acting on $V'$, centralizes the kernel of the projection of $e'$ onto $e$.
\end{itemize}
Then $e_{\gamma, p_0}$ is the universal cover of $e$.  
\end{theorem} 
{\em Proof.} Let $\hee:\Gamma\rightarrow\wtV$ be the universal cover of $e$, let $\hpi$ be the proiection of $\hee$ onto $e$ and $K$ the kernel of $\hpi$. The vector space $\wtV$ can be regarded as a direct sum $K\oplus V$. Moreover, $\hee$ is linearly $G$-homogeneous, the group $G$ stabilizes $K$ and $\pi$ maps the action $\hat{\cdot}$ of $G$ on $\wtV/K$ onto the action $\cdot$ of $G$ on $V$. By $(G2)$, the group $G$ centralizes $K$. Then the action of $G$ on $\wtV$ can be described as follows:
\begin{equation}\label{Intro5}
g\in G ~ : ~ k\oplus v \in K\oplus V ~ \rightarrow ~ (k\oplus v)\hat{\cdot} g  =  (k+\kappa(g,v))\oplus v\cdot g 
\end{equation}
for a suitable function $\kappa:G\times V\rightarrow K$ such that $\kappa(g,.):V\rightarrow K$ is linear for every $g\in G$ and satisfies the follwing `cocyle like' condition: 
\begin{equation}\label{Intro6}
\kappa(g_1g_2, v) ~ = ~ \kappa(g_1, v) + \kappa(g_2, v\cdot g_1), ~~ \forall v\in V, ~\forall g_1, g_2\in G. 
\end{equation} 
In particular, with $g_1 = g_2 = 1$ condition (\ref{Intro6}) implies that $\kappa(1,v) = 0$ for every $v\in V$. 

Given a point $p_0$ of $\Gamma$ let $v_0$ be a representative vector of $e(p_0)$. So $\hee(p_0)$ is represented by $k_0\oplus v_0$ for some $k_0\in K$. However we can assume that $k_0 = 0$. Indeed in the decomposition $\wtV = K\oplus V$ of the vector space $\wtV$ the subspace called $V$ can be any of the complements of $K$ in $\wtV$. Neither $\hee$ not the action $\hat{\cdot}$ of $G$ on $\wtV$ are affected by changing this complement. The only effect of this change might be to change the description (\ref{Intro5}) of the action $\hat{\cdot}$ by replacing $\kappa(.,.)$ with another function, which however still satisfies (\ref{Intro6}). 

So, we can assume to have chosen the complement $V$ of $K$ in such a way that $\hee(p_0)$ is represented by a vector $v_0\in V$. With this choice of $V$, for every $g\in G$, if $p = g(p_0)$ then $\hee(p)$ is represented by the vector $\kappa(g,v_0)\oplus v_0$. However $\wtV$ is spanned by the representative vectors of the $\hee$-images of the point of $\Gamma$ and, by assumption, $G$ acts point-transitively on $\Gamma$. Therefore $\wtV$ is spanned by the set $\widehat{X} := \{\kappa(g,v_0)\oplus v_0\cdot g\}_{g\in G}$. 

By Lemma \ref{nuovo3}, the mapping $e_{\gamma,p_0}$ is a linearly $G$-homogeneous embedding of $\Gamma$ in $V_\Omega = \Der_\Omega(\FF)\oplus V$ and covers $e$. By the universality of $\hee$ relatively to $e$, the embeding $e_{\gamma,p_0}$ is a projection of $\hee$ and, if $\pi'$ and $\pi$ are the projections of $\hee$ onto $e_{\gamma,p_0}$ and respectively $e_{\gamma, p_0}$ onto $e$, then $\pi\circ\pi' = \hpi$. Accordingly, the kernel $K'$ of $\pi'$ is a subspace of $K$ and the quotient $K/K'$ is the kernel of $\pi$. Hence $\wtV/K' \cong V_\Omega$, with $K/K'$ corresponding to $\Der_\Omega(\FF)$ is this isomorphism.

Suppose for a contradiction that $K'\neq\{0\}$ and let $H$ be a hyperplane of $K$ which does not contain $K'$. The action $\hat{\cdot}$ of $G$ on $\wtV$ yields a $G$-module structure on $\wtV/H = K/H\oplus V$, which is thus a central extension of $V$ by the $1$-dimensional trivial module $K/H$. If this extension splits then  $\kappa(g, v_0) \in H$ for every $g\in G$. It follows that $\widehat{X} \subseteq H\oplus V$. However $\widehat{X}$ spans $\wtV$. Consequently $\wtV \subseteq H\oplus V$, which of course cannot be. 

Therefore the extension $\wtV/H$ does not split. By Theorem \ref{thm link}, this extension is described by a non-zero element $F_H$ of $H^1(V^*,G)$. Explictily, given a representative $f$ of $F_H$ in $Z^1(V^*,G)$, the action of $G$ on $\wtV/H$ admits the following description:  
\begin{equation}\label{Intro7}
g\in G~:~ (t,v)\in \FF\times V ~ \rightarrow ~ (t+\langle f(g^{-1}), v\rangle, ~ v\cdot g).  
\end{equation}
(We have replaced the direct sum setting with a direct product setting as in formula (\ref{Intro4}) and $K/H$ with $\FF$ in order to stay consistent with the notation of Theorem \ref{thm link} of Appendix \ref{teoria}) By $(\gamma1)$, the function $\gamma:\Der(\Omega)\rightarrow Z^1(V^*,G)$ yields an isomorphism between $\Der(\Omega)$ and $H^1(V^*,G)$. We know that $\Der(\FF)$ is canonically isomorphic to $\Der_\Omega(\FF)^*$ and, as we have seen in the proof of Lemma \ref{nuovo1}, if $\lambda_d\in\Der_\Omega(\FF)^*$ corresponds to $d\in \Der(\Omega)$ in this isomorphism, then the action of $G$ on $V_\Omega/\ker(\lambda_d)$ is described by formula (\ref{Intro4}). Consequently, if $d_H\in \Der(\FF)$ corresponds to $F_H$ in the isomorphism $\Der(\FF)\cong H^1(V^*,G)$, then $\gamma_{d_H}$ is a representative of $F_H$ in $Z^1(V^*,G)$ and (\ref{Intro4}) is the same as (\ref{Intro7}), but with $f$ renamed as $\gamma_{d_H}$. Hence the extensions $V_{\Omega,\gamma}/\ker(\lambda_{d_H})$ (described by (\ref{Intro4})) and $\wtV/H$ (described by (\ref{Intro7})) are isomorphic. However, $V_{\Omega,\gamma} = \wtV/K'$ and $\ker(\lambda_{d_H}) = H'/K'$ for a hyperplane $H'$ of $K$ containing $K'$. Accordingly, $V_{\Omega,\gamma}/\ker(\lambda_{d_H}) \cong \wtV/H'$ as $G$-modules. However $V_{\Omega,\gamma}/\ker(\lambda_{d_H}) \cong \wtV/H$. Therefore $\wtV/H'$ and $\wtV/H$ are isomorphic central extension of the $G$-module $V$ obtained from distinct hyperplanes $H$ and $H'$ of $K$. In view of Theorem \ref{M* lemma} of Appendix \ref{teoria}, this can happen only if the $G$-module $V$ is non-rigid. However $V$ is rigid by assumption $(G1)$. We have reached a contradiction, which forces us to conclude that $K' = 0$. Hence $\wtV = V_\Omega$, namely $\hee = e_{\gamma,p_0}$.       \hfill $\Box$

\bigskip

The following is an obvious consequence of Theorem \ref{theo nuovo}. 

\begin{corollary}\label{cor nuovo}
Under the hypotheses of Theorem \ref{theo nuovo}, if $\Gamma$ admits the absolutely universal embedding then $e_{\gamma, p_0}:\Gamma\rightarrow V_\Omega$ is the absolutely universal embedding of $\Gamma$. 
\end{corollary} 

\section{The universal cover of the adjoint embedding of the geometry $\AnK$}\label{A} 

In \cite{CGP2} the universal cover of the adjoint embedding of $\AnK$ is described. In this section we shall present that result of \cite{CGP2} in light of Theorem \ref{theo nuovo}. A large part of what we are going to say in this section is contained in \cite{CGP2}. We could just refer the reader to \cite{CGP2} for it, but we prefer to present this material here in quite a bit of detail, also including the proofs of a few statements, since most of it can be reused without major changes to study the absolute embedding of $B_{n,2}(\FF)$ and $D_{n,2}(\FF)$, which is what this paper aims to. 

We firstly describe the adjoint embedding of the geometry $\AnK$ in the Lie algebra $\slnK$ of $\SlnK$. 

\subsection{The adjoint embedding of $\AnK$}

Let $\cM := M_{n+1}(\FF)$ be the set of square matrices of order $n+1$, equipped with a Lie algebra structure with $[x,y] := xy-yx$ as the bracket product, where $xy$ is the usual row-times-column product. The group $G = \SlnK$, regarded a group of matrices, hence a subset of $\cM$, acts on $\cM$ at right by conjugation: $g\in G$ maps every $x\in \cM$ onto $x\cdot g := g^{-1}xg$ (row-times-column product of $g^{-1}$, $x$ and $g$). The group $G$, acting in this way on $\cM$, stabilizes the subalgebra $\cL = \{x\in \cM~|~ \Tr(x) = 0\}$ of the Lie algebra $\cM$. Hence $G$ also acts on $\cL$, thus equipping it with a (right) $G$-module structure. The algebra $\cL$ is the same as the Lie algebra $\slnK$ of $G$ and the action of $G$ on $\cL$ coincides with the action of $G$ on $\slnK$ as defined via exponentiation of nilpotent elements.  

The center of $G$ is the kernel of the action of $G$ on the long-root geometry $\Gamma := \AnK$ as well as on $\cL$. Recall that $G$ acts semi-flag-transitively on $\Gamma$ (in the sense of Definition \ref{semi-trans}), with just two orbits on the set of lines of $\Gamma$, corresponding to line-hyperlane flags of $\PG(n,\FF)$ and their duals. Moreover, $\Gamma$ admits a linearly $G$-homogeneous projective embedding $e:\Gamma\rightarrow\cL$, which we shall now define.    

The points of $\AnK$ are the point-hyperlane flags of $\PG(n, \FF)$. For such a flag $\{a,A\}$, let $v \in V := V(n+1, \FF)$ be a representative vector of the point $a$ and $f \in V^*$ (dual of $V$) a linear functional with the hyperlane $A$ as its kernel. Since $a\in A$, we have $f(v) = 0$. Regarded the vectors of $V$ as rows and the linear functionals as columuns, the condition $f(v) = 0$ can also be rewritten as $0 = vf$ (row-times-column product). The embedding $e$ maps $\{a,A\}$ to the point of $\PG(\cL)$ represented by the matrix $fv$ (note that $Tr(fv) = vf$ and recall that $vf = 0$ by assumption). The mapping $e$ defined in this way is indeed a projective embedding of $\Gamma$ in $\cL$. 

Regarded $G$ as a group of matrices, $G$ acts at right on both $V$ and $V^*$ according to the following rules:
\[g\in G ~ :~ \left\{\begin{array}{rcl}
v\in V & \rightarrow  & vg\\
f\in V^* & \rightarrow & g^{-1}f
\end{array}\right.\]
where $vg$ and $g^{-1}f$ are row-times-column products. Accordingly, $g$ maps the matrix $fv$ onto $g^{-1}fvg$ ($= fv\cdot g$ in the right action of $G$ on $\cL$). Thus, if $v$ represents the point $a\in \PG(V)$ and $\Ker(f) = A$ and moreover $a\in A$, so that $\{a,A\}$ is a point of $\Gamma$, then $g^{-1}fv g$ represents the $e$-image of the point $g(\{a,A\} = \{g(a), g(A)\}$  of $\Gamma$. So, the embeding $e$ is indeed linearly $G$-homogeneous. 

Regarding $\cM$ as a tensor product $\cM = V^*\otimes V$, the $e$-image of the point-set of $\Gamma$ is the set of points of $\PG(\cL)$ represented by {\em singular} pure tensors, namely pure tensors which, as matrices, have null trace. Recall that, as a $G$-module, $\cL$ is the same as the Weyl module $W_\lambda$ associated to the weight $\lambda = \lambda_1+\lambda_n$ which coincides with the highest root of the root system of type $A_n$. For every singular pure tensor $x_0\in\cL$ we can choose the Chevalley basis of the Lie algebra $\slnK = \cL$ in such a way that, with that choice, $x_0$ is a maximal vector of the Weyl module $W_\lambda = \cL$.  

For instance, let $E = (e_1, e_2,..., e_{n+1})$ be the natural basis of $V$ and $E^* = (e^*_1, e^*_2,..., e^*_{n+1})$ the basis of $V^*$ dual of $E$. For $i, j = 1, 2,...., n+1$ put $e_{i,j} = e^*_ie_j$. Choose a Chevalley basis $C$ of $\cL$ as follows:
\[C ~ = ~ \{e_{i,j}\}_{i>j}\cup\{e_{i,j}\}_{i< j}\cup\{e_{i+1,i+1}-e_{i,i}\}_{i=1}^n\]
where for $i\neq j$ the element $e_{i,j}$ spans a root subspace (associated to a positive root if $i > j$ and a negative root if $i < j$), the elements $e_{2,1}, e_{3,2},..., e_{n+1,n}$ correspond to the roots $\alpha_1, \alpha_2,..., \alpha_n$ of a given basis of the root system and $\{e_{i+1,i+1}-e_{i,i}\}_{i=1}^n$ is a basis of the Cartan subalgebra associated to the above system of root subspaces. With this choice of $C$, the matrix $e_{n+1,1}\in \cL$ is a maximal vector of $W_\lambda$. As $G$ maps Chevalley bases onto Chevalley bases and is transitive on the set of $1$-subspaces of $\cL$ spanned by (singular) pure tensors, every singular pure tensor is indeed a maximal vector of the Weyl module $\cL$ for a suitable Chevalley  basis of the Lie algebra $\cL$, as claimed above.   

It is thus clear that the embeding $e$ defined in this subsection is indeed (isomorphic to) the adjoint embedding of $\Gamma$ as defined in Section \ref{Weyl Sec}. 

\subsection{An injective linear mapping $\gamma:\Der(\FF)\rightarrow Z^1(\cL^*,G)$}\label{gamma A} 

With $\cM$ as in the previous subsection, a non-singular symmetric bilinear form $T:\cM\times\cM\rightarrow \FF$ can be defined by putting $T(x,y) = \Tr(xy)$, for every choice of $x, y\in \cM$. If $\perp$ is the orthogonality relation associated to this form, then $\cL^\perp$ is the subspace $\langle I\rangle$ of $\cM$ spanned by the identity matrix $I$. Accordingly, the dual $\cL^*$ of $\cL \subsetneq\cM$ is isomorphic to $\cM/\langle I\rangle$. For every $a\in \cM$ we define the linear functional $f_a\in\cL^*$ as follows:
\[f_a ~:~ x\in \cL ~ \rightarrow ~ T(a,x).\]
Every $f\in \cL^*$ can be obtained as $f = f_a$ for a suitable $a\in\cM$, uniquely determined by $f$ up to multiples of $I$. As $G$ fixes $I$, it also acts on $\cM/\langle I\rangle = \cL^*$ and this action is the same as dual action of $G$ on $\cL^*$. Indeed, for $g\in G$, the $g$-image $f_a\cdot g$ of $f_a$ maps $x\in\cL$ onto $f_a(x\cdot g^{-1}) = f_a(gxg^{-1})$. However $\Tr(agxg^{-1}) = \Tr(g^{-1}agx) = f_{g^{-1}ag}(x)$. The claim follows. 

\begin{remark}\label{killing}
\em
Clearly, $I\in \cL$ if and only if $\ch(\FF)$ divides $n+1$. Hence, if $\tau:\cL\times\cL\rightarrow\FF$ is the form induced by $T$ on $\cL\times\cL$, then $\tau$ is non-singular if and only if $\ch(\FF)$ does not divide $n+1$. In particular, when $\ch(\FF) = 0$ the form $\tau$ is proportional to the Killing form $\kappa$ of $\cL$. Explicitly, $\kappa = (n+1)\tau$.
\end{remark} 

For every derivation $d\in \Der(\FF)$ and every matrix $g = (g_{i,j})_{i,j=1}^{n+1}  \in G$, we put $d(g) := (d(g_{i,j}))_{i,j=1}^{n+1}$. Clearly  $d(I) = O$ (null matrix). Also, $d(g_1g_2)  = d(g_1)g_2 + g_1d(g_2)$ and $d(g^{-1}) = -g^{-1}d(g_1)g^{-1}$.

For every $d\in \Der(\FF)$ let $\gamma_d:G\rightarrow \cL^*$ be the mapping that takes every $g\in G$ to the linear functional $f_{g^{-1}d(g)}$. 

\begin{lemma}\label{A gamma coc}
We have $\gamma_d\in Z^1(\cL^*,G)$ for every $d\in \Der(\FF)$.
\end{lemma}
{\em Proof.} Clearly $\gamma_d(I) = 0$ (null element of $\cL^*$). Moreover, since $d(g_1g_2)  = d(g_1)g_2 + g_1d(g_2)$, for every choice of $g_1, g_2\in G$ and $x\in \cL$ we have
\[\begin{array}{l}
 f_{(g_1g_2)^{-1}d(g_1g_2)}(x)  ~= ~  \Tr(g_2^{-1}g_1^{-1}d(g_1)g_2x) + \Tr(g_2^{-1}g_1^{-1}g_1d(g_2)x) ~ = \\
=   \Tr(g_1^{-1}d(g_1)g_2xg_2^{-1}) + \Tr(g_2^{-1}d(g_2)x)  =  f_{g_1^{-1}d(g_1)}(x\cdot g_2^{-1}) + f_{g_2^{-1}d(g_2)}(x)  = \\
=  (f_{g_1^{-1}d(g_1)}\cdot g_2 + f_{g_2^{-1}d(g_2)})(x).  
\end{array}\]
Thus, $\gamma_d(g_1g_2) = \gamma_d(g_1)\cdot g_2 + \gamma_d(g_2)$.   \hfill $\Box$ \\
  
We denote by $\gamma$ the function from $\Der(\FF)$ to $Z^1(\cL^*, G)$ that maps $d\in \Der(\FF)$ onto $\gamma_d$. 

\begin{remark}\label{g**(-1)d(g)}
\em
Actually $g^{-1}d(g)\in \cL$ for every $d\in \Der(\FF)$ and every $g\in G$. Although we are not going to make any use of this statement in this section, we give its proof here. Indeed having this proof at hand will turn to be useful later in this paper. 

Note firstly that, if $g = g_1g_2$ and both $g_1^{-1}d(g_1)$ and $g_2^{-1}d(g_2)$ belong to $\cL$, then $g^{-1}d(g)$ also belong to $\cL$. Indeed 
\[g^{-1}d(g) = g_2^{-1}g_1^{-1}d(g_1g_2) = g_2^{-1}g_1^{-1}(d(g_1)g_2+g_1d(g_2))  = g_2^{-1}(g_1^{-1}d(g_1))g_2 + g_1^{-1}d(g_1).\] 
As $\cL$ is $G$-invariant, if $g_1^{-1}d(g_1), g_2^{-1}d(g_2) \in \cL$ then $g^{-1}d(g)\in\cL$, as claimed.  

For $i, j\in \{1, 2,..., n+1\}$ with $i\neq j$ and $t\in \FF$, let $u_{i,j} = I+te_{i,j}$. Then $U := \{u_{i,j}(t)\}_{i\neq j, ~t\in \FF}$ is a set of generators of $G$ and, for $g = u_{i,j}(t)$, we have $g^{-1}d(g) = (I-te_{i,j})d(t)e_{i,j} = d(t)e_{i,j}\in \cL$. As every element of $G$ is a product of elements of $U$, the above shows that $g^{-1}d(g)\in \cL$ for every $g\in G$. 
\end{remark}     

The mapping $\gamma$ defined above is obviously linear. 

\begin{lemma}\label{A gamma inj}
The mapping $\gamma$ in injective.
\end{lemma}  
{\em Proof.}  This statement is proved in \cite{CGP2}, but we recall its proof here. Suppose that $\gamma_d = 0$ for some $d\in \Der(\FF)$. Then $\Tr(g^{-1}d(g)x) = 0$ for every $g\in G$ and every $x\in\cL$. In particular, this holds for $g = u_{i, j}(t)$ as defined in Remark \ref{g**(-1)d(g)}. With $g = u_{i,j}(t)$ we have $g^{-1}d(g) = d(t)e_{i,j}$. So, with $x = e_{j,i}$ we get $\Tr(g^{-1}d(g)x) = d(t)$. As $\Tr(g^{-1}d(g)x) = 0$ by assumption, we obtain that $d(t) = 0$. This  holds for every $t\in \FF$. Hence $d = 0$.  \hfill $\Box$ 

\begin{lemma}\label{A gamma1234}
The mapping $\gamma$ satisfies all properties $(\gamma1), (\gamma2), (\gamma3)$ and $(\gamma4)$ of Section \ref{GenConstr}.  
\end{lemma} 
{\em Proof.} Property $(\gamma1)$ is the most difficult to prove. We refer to \cite[Appendix C]{CGP2} for its proof. Property $(\gamma2)$ immediately follows from the fact that, given a derivation basis $\Omega$ of $\FF$, for every $t\in \FF$ we have $d_\omega(t) = 0$ for all but possibly finitely many choices of $\omega\in\Omega$. Lemma 3.10 of \cite{CGP2} states that $\gamma$ satisfies both $(\gamma3)$ and $(\gamma4)$. In order to prove this claim we only need to choose three points $p_0, p_1, p_2 \in \Gamma$ such that the pairs $\{p_0,p_1\}$ and $\{p_0,p_2\}$ span lines $\ell_1$ and $\ell_2$ of $\Gamma$ in distinct orbits of $G$ and, if $g\in G$ fixes $p_0$ or stablizes $\ell_1$ or $\ell_2$ and $m_0\in \cL$ represents $e(p_0)$, then $\Tr(gd(g^{-1})m_0) = 0$ for every $d\in \Der(\FF)$. Of course, we shall choose $p_0, p_1, p_2$ in such a way that the matrices $m_0, m_1, m_2$ which represent their $e$-images in $\PG(\cL)$ admit as many null entries as possible. The matrices $m_0 = e_{n+1,1}$, $m_1 = e_{n+1, 2}$ and $m_2 = e_{n,1}$ indeed do the job.  \hfill $\Box$ 

\subsection{The main result of \cite{CGP2}}

As above, $\cL = \slnK$ and $G = \SlnK$ acting on $\cL$ by conjugation. 

\begin{lemma}\label{rigid A}
The $G$-module $\cL$ is rigid.
\end{lemma}
{\em Proof.} This lemma is Claim 2.11 of  \cite{CGP2}. We recall its proof. Let $\alpha:\cL\rightarrow\cL$ be an invertible linear mapping such that $g\circ\alpha = \alpha\circ g$ for every $g\in G$. Suppose firstly that $|\FF| > 2$. Then, with $g = u_{k,h}(t)$ as in Remark \ref{g**(-1)d(g)}, the previous condition implies that $\alpha([e_{k,h}, x]) =[e_{k,h}, \alpha(x)]$ for every $x\in \cL$. From the latter condition, with $x = e_{i,j}$ and letting the pair $(k,h)$ range in $\{1, 2,..., n+1\}\times\{1,2,.., n+1\}\setminus\{(i,i)\}_{i=1}^{n+1}$, by routine computations we obtain that there a scalar $t_\alpha\in\FF$ such that $\alpha(e_{i,j}) = t_\alpha e_{i,j}$ for every choice of $i\neq j$ in $\{1, 2,..., n+1\}$. Therefore, with $h_i = [e_{i+1,i},e_{i,i+1}] = -e_{i,i}+e_{i+1,i+1}$, we have $\alpha(h_i) =  [e_{i+1,i},\alpha(e_{i,i+1})] =  [e_{i+1,i},t_\alpha e_{i,i+1}] = t_\alpha h_i$. As $\{e_{i\neq j}\}_{i\neq j}\cup\{h_i\}_{=1}^{n+1}$ is a basis of $\cL$, we obtain that $\alpha(x) = t_\alpha x$ for every $x\in \cL$, as claimed. When $\FF = \FF_2$ the same came conclusion can be obtained by a similar method, but the details of the computations are a bit more laborious.  \hfill $\Box$   \\

The next lemma is Lemma 3.1 of \cite{CGP2}. We also include its proof as given in \cite{CGP2}, but for some improvements aimed to clarify a few steps of it.  Indeed, as we will explain below (Remark \ref{central rem}), that proof works as well if we replace $\SlnK$ with $\Omega(2n+1,\FF)$ or $\Omega^+(2n,\FF)$ and $\AnK$ with $B_{n,2}(\FF)$ or $D_{n,2}(\FF)$. 

\begin{lemma}\label{central A} 
The group $G$ and the adjoint embedding $e$ of $\Gamma = \AnK$ satisfy hypothesis $(G2)$ of Theorem \ref{theo nuovo}.
\end{lemma}
{\em Proof.} Let $\hee:\Gamma\rightarrow \widehat{\cL}$ be a cover of the adjoint embedding $e:\Gamma\rightarrow\cL$ of $\Gamma = \AnK$ and let $W$ the kernel of the projection of $\hee$ onto $e$. 

Throughout this proof we adopt the following notation. Given a subfield $\KK$ of $\FF$ and a set $X$ of vectors of an $\FF$-vector space $V$, we denote by $\langle X\rangle_\KK$ the $\KK$-span of $X$ in $V$, namely the set of all $\KK$-linear combinations of a finite
  number of vectors of $X$. The dimension of $\langle X\rangle_\KK$ is its dimension as a $\KK$-vector space. Chosen a basis $B$ of $V$,  we put $V(\KK) := \langle B\rangle_\KK$ and denote by $P(V(\KK))$ the projective geometry of $V(\KK)$, regarded as a subgeometry of $\PG(V)$, every $1$-dimensional subspace of $\langle B\rangle_\KK$ being identified with the $1$-dimensional subspace of $V$ containing it.  

With $V = V(n+1, \FF)$, we denote by $\Gamma(\KK)$ the point-hyperplane geometry of $\PG(V(\KK)) \cong \PG(n,\KK)$, regarded as a subgeometry of $\Gamma = A_{n,\{1,n\}}(\FF)$ and, given a Chevalley basis $C$ of $\cL$ (for instance, $C = \{e_{i,j}\}_{i\neq j}\cup\{h_i\}_{i=1}^n$ where $h_i = [e_{i+1,i}, e_{i,i+1}]$), we put $\cL(\KK) := \langle C\rangle_\KK$. So, $\cL(\KK)$ is the adjoint module for $G(\KK) := \mathrm{SL}(n+1, \KK)$. We also denote by $\widehat{\cL}[\KK]$ the $\FF$-span in $\widehat{\cL}$ of the set of vectors which represent the points of the  image $\hee(\Gamma(\KK))$ of $\Gamma(\KK)$ in $\PG(\widehat{\cL})$ via the embedding $\hee$ and we put  $W[\KK]:=\widehat{\cL}[\KK]\cap W$. 

Clearly, if $\KK\subseteq \KK'$ then $\cL(\KK')$ ($= \langle \cL(\KK)\rangle_{\KK'}$) is a subalgebra of $\cL(\KK')$ and $G(\KK)$ is a subgroup of $G(\KK')$. Moreover $\widehat{\cL}[\KK]\subseteq \widehat{\cL}[\KK']$ and $W[\KK]\subseteq W[\KK']$. Note also that, for every subfield $\KK$ of $\FF$, the projection of $\widehat{\cL}$ onto $\cL$ along $W$ maps $\widehat{\cL}[\KK]$ onto $\cL$. Indeed the image of $\widehat{\cL}[\KK]$ by this projection contains $\cL(\KK)$ and $\langle \cL(\KK)\rangle_\FF = \cL$.   

Let $\FF_0$ be the minimal subfield of $\FF$. Put $G_0:= G(\FF_0)$ for short. By Cooperstein \cite{Coop1} the geometry $\Gamma(\FF_0)$ admits a generating set X of size $n^2+2n = \dim(A)$. (Cooperstein mentions only finite prime fields in \cite{Coop1}, but the results he obtains in that paper hold as well when $\FF_0$ is the field of rational numbers.) Therefore, if we choose a representative vector for every point of $\hee(X)$ we obtain a set $B_X$ of generators of $\widehat{\cL}(\FF_0)$. However $\dim(\widehat{\cL}[\FF_0]) \geq \dim(\cL) = n^2+2n$. Accordingly, $B_X$ is a basis of $\widehat{\cL}[\FF_0]$ and $\dim(\widehat{\cL}[\FF_0]) = \dim(\cL) =  n^2+2n$. As the projection of $\widehat{\cL}$ onto $\cL$ along $W$ maps $\widehat{\cL}[\FF_0]$ onto $\cL$, we obtain that $\widehat{\cL}[\FF_0]$ is a complement of $W$ in $\widehat{\cL}$; with no loss, we can also assume that $\widehat{\cL}[\FF_0] = \cL$. We have proved that $W[\FF_0] = \{0\}$. 

For every subfield $\KK$ of $\FF$ the group $G_0$, being a subgroup of $G(\KK)$, stabilizes $W[\KK]$. Let $\Phi$ be the set of all subfields $\KK$ of $\FF$ such that the action of $G_0$ on the module $W[\KK]$ is trivial. The action of $G_0$ on $W[\FF_0]=\{0\}$ is obviously trivial. Hence $\Phi\neq\emptyset$. Let now $(\{\KK_i\}_{i\in I}, \subseteq)$ be a chain of subfields of $\FF$. For any $x\in W[\bigcup_{i\in I}\KK_i]$ there exists $j\in I$ such that $x\in W[\KK_j]$. It follows that $W[\bigcup_{i\in I}\KK_i] = \bigcup_{i\in I} W[\KK_i]$. Therefore, if $G_0$ acts trivially on $W[\KK_i]$ for every $i\in I$ then it also acts trivially on $W[\bigcup_{i\in I}\KK_i]$. Accordingly, every chain in $\Phi$ admits an upper bound in $\Phi$.  By Zorn's lemma $\Phi$ admits a maximal element, say $\FF_1$.

Proving the lemma amounts to prove that $\FF_1=\FF$. Suppose by contradiction that $\FF_1\subsetneq \FF$. Pick an element $c\in\FF\setminus\FF_1$ and let $\FF_2:=\FF_1(c)$ be the simple extension of $\FF_1$ by means of $c$. In view of Blok and Pasini \cite[Corollary 4.8]{BP1} the geometry $\Gamma(\FF_2)$ is generated by $\Gamma(\FF_1)$ and possibly one extra point of $\Gamma(\FF_2)\setminus\Gamma(\FF_1)$. Therefore $\widehat{\cL}[\FF_1]$ has codimension at most $1$ in $\widehat{\cL}[\FF _2]$. However $\widehat{\cL}[\FF_2]$ and $\widehat{\cL}[\FF_1]$, as $G_0$-modules, are extensions of $\cL$ by $W[\FF_2]$ and $W[\FF_1]$ respectively.  It follows that $W[\FF_1]$ has codimension at most $1$ in $W[\FF_2]$. The group $G_0$ acts trivially on $W[\FF_1]$ by the choice of $\FF_1$. Therefore $W[\FF_1]\subsetneq W[\FF_2]$, since $\FF_1$ is maximal in $\Phi$. So, $W[\FF_1]$ has codimension $1$ in $W[\FF_2]$ and, chosen a vector $v\in W[\FF_2]\setminus W[\FF_1]$, the $G_0$-module $W[\FF_2]$ is an extension of the $1$-space $\langle v\rangle_\FF$ by the $G_0$-module $W[\FF_1]$, which is trivial by the choice of $\FF_1$. However $G_0$ is perfect. Hence it also acts trivially on the $1$-space $\langle v\rangle_\FF$. Therefore the action of $G_0$ on $W[\FF_2]$ can only be as follows: 
   \begin{equation} \label{trac}
 (w, tv)\cdot g ~=~  (t\phi(g)+w, t\lambda(g)v), \qquad  g\in G_0,~ t\in \FF , ~ w\in W[\FF_1]. 
 \end{equation}
where $\phi$ and $\lambda$ are mappings from $G_0$ to $W[\FF_1]$ and the multiplicative group of $\FF$ respectively. 
Clearly, $\lambda$ is a homomorhism. However $G_0$ is perfect. Therefore $\lambda(g) = 1$ for every $g\in G_0$ and (\ref{trac}) boils down to the followng:
   \begin{equation} \label{trac bis}
 (w, tv)\cdot g ~=~  (t\phi(g)+w, t)v).
 \end{equation} 
By (\ref{trac bis}) we get $(w, tv)\cdot g_1g_2 =  (t\phi(g_1g_2) + w, tv)$ as well as
\[(w, tv)\cdot g_1g_2 ~ = ~ (t\phi(g_1)+w, tv)\cdot g_2 ~ = ~ (t\phi(g_2) + t\phi(g_1) + w, tv). \]
Therefore $\phi(g_1g_2) = \phi(g_1)+\phi(g_2)$, namely $\phi$ is a homomorphism from $G_0$ to the additive group of $W[\FF_1]$. Since $G_0$ is perfect, $\phi(g) = 0$ for every $g\in G_0$ and (\ref{trac bis}) now shows that $G_0$ acts trivially on $W[\FF_2]$. Consequently, $\FF_2\in \Phi$. This conclusion contradicts the maximality of $\FF_1$ in $\Phi$. 
 
Therefore $\FF =\FF_1\in\Phi$ and the action of $G_0$ on $W = W[\FF]$ is trivial. To conclude, observe that the group
$G:=\mathrm{SL}(n+1,\KK)$ does not admit any proper normal subgroup containing $G_0$. Indeed $G/Z(G)$ is simple.  Hence the action of $G$ on $W$ must also be trivial. \hfill $\Box$ 

\begin{remark}\label{central rem}
\em
The previous proof relies on the fact that $G_0$ is perfect and $G/Z(G)$ is simple (as it is the case also when $G =\Omega(2n+1,\FF)$ or $G = \Omega^+(2n,\FF)$) and the following two results: $\Gamma(\FF_0)$ can be generated by a set of points of size equal to $\dim(\cL)$ (Cooperstein \cite{Coop1}) and, if $\FF_2$ is a simple extension of $\FF_1$, then $\Gamma(\FF_2)$ can be generated by $\Gamma(\FF_1)\cup\{p\}$ for a suitable point $p\in \Gamma(\FF_2)$ (Blok and Pasini \cite{BP2}). Both these results also hold for $\Gamma = B_{n,2}(\FF)$ and $\Gamma = D_{n,2}(\FF)$ (and $\cL$ the appropriate Lie algebra, of course).   
\end{remark} 

By Lemmas \ref{A gamma1234}, \ref{rigid A} and \ref{central A} and Theorem \ref{theo nuovo}, with $\gamma$ as in the previous subsection and a point $p_0$ of $\Gamma$ satisfying property $(\gamma3)$ (say $p_0$ such that $e(p_0)$ is represented by the matrix $e_{n+1,1}$), the mapping $e_{\gamma, p_0}:\Gamma\rightarrow \cL_\Omega = \Der_\Omega(\FF)\oplus\cL$ defined according to formula (\ref{Intro2}) is the universal cover of $e$. This is indeed the main result of \cite{CGP2}. 

This result can be given a more explicit formulation. Recall that, according to the way $\gamma$ has been defined, for  $d\in \Der(\FF)$, $g\in G$ ad $x\in \cL$, we have $\langle \gamma_d(g^{-1}), x\rangle =  \Tr(gd(g^{-1})x)$. In particular, with $x = e_{n+1,1} = e^*_{n+1}e_1$ we have $\langle \gamma_d(g^{-1}), e_{n+1,1}\rangle = \Tr(gd(g^{-1})e^*_{n+1}e_1) = \Tr(d(g^{-1})e^*_{n+1}e_1g)$. Put $f = g^{-1}e^*_{n+1}$. Then $d(f) = d(g^{-1})e^*_{n+1} + g^{-1}d(e^*_{n+1}) = d(g^{-1})e^*_{n+1}$ (because $d(e^*_{n+1}) = 0$, since the unique non-zero entry of $e^*_{n+1}$ is equal to $1$ and $d(1) = 0$). Therefore, with $ v = e_1g$, we have $\langle \gamma_d(g^{-1}), e_{n+1,1}\rangle = \Tr(d(f)v) = vd(f)$.  So, we can state the above result as follows (see \cite[Theorem 1.2]{CGP2}). 

\begin{theorem}\label{mainA}
If the point $\{a,A\}$ of $\Gamma$ is mapped by $e$ onto the point of $\PG(\cL)$ represented by $fv = g^{-1}e_{n+1,1}g$ (for some $g\in G$, with $f = g^{-1}e^*_{n+1}$ and $v = e_1g$ representing $A$ and $a$ respectively), then the universal cover of $e$ maps $\{a,A\}$ onto the point of $\PG(\cL_\Omega)$ represented by the following vector of $\cL_\Omega = \Der_\Omega(\FF)\oplus \cL$:
\[(\sum_{\omega\in\Omega}vd_\omega(f)d_\omega) \oplus fv\]
\end{theorem}    

\section{The absolutely universal embedding of the geometries $B_{n,2}(\FF)$ and $D_{n,2}(\FF)$}\label{univ BD} 

\subsection{The adjoint embedding of $B_{n,2}(\FF)$, $n > 2$}\label{adj B sec} 

Throughout this subsection $G$ is the derived subgroup $\Omega(2n+1,\FF)$ of the orthogonal group $\mathrm{O}(2n+1,\FF)$ and $\cL = \onK$ is its Lie algebra. We assume that $n > 2$ and suppose, as we can, that $G$, in its natural action on $V := V(2n+1,\FF)$, preserves the quadratic form 
\begin{equation}\label{form B}
q(x_1,..., x_{2n+1}) ~ = ~  \sum_{i=1}^nx_ix_{n+i} + x_{2n+1}^2.
\end{equation}
Moreover, we assume that $G$ acts at right on $V$, namely the vectors of $V$ are regarded as rows. With this convention,  $G$ acts on $\cM :=  M_{2n+1}(\FF) = V^*\otimes V$ at right by congiugation ($g\in G$ maps $m\in\cM$ onto $g^{-1}mg$) and stabilizes the subspace $\cM_\cL$ of $\cM$ formed by the matrices of the following form:
\begin{equation}\label{matrix B}
\left(\begin{array}{rrr}
X & Y & u^t \\
Z & -X^t & v^t \\
-2v & -2u & 0 
\end{array}\right)
\end{equation} 
where the exponent $t$ stands for transposition, $X, Y$ and $Z$ are square matrices of order $n$, the matrices $Y$ an $Z$ are anti-symmetric and $u$ and $v$ are vectors of $V(n,\FF)$. These are precisely the matrices $m \in \cM$ such that the row-times-column product $mJ$ is anti-symmetric, where 
\begin{equation}\label{J B}
J ~ = ~ \left(\begin{array}{ccc}
O_n & I_n & {\bf 0}^t \\
I_n & O_n & {\bf 0}^t \\
{\bf 0} & {\bf 0} & 2 
\end{array}\right) 
\end{equation} 
is the representative matrix of the bilinearization of the quadratic form described in (\ref{form B}). (Here $O_n$ and $I_n$ are the null and identity matrix of $M_n(\FF)$ and $\bf 0$ is the null vector of $V(n,\FF)$.) Regarded $\cM$ as a Lie algebra with bracket $[x,y] = xy-yx$, the subspace $\cM_\cL$ is a subalgebra of $\cM$. When $\ch(\FF) \neq 2$ the Lie algebra $\cM_\cL$ is an isomorphic copy of $\cL$. We shall discuss this isomorphism quite in detail. This will be helpful later, when we will describe the characteristic $2$ case. 

\subsubsection{The case where $\ch(\FF) \neq 2$} 

Let $\ch(\FF) \neq 2$. In order to show that $\cM_\cL\cong \cL$ we just need to show that $\cM_\cL$ admits a Chevalley basis of type $B_n$. 
First of all, we set the notation we shall use for roots. Given a basis $\Pi = \{\alpha_1, \alpha_2,..., \alpha_n\}$ of the root system $\Phi$ of type $B_n$, we put
\begin{equation}\label{Root B} 
\left\{\begin{array}{rl} 
\alpha^+_{(k,h)} := & \sum_{i=k}^{h-1}\alpha_i, ~ 1\leq k < h \leq n; \\
\alpha^+_{[k,n]} := & \alpha^+_{(k,n)} + 2\alpha_n, ~ 1\leq k < n; \\
\alpha^+_{[k,h]} := & \alpha^+_{(k,n)} + \alpha^+_{[h,n]}, ~ 1\leq k < h < n; \\
\alpha^+_{\{k\}} := & \alpha^+_{(k,n)} + \alpha_n, ~ 1\leq k \leq n, ~ (\alpha^+_{(n,n)} = 0 \mbox{ by convention}).
\end{array}\right.
\end{equation}
The above are the positive roots, $\alpha^+_{[1,2]}$ being the highest one. We denote their opposites by $\alpha^-_{(k,h)}, \alpha^-_{[k,n]}, \alpha^-_{[k,h]}$ and $\alpha^-_{\{k\}}$, respectively. Note also that the roots $\alpha^\pm_{(k,h)}$ and $\alpha^\pm_{[k,h]}$ are long while $\alpha^\pm_{\{k\}}$ is short. Clearly, $\alpha^+_{(k,k+1)} = \alpha_k$ and $\alpha^+_{\{n\}} =\alpha_n$. 

As in Section \ref{A}, for $i, j \in \{1, 2,..., 2n+1\}$ we denote by $e_{i,j}$ the matrix of $\cM$ with all null entries but the $(i,j)$-entry, which is equal to $1$; equivalently, if $(e_1, e_2,..., e_{2n+1})$ is the natural basis of $V$, then $e_{i,j} = e_i^te_j$. With this notation we put:
\begin{equation}\label{rootElm B}
\left\{\begin{array}{rcl}
x_{\alpha^+_{(k,h)}} & := & e_{h,k}-e_{k+n,h+n}; \\
x_{\alpha^-_{(k,h)}} &  := &  e_{k,h}-e_{h+n,k+n} \mbox{ (transposed of the above});\\
x_{\alpha^+_{[k,h]}} & : = & e_{h+n,k}-e_{k+n,h}; \\
x_{\alpha^-_{[k,h]}} &  :=  & e_{k,h+n}-e_{h,k+n} \mbox{ (transposed of the above});\\
x_{\alpha^+_{\{k\}}} &  := &  e_{n+k,2n+1}-2e_{2n+1,k}; \\
x_{\alpha^-_{\{k\}}}  & := &  e_{k,2n+1}- 2e_{2n+1,n+k};\\
h_k  & :=  &  -e_{k,k}+e_{k+1,k+1} + e_{n+k,n+k} - e_{n+k+1, n+k+1} = \\
 & & = [x_{\alpha_k}, x_{-\alpha_k}],  \mbox{ (for } k < n); \\ 
h_n & := &  2e_{n,n} - 2e_{2n,2n} = [x_{\alpha_n}, x_{-\alpha_n}].  
\end{array}\right.
\end{equation} 
It is straightforward to check that the above matrices form a Chevalley basis of $\cM_\cL$, where the elements $x_{\alpha^\pm_{(k,h)}}$, $x_{\alpha^\pm_{[k,h]}}$ and $x_{\alpha^\pm_{\{k\}}}$ span the root subspaces and $\{h_1,..., h_{n-1}, h_n\}$ is a basis of the Cartan subalgebra. Therefore $\cM_\cL \cong \cL$, as claimed. Accordingly, we can regard $\cM_\cL$ as the same as $\cL$, thus freely writing $\cL$ when in fact we are referring to $\cM_\cL$.  

The following formulas describe the matrices that represent the root elements of $G$ in their action on the natural module $V$. Of course, in these formulas $I$ is the identity matrix of order $2n+1$ and $t\in \FF$. An abbreviated notation (the 'nicknames') is introduced to replace the complete symbols (here called `full names'), which are a bit too cumbursome.   
 \begin{equation}\label{rootSbgr B}
\left\{\begin{array}{rcl}
\mbox{nickname} & \mbox{full name} & \mbox{representative matrix}\\
\hline\
u_{h,k}(t) & :=  u_{\alpha^+_{(k,h)}}(t)   = & I+tx_{\alpha^+_{(k,h)}}; \\
u_{k,h}(t)  & :=  u_{\alpha^-_{(k,h)}}(t)   = & I+tx_{\alpha^-_{(k,h)}}; \\ 
u^+_{k,h}(t) & : =   u_{\alpha^+_{[k,h]}}(t)   = & I +tx_{\alpha^+_{[k,h]}}; \\
u^-_{k,h}(t) & :=  u_{\alpha^-_{[k,h]}}(t) = &  I+tx_{\alpha^-_{[k,h]}}; \\
u^+_k(t) &  :=  u_{\alpha^+_{\{k\}}}(t)  = & I+tx_{\alpha^+_{\{k\}}} - t^2e_{n+k, k};\\
u^-_k(t) &  :=  u_{\alpha^-_{\{k\}}}(t)  = & I + tx_{\alpha^-_{\{k\}}} - t^2e_{k,n+k}.  
\end{array}\right.
\end{equation} 
It is straightforward to check that for every root $\alpha\in \Phi$ the action of the matrix $u_\alpha(t)$ by conjugation on $\cL$ is the same as if the element $u_\alpha(t) \in G$ was defined by exponentiation:    
\[u_\alpha(t)~ :  ~ x\in \cL ~\rightarrow ~ \exp(t\adj(x_\alpha)(x)) ~=~ \sum_{k=0}^{n(\alpha)-1}\frac{t^k\adj(x_\alpha)^k(x)}{k!}\]
where $n(\alpha)$ is the nilpotent exponent of the (nilpotent) adjoint operator $\adj(x_\alpha) : x\in \cL \rightarrow [x_\alpha, x]$.  

Let $e\colon\BnK\rightarrow \cL$ be the projective embedding that maps a generic totally singular $2$-subspace $\langle x, y\rangle$ of $V$ (i.e. a point of $B_{n,2}(\FF)$) onto the point of $\PG(\cL)$ represented by the matrix
\begin{equation}\label{adj B}
J\cdot(x^ty - y^tx) ~ = ~ (xJ)^ty - (yJ)^tx
\end{equation}
with $J$ as in (\ref{J B}). Clearly, $x^ty-y^tx = O$ (null  matrix) if and only if  $\langle x\rangle = \langle y\rangle$. It is straightforward to check that the vectors defined as in (\ref{adj B}) span $\cL$. Moreover, given a totally $q$-singular $3$-subspace $\langle u, v, w\rangle$ of $V$, the set 
\[\{J\cdot(au^t(vb+wc) - (bv^t+cw^t)ua) ~|~ a,b,c \in \FF\}\]
is a 2-subspace of $\cL$. Hence $e$ is indeed a projective embedding of $\BnK$ in $\cL$. Note also that, since $gJg^t = J$ (namely $Jg^t = g^{-1}J$) for every $g\in G$, for every $g\in G$ and every totally $q$-singular $2$-subspace $\langle x,y\rangle$ of $V$ we have
\[(xgJ)^tyg - (ygJ)^txg ~ = ~ g^{-1}((xJ)^ty - (yJ)^tx)g.\]
Hence $e$ is linearly $G$-homogeneous. Finally, let $(e_1,..., e_{2n+1})$ be the natural basis of $V$. Then $e(\langle e_1, e_2\rangle)$ is represented by 
\[J(e_2^te_1-e_1^te_2) ~= ~ e_{n+2,1}-e_{n+1,2} ~= ~ x_{\alpha^+_{[1,2]}}.\]
It is straightforward to check that the stabilizer of $\langle e_{n+1,2}-e_{n+2,1}\rangle = \langle x_{\alpha^+_{[1,2]}}\rangle$ in $G$ acting on $\cL$ is just the same as the stabilizer of the subspace $\langle e_1, e_2\rangle$ of $V$ in $G$ acting on $V$. Moreover $\langle x_{\alpha^+_{[1,2]}}\rangle$ is the unique $1$-dimensional subspace of $\cL$ stabilized by the stabilizer of $\langle e_1, e_2\rangle$ in $G$. As $\alpha^+_{[1,2]}$ is the highest root of the root system $\Phi$, the matrix $x_{\alpha^+_{[1,2]}}$ is a maximal vector of weight $\lambda_2 = \alpha^+_{[1,2]}$ for the  Weyl module $\cL = W_{\lambda_2}$. The following is now clear.  

\begin{proposition}\label{prop adj B} 
The embedding $e$ as defined in {\rm (\ref{adj B})} is just the adjoint embedding of $\BnK$.  
\end{proposition} 

\begin{remark}\label{plucker B}
\em
When $\ch(\FF) \neq 2$ the adjoint embedding of $B_{n,2}(\FF)$ is isomorphic to the Pl\"{u}cker embedding of $\BnK$ in $V\wedge V$, which maps a totally $q$-singular line $\langle x, y\rangle$ of $\PG(V)$ onto the point of $\PG(V\wedge V)$ represented by the vector $x\wedge y$ with coordinates $(x\wedge y)_{i,j} = x_iy_j-x_jy_i$ for $i< j$, where $x = (x_i)_{i=1}^{2n+1}$ and $y = (y_i)_{i=1}^{2n+1}$. 

Indeed let $m_{x,y} := (xJ)^ty-(yJ)^tx$ (see (\ref{adj B})) and let $m_{i,j}$ the $(i,j)$-entry of $m_{x,y}$. Then
\[\begin{array}{l}
m_{i,j} = x_{i+n}y_j - y_{i+n}x_j = -(x\wedge y)_{j,i+n} ~ \mbox{ for } i, j \leq n;\\
m_{i,j} = x_iy_{j+n}-y_ix_{j+n} = (x\wedge y)_{i,j+n} ~ \mbox{ for } n < i, j \leq 2n;\\
m_{i,j} = x_{i+n}y_j - y_{i+n}x_j = \left\{\begin{array}{l} (x\wedge y)_{i+n,j} ~ \mbox{ for } i \leq n < j \leq 2n, ~ i+n < j; \\
 0 ~ \mbox{ for } i \leq n < j = i+n; \\
-(x\wedge y)_{j, i+n} ~ \mbox{ for } i \leq n < j < i+n;
\end{array}\right.  \\
m_{i,j} = x_{i-n}y_j - y_{i-n}x_j = \left\{\begin{array}{l} (x\wedge y)_{i-n,j} ~ \mbox{ for } i > n \geq j > i-n; \\
 0 ~ \mbox{ for } 2n \geq i > n \geq j = i-n; \\
-(x\wedge y)_{j, i-n} ~ \mbox{ for } 2n \geq i > n, ~ i-n > j;
\end{array}\right.  \\
m_{i,2n+1} = (x_{i+n}y_{2n+1}-y_{i+n}x_{2n+1}) = (x\wedge y)_{i+n,2n+1} ~ \mbox{ for } i \leq n;\\
m_{i,2n+1} = x_{i-n}y_{2n+1}-y_{i-n}x_{2n+1} = (x\wedge y)_{i-n,2n+1} ~ \mbox{ for } n < i \leq 2n;\\
m_{2n+1,j} = 2(x_{2n+1}y_j-y_{2n+1}x_j)= -2(x\wedge y)_{j, 2n+1} ~ \mbox{ for } j\leq 2n; \\
m_{2n+1,2n+1} = 0.
\end{array}\]
Up to some repetitions and multiplications by $-1$ or $2$, the non-zero entries of $m_{x,y}$ are just the coordinates of $x\wedge y$. 
\end{remark} 
 
\subsubsection{The case where $\ch(\FF) = 2$}\label{L0 B}

Let now $\ch(\FF) = 2$. In this case the last row of the matrices of $\cM_\cL$ (see (\ref{matrix B})) is null as well as the last row and last column of $J$ (see (\ref{J B})). The matrix $J$ indeed represents the (degenerate) alternating form associated to the quadratic form $q$. But this is not so important here. The important issue is that $\cM_\cL$ is not isomorphic to the Lie algebra $\cL$ of $G$. 
Indeed the $G$-orbit in $\cM_\cL$ of the matrix $e_{n+1,2}-e_{n+2,1} = J\cdot (e_1^te_2 - e_2^te_1)$ spans a (necessarily $G$-invariant) hyperplane $\cL_0$ of $\cM_\cL$, formed by the matrices as in (\ref{matrix B}) with $\Tr(X) = 0$. In contrast, $\cL$ admits no $G$-invariant hyperplane. 

Formula (\ref{adj B}) still defines a linearly $G$-homogeneous projective embedding $e_0:\BnK\rightarrow\cL_0$ but, as we can see from Remark \ref{plucker B}, this is the same as the Pl\"{u}cker embedding of $\BnK$ in $\cL_0 \cong V\wedge V$, which is now a proper projection of the adjoint embedding $e:\BnK\rightarrow \cL$. The kernel of the projection of $e$ onto $e_0$ is the center $Z(\cL)$ of $\cL$. Indeed, while when $\ch(\FF)\neq 2$ the algebra $\cL$ is simple, when $\ch(\FF) = 2$ the Lie algebra $\cL$  admits a $1$-dimensional center wich, for every Chevalley basis $C = \{x_\alpha\}_{\alpha\in\Phi}\cup\{h_i\}_{i=1}^n$ of $\cL$, is equal to the $1$-dimensional subspace spanned by $h_n$.

We can regard $\cL_0$ as a complement of $Z(\cL)$ in the vector space $\cL$. In this setting, we can describe the projection $\pi_0$ of $\cL$ onto $\cL_0$ along $Z(\cL)$ by rewriting table (\ref{rootElm B}) as follows: 

\begin{equation}\label{rootElm B char 2}
\pi_0 : \left\{ \begin{array}{rcl}
x_{\alpha^+_{(k,h)}} & \rightarrow & x_{\alpha^+_{(k,h)}} = e_{h,k}-e_{k+n,h+n}; \\
x_{\alpha^-_{(k,h)}} & \rightarrow &  x_{\alpha^-_{(k,h)}} = e_{k,h}-e_{h+n,k+n};\\
x_{\alpha^+_{[k,h]}} & \rightarrow & x_{\alpha^+_{[k,h]}} = e_{h+n,k}-e_{k+n,h}; \\
x_{\alpha^-_{[k,h]}} & \rightarrow  & x_{\alpha^-_{[k,h]}} = e_{k,h+n}-e_{h,k+n};\\
x_{\alpha^+_{\{k\}}} & \rightarrow & x_{\alpha^+_{\{k\}}} = e_{n+k,2n+1}-2e_{2n+1,k} = e_{n+k, 2n+1};\\
x_{\alpha^-_{\{k\}}}  & \rightarrow & x_{\alpha^-_{\{k\}}} =  e_{k,2n+1}- 2e_{2n+1,n+k} = e_{k,2n+1}; \\
h_k  & \rightarrow  & -e_{k,k}+e_{k+1,k+1} + e_{n+k,n+k} - e_{n+k+1, n+k+1} = \\
 & & = ~ h_k, \mbox{ (for } k = 1, 2,..., n-1);\\ 
h_n & \rightarrow &  2e_{n,n} - 2e_{2n,2n} = O ~ (\mbox{the null matrix }).   
\end{array}\right.
\end{equation} 

\begin{remark}
\em
In general $[x_{\alpha^+_{\{k\}}}, x_{\alpha^-_{\{k\}}}] = h_n - 2\sum_{i=k}^{n-1}h_i$ for every $k = 1, 2,..., n$. When $\ch(\FF) = 2$, as we assuming now, the above equality boils down to $[x_{\alpha^+_{\{k\}}}, x_{\alpha^-_{\{k\}}}] = h_n$ for every $k = 1, 2,..., n$.
\end{remark}

Thus, $\cL_0$ is the hyperplane of $\cL$ spanned by the subset 
\[C_0 ~ := ~ \{x_\alpha\}_{\alpha\in\Phi}\cup\{h_i\}_{i=1}^{n-1}\]
of $C$ (but it is neither a subalgebra nor a $G$-submodule of $\cL$, although it is itself both a Lie algebra and a $G$-module). Table (\ref{rootSbgr B}) still describes the matrices that yield the root elements of $G$ in its right action on $V$. Conjugation by these matrices yields the right action on $\cL_0$ of the root elements of $G$.  

The projection $\pi_0:\cL\rightarrow\cL_0$ maps the adjont embedding $e:\BnK\rightarrow \cL$ onto the Pl\"{u}cker embedding $e_0:\BnK\rightarrow\cL_0$. Clearly $G$ stabilizes $Z(\cL)$ in its action on $\cL$. Therefore $\pi_0$ maps the action of $G$ on $\cL$ onto the action of $G$ on $\cL_0$. Moreover $G$, being perfect, centralizes $Z(\cL)$, which is $1$-dimensional. So, as a $G$-module, $\cL$ is a non-split central extension $\cL = K\cdot\cL_0$ of the $G$-module $\cL_0$ by the $1$-dimensional trivial $G$-module $K := Z(\cL) = \langle h_n\rangle$. In view of Theorem \ref{thm link} this extension can be described by means of a cocycle $F_0\in Z^1(\cL_0^*,G)$ of $G$ over the dual $\cL^*_0$ of $\cL$. Explicitly, $F_0$ is not a coboundary and $G$ acts as follows on $\cL = K\cdot \cL_0$: 
\begin{equation}\label{F0 B 0}
g ~:~ th_n\oplus x \in K\oplus\cL_0  ~ \longrightarrow ~ (t+\langle F_0(g^{-1}), x\rangle)h_n\oplus g^{-1}xg 
\end{equation}
for every $g\in G$, $t\in \FF$ and $x\in \cL_0$, where $g^{-1}xg$ is the $g$-image of $x$ in the action of $G$ on $\cL_0$ and the symbol $\langle . , . \rangle$ means the same as in Section \ref{GenConstr}. Moreover, if $\pi_K:\cL\rightarrow K = Z(\cL)$ is the projection of $\cL$ onto $K$ along $\cL_0$ then 
\begin{equation}\label{F0 B 1} 
\langle F_0(g^{-1}), x\rangle h_n ~ = ~ \pi_K(x\cdot g), \hspace{5 mm} \forall g\in G, ~\forall x\in \cL_0.  
\end{equation}
 for every $g\in G$ and $x\in \cL_0$, where $x\cdot g$ is the $g$-image of $x\in \cL_0$ in the $G$-module $\cL$.   

\begin{definition}\label{def F0} 
{\rm We call $F_0$ the} seed cocycle {\rm of $G$ over $\cL_0^*$.}
\end{definition}
 
Formula (\ref{F0 B 1}) explains what $\langle F_0(g^{-1}), x\rangle$ is but not how to compute it. Indeed (\ref{F0 B 1}) defines $F_0$ by means of the action of $g$ on $\cL$ while in (\ref{F0 B 0}) this action is defined via $F_0$. The next lemma offers us a way out from this vicious circle. 

Recall that $G$ is generated by the set $U = \{u_\alpha(t)\}_{\alpha \in \Phi, ~t\in \FF}$ of root elements described in table (\ref{rootSbgr B}). Therefore, by the cocycle condition on $F_0$ and linearity, if we know $\langle F_0(g^{-1}), x\rangle$ for every $g\in U$ and every $x\in C_0$ then, at least in principle, we also know $\langle F_0(g^{-1}), x\rangle$ for every $g\in G$ and every $x\in\cL_0$. 

Let $\Phi_0 := \cup_{1 \leq k < h \leq n+1}\{\alpha_{k,h}^+, \alpha_{k,h}^-\}$ be the subsystem of $\Phi$ of type $A_{n-1}$ generated by the roots $\alpha_1, \alpha_2,..., \alpha_{n-1}$.

\begin{lemma}\label{lemma F0 B}
With $\Phi_0$ as above, the following hold for every $t\in\FF$.  
\begin{itemize}
\item[$(1)$] If $\alpha\in \Phi_0$ then $\langle F_0(u_\alpha(t)), x\rangle = 0$ for every $x\in \cL_0$;
\item[$(2)$] Let $\alpha\in\Phi\setminus\Phi_0$. Then $\langle F_0(u_\alpha(t)), x\rangle = 0$ for every $x\in C_0\setminus\{x_{-\alpha}\}$ while $\langle F_0(u_\alpha(t)), x_{-\alpha} \rangle = t$. 
\end{itemize}
\end{lemma}
{\em Proof.}  We switch from $\FF$ to the ring $\cZ = \ZZ[t_1,t_2,..., t_n,...]$ of polynomials over the ring $\ZZ$ of integers. Let $M(\cZ)$ be the $\cZ$-module of all square matrices of order $2n+1$ with entries in $\cZ$, equipped with the usual row-times-column product and the natural bracket product. By tables (\ref{rootElm B}) and (\ref{rootSbgr B}) all elements of $C$ as well as those of $U$ (but with $t$ regarded as an unknown) belong to $M(\cZ)$. Moreover, if $g = u_\alpha(t)\in U$ then $g^{-1} = u_\alpha(-t)$. So, $g^{-1}\in M(\cZ)$ for every $g\in U$ and $g^{-1}zg\in M(\cZ)$ for every $x\in C$ and every $g\in U$.  

Let $L(\cZ)$ be the substructure of $M(\cZ)$ generated by $C$. Then $L(\cZ) = K(\cZ)\oplus L_0(\cZ)$ where $L(\cZ)$ and $K(\cZ)$ are the $\cZ$-spans of $C_0$ and respectively $h_n = 2(e_{n,n}-e_{2n,2n})$. Given $x\in C_0$ and $g = u_\alpha(t)\in U$, the product $y =  g^{-1}x g$ belongs to $L(\cZ)$. Hence we can decompose $y$ in a sum $y = y_K+y_0$ with $y_K \in K(\cZ)$ and $y_0\in L_0(\cZ)$. So, $y_K = f_{\alpha,x}(t)h_n$ for a polynomial $f_{\alpha,x}(t)\in \cZ$. Let $f_{\alpha,x}(t)_{|2} \in \FF_2[t]$ be the polynomial obtained from $f_{\alpha,x}(t)$ by computing its coefficients modulo 2. Then $f_{\alpha,x}(t)_{|2}$ also belongs to $\FF[t]$, since $\FF_2$ is the minimum subfield of $\FF$. Regarded now $t$ as a generic element of $\FF$, we have $f_{\alpha,x}(t)_{|2} = \langle F_0(u_\alpha(t)^{-1}), x\rangle$.   

The results mentioned in the lemma can be obtained by the above method. The (straightforward) computations are left to the reader.  \hfill $\Box$  \\

The Pl\"{u}cker embedding $e_0:\BnK\rightarrow \cL_0$ maps the point $\langle e_1, e_2\rangle$ of $\BnK$ onto the point of $\PG(\cL_0)$ represented by $x_{\alpha^+_{[1,2]}} = e_{n+2,1}-e_{n+1,2} = J\cdot(e_2^te_1-e_1^te_2)$. In view of the above and the flag-transitivity of $G$ on $\BnK$, the adjoint embedding $e:\BnK\rightarrow\cL$ can be described as follows

\begin{proposition}\label{prop adj B car 2} 
For every $g\in G$, the adjoint embedding maps the point $\langle e_1, e_2\rangle\cdot g$ of $\BnK$ onto he point of $\PG(\cL)$ represented by the vector
\[\langle F_0(g^{-1}), x_{\alpha^+_{[1,2]}}\rangle h_n\oplus g^{-1}x_{\alpha^+_{[1,2]}}g.\]
\end{proposition}

\subsection{The adjoint embedding of $D_{n,2}(\FF)$}\label{adj D sec}   

All we have said for $\BnK$ in Subsection \ref{adj B sec} holds with obvious minor changes for $\DnK$ as well. Explicitly, let $G = \Omega^+(2n,\FF)$ (the derived subgroup of the orthogonal group $\mathrm{O}^+(2n,\FF)$) and let $\cL = \ohK$ be its Lie algebra. According to a well established convention, $n \geq 4$. Suppose as we can that $G$, in its natural action on $V := V(2n,\FF)$, preserves the form
\begin{equation}\label{form D}
q(x_1,..., x_{2n+1}) ~ = ~  \sum_{i=1}^nx_ix_{n+i}.
\end{equation}
As in Subsection \ref{adj B sec}, we assume that $G$ acts at right on $V$, hence at right by conjugation on $\cM :=  M_{2n}(\FF) = V^*\otimes V$, thus stabilizing the subspace $\cM_\cL$ of $\cM$ formed by the matrices of the following form:
\begin{equation}\label{matrix D}
\left(\begin{array}{rr}
X & Y \\
Z & -X^t  
\end{array}\right)
\end{equation} 
where $X, Y$ and $Z$ are square matrices of order $n$, with $Y$ and $Z$ anti-symmetric. These are precisely the matrices $m \in \cM$ such that $mJ$ is anti-symmetric,  where 
\begin{equation}\label{J D}
J ~ = ~ \left(\begin{array}{cc}
O_n & I_n  \\
I_n & O_n
\end{array}\right) 
\end{equation} 
is the representative matrix of the bilinearization of  (\ref{form D}). Regarded $\cM$ as a Lie algebra with bracket $[x,y] = xy-yx$, the subspace $\cM_\cL$ is a subalgebra of $\cM$. 

\subsubsection{The case where $\ch(\FF) \neq 2$} 

Let $\ch(\FF) \neq 2$. Then $\cM_\cL\cong \cL$. As in Subsection \ref{adj B sec}, we can prove this claim by showing that $\cM_\cL$ admits a Chevalley basis of the right type, namely $D_n$. As in Subsection \ref{adj B sec}, given a basis $\Pi = \{\alpha_1, \alpha_2,..., \alpha_n\}$ of the root system $\Phi$ of type $D_n$, we put
\begin{equation}\label{Root D} 
\left\{\begin{array}{rl} 
\alpha^+_{(k,h)} := & \sum_{i=k}^{h-1}\alpha_i, ~ 1\leq k < h \leq n; \\
\alpha^+_{[n-1,n]} := & \alpha_n;\\
\alpha^+_{[k,n]} := & \alpha^+_{(k,n-1)} + \alpha_n, ~ 1\leq k < n-1;\\
\alpha^+_{[k,h]} := & \alpha^+_{(k,n)} + \alpha^+_{[h,n]}, ~ 1\leq k < h < n. 
\end{array}\right.
\end{equation}
These are the positive roots, $\alpha^+_{[1,2]}$ being the highest one. We denote their opposite by $\alpha^-_{(k,h)}$ and $\alpha^+_{[k,h]}$. As before, $e_{i,j}$ is the matrix of $\cM$ with all null entries but the $(i,j)$-entry, which is equal to $1$. With this notation we put:
\begin{equation}\label{rootElm D}
\left\{\begin{array}{rcl}
x_{\alpha^+_{(k,h)}} & := & e_{h,k}-e_{k+n,h+n}; \\
x_{\alpha^-_{(k,h)}} &  := &  e_{k,h}-e_{h+n,k+n} \mbox{ (transposed of the above});\\
x_{\alpha^+_{[k,h]}} & : = & e_{h+n,k}-e_{k+n,h}; \\
x_{\alpha^-_{[k,h]}} &  :=  & e_{k,h+n}-e_{h,k+n} \mbox{ (transposed of the above});\\
h_k  & :=  &  -e_{k,k}+e_{k+1,k+1} + e_{n+k,n+k} - e_{n+k+1, n+k+1} = \\
 & & = [x_{\alpha_k}, x_{-\alpha_k}],  \mbox{ (for } k < n); \\ 
h_n & := & -e_{n-1,n-1}-e_{n,n} + e_{2n-1, 2n-1}+ e_{2n,2n} = \\
 & & =  [x_{\alpha_n}, x_{-\alpha_n}].  
\end{array}\right.
\end{equation} 
These matrices form a Chevalley basis of $\cM_\cL$, where the elements $x_{\alpha^\pm_{(k,h)}}$ and $x_{\alpha^\pm_{[k,h]}}$ span the root subspaces and $\{h_1,..., h_{n-1}, h_n\}$ is a basis of the Cartan subalgebra. Therefore $\cM_\cL =  \cL$, as claimed. 

The following formulas describe the matrices that represent the root elements of $G$ in their action on the natural module $V$.   
 \begin{equation}\label{rootSbgr D}
\left\{\begin{array}{rcl}
\mbox{nickname} & \mbox{full name} & \mbox{representative matrix}\\
\hline\
u_{h,k}(t) & :=  u_{\alpha^+_{(k,h)}}(t)   = & I+tx_{\alpha^+_{(k,h)}}; \\
u_{k,h}(t)  & :=  u_{\alpha^-_{(k,h)}}(t)   = & I+tx_{\alpha^-_{(k,h)}}; \\ 
u^+_{k,h}(t) & : =   u_{\alpha^+_{[k,h]}}(t)   = & I +tx_{\alpha^+_{[k,h]}}; \\
u^-_{k,h}(t) & :=  u_{\alpha^-_{[k,h]}}(t) = &  I+tx_{\alpha^-_{[k,h]}}.  
\end{array}\right.
\end{equation} 
Let $e\colon\DnK\rightarrow \cL$ be the projective embedding that maps a generic totally singular $2$-subspace $\langle x, y\rangle$ of $V$ onto the point of $\PG(\cL)$ represented by the matrix
\begin{equation}\label{adj D}
J\cdot(x^ty - y^tx) ~ = ~ (xJ)^ty - (yJ)^tx
\end{equation}
with $J$ as in (\ref{J D}). Clearly, $x^ty-y^tx = O$ if and only if  $\langle x\rangle = \langle y\rangle$. The vectors defined as in (\ref{adj D}) span $\cL$ and $e$ is a linearly $G$-homogeneous projective embedding of $\DnK$ in $\cL$. As in Subsection \ref{adj B sec}, the following holds.

\begin{proposition}\label{prop adj D} 
The embedding $e$ as defined in {\rm (\ref{adj D})} is indeed the adjoint embedding of $\DnK$.  
\end{proposition} 

\begin{remark}\label{plucker D}
\em
Likewise for $\BnK$, when $\ch(\FF) \neq 2$ the adjoint embedding of $D_{n,2}(\FF)$ is isomorphic to the Pl\"{u}cker embedding of $\DnK$ in $V\wedge V$.  
\end{remark} 
 
\subsubsection{The case where $\ch(\FF) = 2$}\label{L0 D}

Let $\ch(\FF) = 2$. In this case  $\cM_\cL\not\cong \cL$. Indeed the $G$-orbit n $\cM_\cL$ of the matrix $e_{n+1,2}-e_{n+2,1} = J\cdot(e_1^te_2 - e_2^te_1)$ spans a hyperplane $\cL_0$ of $\cM_\cL$, formed by the matrices as in (\ref{matrix D}) with $\Tr(X) = 0$, while $\cL$ admits no $G$-invariant hyperplane. Formula (\ref{adj D}) still defines a linearly $G$-homogeneous projective embedding $e_0:\DnK\rightarrow \cL_0$ but this embedding is the Pl\"{u}cker embedding of $\DnK$ in $\cL_0 \cong V\wedge V$, which is now a proper projection of the adjoint embedding $e:\DnK\rightarrow \cL$. The kernel $K(\cL)$ of the projection of $e$ onto $e_0$ is $1$-dimensional and it is contained in the center $Z(\cL)$ of $\cL$. If $n$ is odd then $K(\cL) = Z(\cL)$ and $Z(\cL_0)$ is trivial while, when $n$ is even, then $I \in \cL_0$ and $Z(\cL)/K(\cL) = Z(\cL_0) = \langle I\rangle$. In any case, $G$ centralizes $Z(\cL)$ (because $G$ is perfect and stabilizes both $K(\cL)$ and $Z(\cL)$). 

It is convenient to give $K(\cL)$ a name to distinguish it from $Z(\cL)$ when $n$ is even. We call $K(\cL)$ the {\em deep center} of $\cL$, keeping this name even when $n$ is odd. With $C = \{x_\alpha\}_{\alpha\in\Phi}\cup\{h_1, h_2,..., h_n\}$ a Chevalley basis of $\cL$, the deep center $K(\cL)$ of $\cL$ is spanned by the element $k_n := h_{n-1} - h_n$. 

\begin{remark}
\em 
With $C$ as above and $i\in \{2, 3,..., n\}$, put 
\[k_i ~ := ~  h_{i-1} - [x_{\alpha^+_{[i-1,i]}}, x_{\alpha^-_{[i-1, i]}}].\]
Then for $i < j$ we have $k_j -k_i = 2[x_{\alpha^+_{(i,j)}}, x_{\alpha^-_{(i,j)}}]$. Hence, when $\ch(\FF) = 2$ as we are assuming here, $k_i = k_n$ for every choice of $i, j \in \{2, 3,..., n\}$. If moreover $n$ is even, then $I = \sum_{i=1}^{n/2}h_{2i-1}$. 
\end{remark} 

We can safely regard $\cL_0$ as the subspace of $\cL$ spanned by the subset $C_0 = \{x_\alpha\}_{\alpha\in\Phi}\cup\{h_i\}_{i=1}^{n-1}$ of the Chevalley basis $C = \{x_\alpha\}_{\alpha\in\Phi}\cup\{h_i\}_{i=1}^n$ of $\cL$. Thus $\cL_0$ is a complement of $K(\cL)$ and the projection $\pi_0$ of $\cL$ onto $\cL_0$ along $K(\cL)$ acts on $C$ as follows: 
\begin{equation}\label{rootElm D char 2}
\pi_0 : \left\{ \begin{array}{rcl}
x_{\alpha^+_{(k,h)}} & \rightarrow & x_{\alpha^+_{(k,h)}} = e_{h,k}-e_{k+n,h+n}; \\
x_{\alpha^-_{(k,h)}} & \rightarrow &  x_{\alpha^-_{(k,h)}} = e_{k,h}-e_{h+n,k+n};\\
x_{\alpha^+_{[k,h]}} & \rightarrow & x_{\alpha^+_{[k,h]}} = e_{h+n,k}-e_{k+n,h}; \\
x_{\alpha^-_{[k,h]}} & \rightarrow  & x_{\alpha^-_{[k,h]}} = e_{k,h+n}-e_{h,k+n};\\
h_k  & \rightarrow  & -e_{k,k}+e_{k+1,k+1} + e_{n+k,n+k} - e_{n+k+1, n+k+1} = \\
 & & = ~ h_k, \mbox{ (for } k = 1, 2,..., n-1);\\ 
h_n & \rightarrow &   -e_{n-1,n-1}-e_{n,n} + e_{2n-1,2n-1} + e_{2n, 2n} = \\ 
 & & =  h_{n-1}. 
\end{array}\right.
\end{equation} 
Table (\ref{rootSbgr D}) still describes the matrices that correspond to the root elements of $G$ in its right action on $V$ and conjugation by these matrices yields the action of the root elements of $G$ on $\cL_0$.  

As a $G$-module, $\cL$ is a non-split central extension $\cL = K\cdot\cL_0$ of the $G$-module $\cL_0$ by the $1$-dimensional trivial $G$-module $K := K(\cL) = \langle k_n\rangle$. In view of Theorem \ref{thm link} this extension is described by a cocycle $F_0\in Z^1(\cL_0^*,G)\setminus B^1(\cL_0^*,G)$. The group $G$ acts as follows on $\cL = K\cdot \cL_0$: 
\[g ~:~ tk_n\oplus x \in K\oplus\cL_0 ~ \longrightarrow ~ (t+\langle F_0(g^{-1}), x\rangle)k_n\oplus g^{-1}xg\]
for every $g\in G$, $t\in \FF$ and $x\in \cL_0$. Moreover, 
\[\langle F_0(g^{-1}), x\rangle k_n ~ = ~ \pi_K(x\cdot g), \hspace{5 mm} \forall g\in G, ~\forall x\in \cL_0\]
for every $g\in G$ and $x\in \cL_0$, where $x\cdot g$ is the $g$-image of $x$ in the $G$-module $\cL$ and $\pi_K$ is the projection of $\cL$ onto $K = Z(\cL)$ along $\cL_0$.  

We call $F_0$ the {\em seed cocycle} of $G$ over $\cL_0^*$. The statements of Lemma \ref{lemma F0 B} and Proposition \ref{prop adj D car 2} hold true word by word for $\DnK$ too. We recall them here, so we can refer to them later.  

\begin{lemma}\label{lemma F0 D}
Let $\Phi_0 := \cup_{1 \leq k < h \leq n+1}\{\alpha_{k,h}^+, \alpha_{k,h}^-\}$ be the subsystem of $\Phi$ of type $A_{n-1}$ generated by the roots $\alpha_1, \alpha_2,..., \alpha_{n-1}$. Then the following holds for every $t\in \FF$.   
\begin{itemize}
\item[$(1)$] If $\alpha\in \Phi_0$ then $\langle F_0(u_\alpha(t)), x\rangle = 0$ for every $x\in \cL_0$;
\item[$(2)$] Let $\alpha\in\Phi\setminus\Phi_0$. Then $\langle F_0(u_\alpha(t)), x\rangle = 0$ for every $x\in C_0\setminus\{x_{-\alpha}\}$ while $\langle F_0(u_\alpha(t)), x_{-\alpha} \rangle = t$. 
\end{itemize}
\end{lemma}

\begin{proposition}\label{prop adj D car 2} 
For every $g\in G$, the adjoint embedding maps the point $\langle e_1, e_2\rangle\cdot g$ of $\DnK$ onto the point of $\PG(\cL)$ represented by the vector
\[\langle F_0(g^{-1}), x_{\alpha^+_{[1,2]}}\rangle k_n\oplus g^{-1}x_{\alpha^+_{[1,2]}}g.\]
\end{proposition}

\subsection{The mapping $\gamma:\Der(\FF)\rightarrow Z^1(\cL^*,G)$}\label{gamma BD} 

In the sequel we shall discuss the $B_n$-case and $D_n$-case at the same time, to the extent we can. So, henceforth $X_n, X_{n,2}(\FF), G, \cL$ and $\cM$ stand for $B_n$ $(n> 2)$, $\BnK$, $\OnK$, $\onK$ and $M_{2n+1}(\FF)$ or $D_n$, $\DnK$, $\OhK$, $\ohK$ and $M_{2n}(\FF)$ respectively, $\cM_\cL$ being defined as in Subsection \ref{adj B sec} or \ref{adj D sec} respectively.   

We shall define a function $\gamma:\Der(\FF)\rightarrow Z^1(\cL^*,G)$ in such a way that $\gamma$ satisfies properties $(\gamma1)$, $(\gamma2)$, $(\gamma3)$ and $(\gamma4)$ of Section \ref{GenConstr}. By imitating the construction made in Subsection \ref{gamma A}, we shall define a suitable bilinear mapping $\tau:\cM_\cL\times\cL\rightarrow\FF$ and for every $d\in\Der(\FF)$ we state that, by definition, the $\gamma$-image $\gamma_d$ of $d$ is the $1$-cocycle that maps $g\in G$ onto the linear functional 
\begin{equation}\label{gamma BD eq} 
\gamma_d(g) ~ : ~ x\in\cL ~ \longrightarrow ~ \langle \gamma_d(g), c) ~: = ~ \tau(g^{-1}d(g), x). 
\end{equation}
This definition makes sense since, by an argument quite similar to the one used in Remark \ref{g**(-1)d(g)}, we can prove that $g^{-1}d(g) \in \cM_\cL$ for every $g\in G$.

\begin{remark}\label{g**(-1)d(g) BD} 
\em
The same argument as in Remark \ref{g**(-1)d(g)} also shows that if $\ch(\FF) = 2$ then $g^{-1}d(g)\in \cL_0$ for every $g\in G$ and every $d\in \Der(\FF)$.
\end{remark}  

We shall prove that $\tau$ can be chosen in such a way that the function $\gamma$ defined as above is indeed an injective linear mapping from $\Der(\FF)$ to $Z^1(\cL^*,G)$ and satisfises properties $(\gamma1)$, $(\gamma2)$, $(\gamma3)$ and $(\gamma4)$ of Section \ref{GenConstr}. 

\subsubsection{The form $\tau:\cM_\cL\times\cL\rightarrow\FF$}\label{tau 1}

Suppose firstly that $\ch(\FF) \neq 2$. Then $\cL = \cM_\cL \subset \cM$. As in Section \ref{gamma A}, let $T:\cM\times \cM\rightarrow \FF$ be the symmetric bilinear form which maps every pair$(x,y)\in\cM\times\cM$ onto the scalar $T(x,y) := \Tr(xy)$. This form is non-degenerate and induces a symmetric bilinear form on $\cL\times \cL$. It is not difficult to check that the subspace of $\cM$ orthogonal to $\cL$ with respect to $T$ meets $\cL$ trivially. In other words, the bilinear form induced by $T$ on $\cL$ is non-degenerate. We define $\tau$ as the restriction of $T$ to $\cL\times \cL$. 

\begin{remark}
\em
When $\ch(\FF) = 0$ the form $\tau$ defined in this way is proportional to the Killing form $\kappa$ of $\cL$. In particular, if $X_n = B_n$ then $\kappa = (2n+1)\tau$ while if $X_n = D_n$ then $\kappa = 2n\tau$. (Compare Remark \ref{killing}.) 
\end{remark} 

Let now $\ch(\FF) = 2$. We know that in this case $\cL$ is not a subalgebra of $\cM$ but the worst is that $T$ induces the null form on $\cM_\cL\times \cM_\cL$. This makes it clear that we cannot hope to find the right definition of $\tau$ only through small tweaks to the definition given when $\ch(\FF) \neq 2$. However there is a way out, which we are about to show now.   

Dropping for a while the hypothesis that $\ch(\FF) = 2$, suppose that $\ch(\FF) \neq 2$ and compute $\frac{1}{2}T(m, x)$ for $b, x \in\cM_\cL$, say 
\[\begin{array}{ccl}
m ~ = ~ \left(\begin{array}{cc}
A & B \\
C & -A^t 
\end{array}\right), & 
x ~ = ~ \left(\begin{array}{cc}
X & Y  \\
Z & -X^t 
\end{array}\right), & \mbox{ if } X_n = D_n;\\
 & & \\
m ~ = ~ \left(\begin{array}{ccl}
A & B & a^t\\
C & -A^t & b^t\\
-2b & -2a & 0 
\end{array}\right), &
x ~ = ~ \left(\begin{array}{ccl}
X & Y & x^t \\
Z & -X^t  & y^t \\
-2y & -2x & 0
\end{array}\right), & \mbox{ if } X_n = B_n.
\end{array}\]
Then, if $b_{i,j}, c_{i,j}, y_{i,j}$ and $z_{i,j}$ are the $(i,j)$-entries of the matrices $B,C, Y$ and $Z$ respectively, for $X = D_n$ we have
\[\frac{1}{2}T(m,x) ~ = ~ \Tr(AX) + \sum_{i<j}b_{i,j}z_{j,i} + \sum_{i<j}c_{i,j}y_{j,i},\]
while when $X = B_n$ we have 
\[\frac{1}{2}T(m,x) ~= ~ \Tr(AX) + \sum_{i<j}b_{i,j}z_{j,i} + \sum_{i<j}c_{i,j}y_{j,i} - 2ay^t - 2bx^t.\]
Turning back to characteristic $2$, when $\ch(\FF) = 2$ the above equations boil down to the same. We take them as definitions, namely we define a symmetric bilinear form $\tau_\cM$ on $\cM_\cL$ by setting  
\begin{equation}\label{demitrace0}
\tau_\cM(m, x)  ~ := ~ \Tr(AX) + \sum_{i<j}b_{i,j}z_{j,i} + \sum_{i<j}c_{i,j}y_{j,i}.
\end{equation}
This form is still degenerate.  Indeed if $X_n = D_n$ then $\mathrm{Rad}(\tau_\cM) = \langle I\rangle$; if $X_n = B_n$ then $\mathrm{Rad}(\tau_\cM)$ is spanned by $I$ together with the matrices $e_{i,2n+1}$ for $i = 1, 2,..., 2n$. However $\mathrm{Rad}(\tau_\cM)$ is small enough for we can hope that a form $\tau:\cM_\cL\times \cL\rightarrow \FF$ defined by means of $\tau_\cM$ will do the job we need. We define $\tau$ as follows: for every $m\in \cM_\cL$ and $x\in \cL$ we put 
\begin{equation}\label{tau car 2}
\tau(m, x) ~ := ~ \tau_\cM(m, \pi_0(x))
\end{equation}
where $\pi_0$ is the projection of $\cL$ onto $\cL_0$ along the center $Z(\cL)$ of $\cL$ (if $X_n = B_n$) or along the deep center $K(\cL)$ of $\cL$ (if $X_n = D_n$).  

\begin{lemma}\label{GT} 
The form $\tau_\cM$ is preserved by $G$ in its action on $\cM_\cL$. 
\end{lemma}
\noindent
{\em Proof.} Suppose firstly that $\ch(\FF)\neq 2$. The group $G$ preserves the form $T$ (the trace of a matrix is invariant under conjugation). Hence $G$ also preserves $\frac{1}{2}T$. The scalar $\frac{1}{2}T(m,x)$ can be regarded as a polynomial with integral coefficients and the entries of $m$ and $x$ as unknowns. The entries of the matrices that represent the generators $u_\alpha(t)$ of $G$ are also polynomial with integral coefficients and $t$ as the unknown. Therefore, since when $\ch(\FF)\neq 2$ these generators preserve $\frac{1}{2}T$ and since $\tau_\cM(m,x)$ is nothing but the reduction modulo $2$ of the polynomial $\frac{1}{2}T(m,x)$, when $\ch(\FF) =2$ the generators of $G$ preserve $\tau_\cM$. The group $G$ does the same.   \hfill $\Box$

\begin{corollary}\label{G tau}
The form $\tau$ defined as in {\rm (\ref{tau car 2})} is preserved by $G$ in its action on $\cM_\cL\times\cL$. In particular
\begin{equation}\label{tau coc}
\tau(g^{-1}mg, x) ~ = ~ \tau(m, x\cdot g^{-1}), \hspace{5 mm} \forall m\in \cM_\cL,~ \forall x\in \cL, ~\forall g\in G.
\end{equation}
\end{corollary} 
{\em Proof.} For $(m,x)\in \cM_\cL\times\cL$ and $g\in G$ we have $(m, x)\cdot g = (g^{-1}mg, x\cdot g)$ and $\pi_0(x\cdot g) = g^{-1}\pi_0(x)g$. Also $\tau_\cM(g^{-1}mg, g^{-1}\pi_0(x)g) = \tau_\cM(m, \pi_0(x))$ by Lemma \ref{GT}. Hence $\tau(g^{-1}mg, x\cdot g) = \tau(m, x)$, as claimed. Formula (\ref{tau coc}) is a trivial consequence of the fact that $G$ preserves $\tau$.  \hfill $\Box$  

\subsubsection{Properties of the mapping $\gamma$}

For every $d\in \Der(\FF)$ let $\gamma_d$ be defined as in (\ref{gamma BD eq}) with $\tau$ as in Subsection \ref{tau 1}. As in the proof of Lemma \ref{A gamma coc} (and with the help of (\ref{tau coc}) when $\ch(\FF) = 2$) we can prove that $\gamma_d$ indeed belongs to $Z^1(\cL^*,G)$. By an argument similar to the one used in the proof of Lemma \ref{A gamma inj} we can also prove that the linear mapping 
\[\begin{array}{rcccl}
\gamma& : & \Der(\FF) & \longrightarrow  & Z^1(\cL^*,G) \\
 & & d & \longrightarrow & \gamma_d
\end{array}\]
is injective. Since for every derivation basis $\Omega$ of $\FF$ every $t\in \FF$ belongs to the subfield of $\FF$ generated by $N_{\mathrm{der}}(\FF)\cup \Omega_t$ for a finite subset $\Omega_t\subseteq \Omega$, the mapping $\gamma$ satisfies property $(\gamma2)$ of Section \ref{GenConstr}. Turning to properties $(\gamma3)$ and $(\gamma4)$, recall that $G$ acts flag-transitively on $X_{n,2}(\FF)$. So, we only need to consider a point $p_0$ of $X_{n,2}(\FF)$ and a line $\ell$ of $X_{n,2}(\FF)$ through it and check if properties $(\gamma3)$ and $(\gamma4)$ hold for that choice of $p_0$ and $\ell$. The line $p_0 = \langle e_1, e_2\rangle$ of the polar space $X_{n,1}(\FF)$ and the point-plane flag $\ell = \{\langle e_1\rangle, \langle e_1, e_2, e_3\rangle\}$ of $X_{n,1}(\FF)$ do the job, as the reader can check. 

As for property $(\gamma1)$, the following will be proved in Sections \ref{Cohomology Th odd} and \ref{Cohomology Th even}.  

\begin{theorem}\label{Cohomology}
If $|\FF| > 3$ then $\gamma$ also satisfies property $(\gamma1)$.
\end{theorem}     
\begin{remark}
\em
Maybe $H^1(\cL^*,G)$ is trivial when $|\FF| \in \{2,3\}$, so that $\gamma$ satisfies $(\gamma1)$ in this case too, but I am unable to prove this claim; anyway, as we shall explain in the next subsection, we don't need it. 
\end{remark}   

\subsection{The main theorem of this paper} 

As in Subsection \ref{gamma BD}, throughout this subsection $X_n, X_{n,2}(\FF), G, \cL$ and $\cM$ stand for $B_n$ $(n> 2)$, $\BnK$, $\OnK$, $\onK$ and $M_{2n+1}(\FF)$ or $D_n$, $\DnK$, $\OhK$, $\ohK$ and $M_{2n}(\FF)$ respectively and $\cM_\cL$ is defined as in Subsection \ref{adj B sec} or \ref{adj D sec} respectively. The linear mapping $\gamma:\Der(\FF)\rightarrow Z^1(\cL^*,G)$ defined in Subsection \ref{gamma BD} is injective and satisfies properties $(\gamma2)$, $(\gamma3)$ and $(\gamma4)$. 

Suppose that $|\FF| > 3$. Then $\gamma$ also satisfies property $(\gamma1)$.  Moreover, the $G$-module $\cL$ is rigid (this can be proved by an argument similar to the one used to prove Lemma \ref{rigid A}) and, as noticed in Remark \ref{central rem}, the adjoint embedding $e:X_{n,2}(\FF)\rightarrow \cL$ of $X_{n,2}(\FF)$ satisfies property $(G2)$ of Theorem \ref{theo nuovo}. Therefore, by Theorem \ref{theo nuovo}, when $|\FF| > 3$ the universal cover of $e$ is the embedding $e_{\gamma, p_0}:X_{n,2}(\FF)\rightarrow\Der_\Omega(\FF)\oplus \cL$ constructed as in Section \ref{GenConstr}. Explicitly, recall that $G$ acts as follows on $\Der_\Omega(\FF)\oplus \cL$
\[g\in G ~ : ~ d\oplus x ~ \longrightarrow (d\oplus x)\cdot_\gamma g ~ = ~  (d+ \sum_{\omega\in\Omega}\tau(gd_\omega(g^{-1}), x)d_\omega)\oplus x\cdot g\]
and, given $p_0 = \langle e_1, e_2\rangle$ in $X_{n,2}(\FF)$, so that $e(p_0)$ is represented by the vector 
\[J\cdot(e_1^te_2-e_2^te_1) ~= ~e_{n+2,1}-e_{n+1,2} ~= ~  x_{\alpha^+_{[1,2]}},\] 
for $g\in G$ the point $e(p_0\cdot g)$ of $X_{n,2}(\FF)$ is mapped by $e_{\gamma, p_0}$ onto the point of $\PG(\Der_\Omega(\FF)\oplus \cL)$ represented by the $g$-image $(x_{\alpha^+_{[1,2]}})\cdot_\gamma g$ of $x_{\alpha^+_{[1,2]}}$ in the $G$-module $\Der_\Omega(\FF)\oplus \cL$: 
\begin{equation}\label{image BD}
(x_{\alpha^+_{[1,2]}})\cdot_\gamma g ~ = ~ (\sum_{\omega\in\Omega}\tau(gd_\omega(g^{-1}), x_{\alpha^+_{[1,2]}})d_\omega)\oplus (x_{\alpha^+_{[1,2]}}\cdot g).
\end{equation}
We have stated the above for $|\FF| > 3$. However, by Cooperstein \cite{Coop1} (also V\"{o}lklein \cite{V2}), when $\FF$ is prime the adjoint embedding of $X_{n,2}(\FF)$ is relatively universal. On the other hand, $\Der(\FF) = \{0\}$ when $\FF$ is prime. Therefore all the above trivially holds true  when $|\FF| \in \{2, 3\}$ too. Moreover, the geometry $X_{n,2}(\FF)$ admits the absolutly universal embedding (Kasikova and Shult \cite{KS}); consequently, the universal cover $e_{\gamma,p_0}$ of the adjont embedding of $X_{n,2}(\FF)$ is just the absolutely universal embedding of $X_{n,2}(\FF)$. So, recalling that the points of $X_{n,2}(\FF)$ are the totally singular $2$-subspaces of $V$, all the above can be rephrased as follows.

\begin{theorem}\label{main}
The absolute embedding of $X_{n,2}(\FF)$ lives in $\Der_\Omega(\FF)\oplus\cL$ and is isomorphic to the embedding that maps every totally singular $2$-subspace $\langle v, w\rangle$ of $V$ onto the point of $\PG(\Der_\Omega(\FF)\oplus\cL)$ represented by the following vector:
\begin{equation}\label{main eq}
\sum_{\omega\in\Omega}d_\omega(v)Jw^td_\omega \oplus e(v,w)
\end{equation}
where $e(v,w)\in \cL$ is defined as follows.
\begin{itemize}
\item[(1)] If $\ch(\FF) \neq 2$ then $e(v,w) := J(v^tw-w^tv)$. 
\item[(2)] Let $\ch(\FF) = 2$, let $e:X_{n,2}(\FF)\rightarrow\cL$ be the adjont embedding of $X_{n,2}(\FF)$ and $\pi_0$ the projection of $\cL$ onto $\cL_0$ along the center or the deep center of $\cL$ (according to whether $X_n = B_n$ or $X_n = D_n$). Then $e(v,w)$ is the unique representative vector of the point $e(\langle v,w\rangle)$ of $\PG(\cL)$ that is mapped onto $J(v^tw-w^tv)$ by $\pi_0$. 
\end{itemize} 
\end{theorem}   
{\em Proof.} Let $v, w\in V$ be non-proportional singular vectors such that $vJw^t = 0$. If we replace $v$ and $w$ with $rv$ and $sw$ for scalars $r,s\in \FF$ then $J(v^tw-w^tv)$ is replaced by $rs\cdot J(v^tw-w^tv)$ and, since $d_\omega(rv)J(sw)^t = d_\omega(r)s\cdot vJw^t + rs\cdot d_\omega(v)Jw^t = rs\cdot d_\omega(v)Jw^t$, the vector $\sum_{\omega\in\Omega}d_\omega(v)Jw^td_\omega$ is replaced by
$rs(\sum_{\omega\in\Omega}d_\omega(v)Jw^td_\omega)$. Thus, replacing $v$ and $w$ with $rv$ and $sw$ amounts to multiplying the vector described in (\ref{main eq}) by $rs$. So, up to rescaling $v$ and $w$, we can assume that $v = e_1\cdot g$ and $w = e_2\cdot g$ for some $g\in G$. Under this hypothesis and according to (\ref{image BD}), the absolute embedding of $X_{n,2}(\FF)$ maps the $2$-subspace $\langle v,  w\rangle$ of $V$ onto the point of $\PG(\Der_\Omega(\FF)\oplus\cL)$ represented by the following vector: 
\begin{equation}\label{main eq 1}
(\sum_{\omega\in\Omega}\tau(gd_\omega(g^{-1}), J(e_1^te_2-e_2^te_1))d_\omega)\oplus e(v,w). 
\end{equation}
Suppose firstly that $\ch(\FF) \neq 2$. Then (\ref{main eq 1}) can be written as follows:
\begin{equation}\label{main eq 2}
(\sum_{\omega\in\Omega}\Tr(gd_\omega(g^{-1})J(e_1^te_2-e_2^te_1))d_\omega)\oplus J(v^tw-w^tv). 
\end{equation}
Recall that $gJg^t = J$. Therefore $g^{-1}J = Jg^t$. Hence $d_\omega(g^{-1})J = d_\omega(g^{-1}J) = d_\omega(Jg^t) = Jd_\omega(g^t)$. Accordingly, 
\[\begin{array}{l}
\Tr(gd_\omega(g^{-1})J(e_1^te_2-e_2^te_1)) =  \Tr(gJ(d_\omega(g^t)e_1^te_2- d_\omega(g^t)e_2^te_1)) =\\
=  \Tr(gJd_\omega(e_1\cdot g)^te_2- d_\omega(e_2\cdot g)^te_1)) = \Tr(gJ(d_\omega(v^t)e_2 - d_\omega(w^t)e_1)) = \\
= \Tr(J(d_\omega(v^t)e_2g - d_\omega(w^t)e_1g)) =  \Tr(J(d_\omega(v^t)w - d_\omega(w^t)v).
\end{array}\]
So, (\ref{main eq 2}) can be rewriten as follows, where no mention of $g$ is made. 
\begin{equation}\label{main eq 3}
(\sum_{\omega\in\Omega}\Tr(J(d_\omega(v)^tw - d_\omega(w)^tv))d_\omega)\oplus J(v^tw-w^tv). 
\end{equation}
Note that $\Tr(Jx^ty) = xJy^t$ for any choice of vectors $x, y \in V$. Hence 
\[\begin{array}{l}
\Tr(J(d_\omega(v)^tw - d_\omega(w)^tv) = d_\omega(v)Jw^t - d_\omega(w)Jv^t = \\
=  d_\omega(v)Jw^t + vJd_\omega(w)^t - vJd_\omega(w)^t - d_\omega(w)Jv^t = d_\omega(vJw^t) - 2 vJd_\omega(w)^t.
\end{array}\]
However $vJw^t = 0$ by assumption. Therefore $\Tr(J(d_\omega(v)^tw - d_\omega(w)^tv)  = -2vJd_\omega(w)^t$. Moreover $-vJd_\omega(w)^t = d_\omega(v)Jw^t$, since $d_\omega(v)Jw^t + vJd_\omega(w)^t = d_\omega(vJw^t) = d_\omega(0) = 0$. So  
(\ref{main eq 3}) boils down to the following:
\begin{equation}\label{main eq 4}
(2\cdot \sum_{\omega\in\Omega}d_\omega(v)Jw^td_\omega)\oplus J(v^tw-w^tv).
\end{equation}
The linear mapping that maps every vector $d\oplus x \in \Der_\Omega(\FF)\oplus\cL$ onto $(\frac{1}{2}d)\oplus x$ transforms (\ref{main eq 4}) into (\ref{main eq}). When $\ch(\FF) \neq 2$ the statement is proven. 

Let now $\ch(\FF) = 2$. Then (\ref{main eq 1}) must be read as follows:
\[(\sum_{\omega\in\Omega}\tau_{\cM}(gd_\omega(g^{-1}), \pi_0(J(e_1^te_2-e_2^te_1)))d_\omega)\oplus e(v,w),\]
hence also as follows, since $\pi_0(J(e_1^te_2-e_2^te_1)) = J(e_1^te_2-e_2^te_1)$ because $J(e_1^te_2-e_2^te_1) = x_{\alpha^+_{[1,2]}} \in \cL_0$. 
\begin{equation}\label{main eq 5}
(\sum_{\omega\in\Omega}\tau_{\cM}(gd_\omega(g^{-1}), J(e_1^te_2-e_2^te_1))d_\omega)\oplus e(v,w) 
\end{equation}
Here $e(v,w) = F_0(g^{-1})(x_{\alpha^+_{[1,2]}})\chi_n+ x_{\alpha^+_{[1,2]}}$, where $F_0$ is the seed cocycle of $G$ over $\cL_0^*$ and $\chi_n$ stands for $h_n$ (if $X_n = B_n$) or $h_{n-1}-h_n$ (if $X_n = D_n$). By arguments similar to those exploited when $\ch(\FF) \neq 2$ we can show that (\ref{main eq 5}) boils down to (\ref{main eq}). We leave the details to the reader. (Hint: replace $\Tr$ with $\frac{1}{2}\Tr$ in the proof we have given for the case $\ch(\FF)\neq 2$.)  \hfill $\Box$ 

\section{Proof of Theorem \ref{Cohomology} when $\ch(\FF) \neq 2$}\label{Cohomology Th odd}

As in Subsection \ref{gamma BD}, throughout this section $X_n, X_{n,2}(\FF), G, \cL$ and $\cM$ stand for $B_n$ $(n> 2)$, $\BnK$, $\OnK$, $\onK$ and $M_{2n+1}(\FF)$ or $D_n$, $\DnK$, $\OhK$, $\ohK$ and $M_{2n}(\FF)$ respectively and $\cM_\cL$ is defined as in Subsection \ref{adj B sec} or \ref{adj D sec} respectively. Moreover we assume that $\ch(\FF) \neq 2$. Hence $\cL = \cM_\cL$ and the mapping 
\[ m\in \cL ~  \longrightarrow ~ m^* ~ : ~ \left(\begin{array}{c}
x\in \cL \\
\downarrow \\
m^*(x) := \Tr(mx) 
\end{array}\right)\]
is an isomorphism of $G$-modules from $\cL$ to $\cL^*$. So, we can replace $\cL^*$ with $\cL$, thus regarding $\gamma$ as an injective  linear mapping from $\Der(\FF)$ to $Z^1(\cL, G)$ rather than to $Z^1(\cL, G)$.  Explicitly, 
\[\gamma ~ : ~  d\in \Der(\FF) ~  \longrightarrow ~ \gamma_d  ~:~ 
\left(\begin{array}{c}
g\in G \\
\downarrow\\
g^{-1}d(g) \in \cL
\end{array}\right).\]
So, when $\ch(\FF)\neq 2$ the statemente of Theorem \ref{Cohomology} is equivalent to the following. 

\begin{theorem}\label{Cohom odd}
Let $\ch(\FF)\neq 2$ and $|\FF|\neq 3$. Then $\gamma$, regarded as an injective linear transformation from $\Der(\FF)$ to $Z^1(\cL,G)$, maps $\Der(\FF)$ onto a complement of $B^1(\cL, G)$ in $Z^1(\cL,G)$.
\end{theorem}   

The present section is devoted to the proof of the above. The proof we shall give in Section \ref{Cohomology Th even} for the characteristc 2 case of Theorem \ref{Cohomology} can be rearranged so that to obtain a proof of Theorem \ref{Cohom odd}, but we prefer to offer a shorter proof of this theorem, which relies on what literature offers us. We firstly prove the following.

\begin{lemma}\label{Cohom odd Lem 1} 
Let $\ch(\FF)\neq 2$. Then the image $\gamma(\Der(\FF))$ of the mapping $\gamma:\Der(\FF)\rightarrow Z^1(\cL,G)$  meets $B^1(\cL, G)$ trivially. 
\end{lemma}
{\em Proof.} Supose that for some $d\in \Der(\FF)$ there exists $x\in \cL = \cM_\cL$ such that $g^{-1}d(g) = x-g^{-1}xg$ for every $g\in G$. Put $g = u_\alpha(t)$ for $\alpha$ any root of the root system of $\cL$ and $t$ any scalar in $\FF$. With $g$ chosen in this way, it is straightfoward to check that the equality $g^{-1}d(g) = x-g^{-1}xg$ forces $x = O$. (The details of this check are left to the reader.) Hence $\gamma_d = 0$ and, threrefore, $d = 0$. \hfill $\Box$  \\

According to V\"{o}lklein \cite{V1}, if $\ch(\FF) \neq 2$ and $|\FF| \neq 3$ then $\Der(\FF) \cong H^1(\cL,G)$. By combining this result with Lemma \ref{Cohom odd Lem 1} we immediately obtain the following.

\begin{corollary}\label{Cohom odd Lem 2} 
Let $\ch(\FF)\neq 2$ and $|\FF|\neq 3$ and suppose that $\dim(\Der(\FF)) < \infty$. Then $\gamma(\Der(\FF))$ is a complement of  $B^1(\cL, G)$ in $Z^1(\cL, G)$. 
\end{corollary}

Henceforth, keeping the hypotheses that $\ch(\FF) \neq 2$, we assume that $\dim(\Der(\FF))$ is infinite. We shall only consider the case where $X_n = B_n$, but all we are going to say on this case can be repeated nearly word for word for the case that $X_n = D_n$.  

So henceforth  $G = \Omega(2n+1, \FF)$ with $\ch(\FF) \neq 2$ and $\cL = \mathfrak{o}(2n+1,\FF)$. In view of Lemma \ref{Cohom odd Lem 1}, the following is all we need to prove.

\begin{claim}\label{main claim odd}
For every cocycle $F\in Z^1(\cL,G)$ there exists $d\in\Der(\FF)$ and $m\in \cL$ such that
\[F(g)-g^{-1}d(g) ~ = ~ m-g^{-1}mg, \hspace{5 mm} \forall g\in G.\]
\end{claim} 

\subsection{A bit of notation}\label{conv} 

We firstly fix some notation. Given $F\in Z^1(\cL,G)$, we put $\delta_F(g) = gF(g)$. With this convention we have
\begin{equation}\label{delta}
\delta_F(g_1g_2) ~ = ~ \delta_F(g_1)g_2 + g_1\delta_F(g_2), ~~~ \forall g_1, g_2\in G.
\end{equation}
Clearly, $\delta_F(1) = O$ and $\delta_F(g^{-1}) = -g^{-1}\delta_F(g)g^{-1}$. 

Given a subfield $\KK$ of $\FF$, we denote by $\cL(\KK)$ (respectively $G(\KK)$) the $\KK$-subalgebra of $\cL$ (the subgroup of $G$) formed by the matrices of $\cL$ (of $G$) with all entries in $\KK$.  Given a basis $B$ of $\FF$ over $\KK$ such that $1\in B$, let $p_{\KK, B}:\FF\rightarrow \KK$ be the projection of $\FF$ onto $\KK$ along the subspace of $\FF$ spanned by $B\setminus\{1\}$. As $p_{\KK,B}$ is $\KK$-linear, if $x\in \KK$ and $y\in \FF$ then $p_{\KK,B}(xy) = xp_{\KK,B}(y)$ ($= p_{\KK,B}(yx)$). Consequently 
\[p_{\KK,B}(gF(g)) ~ = ~ gp_{\KK,B}(F(g)), ~~~ \forall g\in G(\KK).\] 
Therefore, with $F\in Z^1(\cL,G)$ and $g_1, g_2\in G(\KK)$, we have
\begin{equation}\label{G(K)bis}
\begin{array}{rcl}
p_{\KK,B}(F(g_1g_2)) & = & g_2^{-1}p_{\KK,B}(F(g_1))g_ 2+ p_{\KK,B}(F(g_2));\\
p_{\KK,B}(\delta_F(g_1g_2)) & =& p_{\KK,B}(\delta_F(g_1))g_2 + g_1p_{\KK,B}(\delta_F(g_2)).
\end{array}
\end{equation} 
It follows that the function which maps every $g\in G(\KK)$ onto $p_{\KK,B}(F(g))\in \cL(\KK)$ is a cocycle of $G(\KK)$ over $\cL(\KK)$. Note also that the two equations of  (\ref{G(K)bis}) are mutually equivalent. 

\subsection{Down to subfields with finite derivation bases}\label{sec 121}

We firstly state a few definitions. Let $\ch(\FF) = p >  0$. Then we put $N(\FF) := \cap_{n=0}^\infty\FF^{p^n}$ (of course, $x^{p^0} = x$, by convention). In other words, $N(\FF)$ the largest perfect subfield of $\FF$. If  $x\in \FF^p$ we denote by $x^{p^{-1}}$ the unique element $y\in \FF$ such that $x = y^p$ and when $x\in \FF\setminus\FF^p$ we write $x^{p^{-1}} = \infty$, by convention. So, the elements of $\FF\setminus N(\FF)$ are precisely the $x\in \FF$ such that $x^{p^{-n}} = \infty$ for some positive integer $n$. Given a subset $X$ of $\FF\setminus N(\FF)$, we put 
\[X^{p^{-1}} ~ :=  ~ \{x^{p^{-1}}~|~ x\in X, ~ x^{p^{-1}} \neq \infty\}.\]
Clearly, $X\cap \FF^{p^n} = \emptyset$ if and only if $X^{p^{-n}} = \emptyset$. Still assuming that $X\cap N(P) = \emptyset$, we put $\overline{X} := \cup_{n = 0}^\infty X^{p^{-n}}$. It is clear that if $X$ is finite then $\overline{X}$ is finite. Indeed when $X$ is finite there exists a nonnegative integer $n$ such $X^{p^{-n}} = \emptyset$, hence $\overline{X} = \cup_{k = 0}^{n-1}X^{p^{-k}}$. It is also clear that, if $\KK_X$ is the subfield of $\FF$ generated by $N(\FF)\cup\overline{X}$, then
\begin{equation}\label{N(F)} 
\KK_X^p ~ = ~ \KK_X\cap\FF^p.
\end{equation}
Let $\ch(\FF) = 0$. Then $N(\FF) := N_{\mathrm{der}}(\FF)$ is the algebraic closure in $\FF$ of the minimum subfield of $\FF$. For uniformity of notation, for every $X\subseteq \FF\setminus N(\FF)$ we put $\overline{X} = X$.  

Let $\mathfrak{F}(\FF)$ be the family of all subfields $\KK$ of $\FF$ such that $\KK$ is generated by $N(\FF)\cup\overline{X}$ for a finite subset $X$ of $\FF\setminus N(\FF)$ (and $|\KK| > 4$ if $|N(\FF)| \leq 4$). This family is nonempty (in fact, since $\dim(\Der(\FF)) = \infty$, it contains infinitely many members) and $\dim(\Der(\KK)) < \infty$ for every $\KK\in \mathfrak{F}(\FF)$. Moreover, for a derivation $d\in \Der(\FF)$, if $d(x)\in \KK$ for every $x\in\KK$ then the restriction of $d$ to $\KK$ belongs to $\Der(\KK)$. Conversely, every derivation $d\in \Der(\KK)$ can be extended to a derivation of $\FF$. 

Let $\KK\in \mathfrak{F}(\FF)$. By Corollary \ref{Cohom odd Lem 2} there exists a derivation $d_{\KK,B}$ of $\KK$ and a matrix $m_{\KK,B}\in \cL(\KK)$ such that 
\begin{equation}\label{G(K)ter}
p_{\KK,B}(F(g)) ~ = ~ g^{-1}d_{\KK,B}(g) + m_{\KK,N,F}-g^{-1}m_{\KK,B}g, ~~~ \forall g\in G(\KK).
\end{equation}
Equivalently,
\begin{equation}\label{G(K)ter'}
p_{\KK,B}(\delta_F(g)) ~ = ~ d_{\KK,B}(g) + gm_{\KK,B}-m_{\KK,B}g, ~~~ \forall g\in G(\KK).
\end{equation}

\subsection{The matrix $F(g)$ for $g$ belonging to the subgroup of $G$ generated by $\{u_{k,h}(t)~|~ t\in \FF, ~ k, h \in \{1, 2,..., n\},~ k\neq h\}$}\label{sec 2.3}  

Let $G_1\cong \mathrm{SL}(n,\FF)$ be the subgroup of $G$ generated by the elements $u_{k,h}(t)$ with $k, h \leq n$ (notation as in table (\\ref{rootSbgr B})) and let $\rho:\cM\rightarrow M_n(\FF)$ be the function which maps every matrix of $\cM = M_{2n+1}(\FF)$ onto its $\{1,..., n\}\times\{1,...,n\}$-submatrix. The restriction of $\rho$ to $G_1$ yields the isomorphim $G_1\cong \mathrm{SL}(n,\FF)$ while the function wich maps every $F\in Z^1(\cL,G)$ onto the $1$-cochain of $G_1$ over $M_n(\FF)$ which maps every $g\in G_1$ onto $\rho(F(g))$ is a homomorphism from $Z^1(\cL,G)$ to $Z^1(M_n(\FF),G_1)$. By \cite[Theorem 1]{TZ}, there exists a matrix $m_{1,F}\in M_n(\FF)$ and a derivation $d_F\in \Der(\FF)$ such that
\[\rho(F(g)) ~ = ~  \rho(g^{-1}d_F(g)) + m_{1,F}-\rho(g)^{-1}m_{1,F}\rho(g)\]
for every $g\in G_1$. We can replace $F(g)$ with $F(g) - (m_F-g^{-1}m_Fg)$ where 
\[m_F ~ =~ \left(\begin{array}{ccc}
m_{1,F} & O_n & {\bf 0}^t \\
O_n & -m_{1,F}^t & {\bf 0}^t \\
{\bf 0} & {\bf 0} & 0
\end{array}\right)\]
With this replacement, we have $\rho(F(g)) =  \rho(g^{-1}d_F(g))$ for every $g\in G_1$. In particular, with $g = u_{k,h}(t)$ we have 
\[\begin{array}{rcl}
\rho(F(u_{k,h}(t))) & = &   \rho(u_{k,h}(-t)d_F(u_{k,h}(t))) ~ =\\
= ~ d_F(t)\rho(e_{k,h} - e_{n+h,n+k}) & = & \rho(\delta_F(u_{k,h}(t)).
\end{array}\]
Accordingly,  
\[\begin{array}{l}
\delta_F(u_{k,h}(t))) ~ = ~ d_F(t)(e_{k,h}-e_{n+h,n+k}) + u_{k,h}(t)M_{k,h}-M_{k,h}u_{k,h}(t) ~ = \\
=  d_F(t)(e_{k,h}-e_{n+h,n+k}) + t((e_{k,h}-e_{n+h,n+k})M_{k,h}-M_{k,h}(e_{k,h}-e_{n+h,n+k}))
\end{array}\]
where 
\[M_{k,h} ~ = ~ \left(\begin{array}{ccc}
O_n & U_{k,h} & u_{k,h}^t \\
V_{k,h} & O_n & v_{k,h}^t\\
-2v_{k,h} & -2u_{k,h} & 0 
\end{array}\right)\]
with $u_{k,h}, v_{k,h} \in V(n,\FF)$ and $U_{k,h}$, $V_{k,h}$ anti-symmetric matrices of order $n$. It will be useful to have a name for matrices of this form. We shall call them {\em block -anti-diagonal} matrices.  

Put $A_{k,h} := (e_{k,h}-e_{n+h,n+k})M_{k,h}-M_{k,h}(e_{k,h}-e_{n+h,n+k})$.  So,
\begin{equation}\label{ukh 0}
\delta_F(u_{k,h}(t))) ~ = ~   d_F(t)(e_{k,h}-e_{n+h,n+k}) + tA_{k.h}. 
\end{equation} 
As the reader can check, the matrix $A_{k,h}$ belongs to $\cL$ and is block-anti-diagonal. Moreover, if $a_{i,j}^{(k,h)}$ is the $(i,j)$-entry of $A_{k,h}$, then $a^{(k,h)}_{i,j} \neq 0$ only if $(i,j)$ satisfies one of the following: 
\begin{equation}\label{Akh}
\left\{\begin{array}{l}
i = k, ~ j > n, ~  j\not\in\{n+k, n+h\}; \\
j = h, ~ i > n, ~ i \not\in\{n+k,n+h\}; \\
i = n+h, ~  j \leq n, ~ j\not\in\{k, h\}; \\
j = n+k, ~ i \leq n, ~ i\not\in\{k,h\}; \\
i = n+h,~  j = 2n+1;\\
j = n+k, ~ i = 2n+1. 
\end{array}\right.
\end{equation} 

\begin{result}\label{uniform der1}
Let $\KK \in \mathfrak{F}(\FF)$. With $d_F$ as above and $d_{\KK,B}$ and $p_{\KK,B}$ as in Subsection \ref{sec 121}, we have $d_{\KK, B, F}(t) = p_{\KK,B}(d_F(t))$ for every $t\in \KK$.  
\end{result} 
\noindent
{\em Proof.} As $\KK\in \mathfrak{F}(\FF)$, the values that a derivation $d$ of $\FF$ takes on $\KK$ are $\KK$-linear combinations of the values that $d$ takes on a given derivation basis of $\KK$. (When $\ch(\FF) = 0$ this fact follows from the inclusion $N(\FF)\subseteq \KK$; when $\ch(\FF) = p > 0$ we obtain it from (\ref{N(F)}).) All derivation bases of $\KK$ are finite (of the same size). Therefore $\KK$ admits an extension $\KK'\in\mathfrak{F}(\FF)$ whih contains all values $d_F(t)$ as $t$ ranges in $\KK$. Let $B'$ be a basis of $\FF$ over $\KK'$ with $1\in B'$. With no loss, we can assume that $B \supseteq B'$.  

If $g = u_{k,h}(t)$ with $t\in \KK$, the summand $gm_{\KK',B'} - m_{\KK',B'}g$ in the right side of  (\ref{G(K)ter'}) (but with $\KK$ and $B$ replaced by $\KK'$ and $B'$) is equal to $tm$ for a suitable matrix $m \in L(\KK')$. Moreover, $\delta_F(u_{k,h}(t)) \in L(\KK')$, by the choice of $\KK'$. We can also assume to have chosen $\KK'$ as large as to contain all entries of $A_{k,h}$. So, up to replacing $\KK$ with $\KK'$ in (\ref{G(K)ter'}), we can compare (\ref{ukh 0}) and (\ref{G(K)ter'}), with $g = u_{k,h}(t)$ in (\ref{G(K)ter'}). We obtain that there is a scalar $a\in \KK'$ such that $d_F(t) - d_{\KK',B'}(t) = ta$ for every $t\in \KK$. In particular, $d_F(ts) - d_{\KK',B'}(ts) = tsa$ for every choice of $s, t \in \KK$. However
\[ \begin{array}{rl}
tsa & = ~  d_F(ts) - d_{\KK',B}(ts) ~ = ~   d_F(t)s + td_F(s) - d_{\KK',B}(t)s - td_{\KK',B}(s)  ~  =  \\
 & = ~  (d_F(t)-d_{\KK',B}(t))s + t(d_F(s) - d_{\KK',B}(s)) ~ = ~ tas + tsa. 
\end{array}\]
Therefore $tsa = 0$, namely $a = 0$. So, $d_{\KK',B'}(t) = d_F(t)$ for every $t\in \KK$. However we also have $p_{\KK,B}(d_{\KK',B'}(t))) = d_{\KK,B}(t)$. Hence $d_{\KK, B, F}(t) = p_{\KK,B}(d_F(t))$.  \hfill $\Box$\\

As a consequence of Result \ref{uniform der1},  the derivation $d_{\KK,B}$ does not depend on the particular choice of  $\KK$ and $B$, provided that $\KK$ is sufficiently large. Explicitly, 

\begin{corollary} \label{uniform der2}
For every $t\in \FF$, if $\KK\in\mathfrak{F}(\FF)$ contains both $t$ and $d_F(t)$, then $d_{\KK,B}(t) = d_F(t)$. 
\end{corollary} 
\noindent
{\em Proof.} We have $p_{\KK,B}(d_F(t)) = d_F(t)$ since, by assumption, $d_F(t)\in \KK$. The conclusion follows from Result \ref{uniform der1}. \hfill $\Box$ \\

Our next goal is proving the following.

\begin{lemma}\label{uni-cobound-0}
There exists a block-anti-diagonal matrix  $\ovm_F \in \cL$ such that for every choice of the pair $(k,h)$ with $k, h  \leq n$ and $k\neq h$ we have 
\begin{equation}\label{ukh 1}
A_{k,h} ~ = ~  (e_{k,h}-e_{n+h,n+k})\ovm_F-\ovm_F(e_{k,h}-e_{n+h,n+k}).
\end{equation}
Equivalently 
\[tA_{k,h} ~ = ~  u_{k,h}(t)\ovm_F-\ovm_Fu_{k,h}(t).\] 
\end{lemma} 
\noindent
{\em Proof.} Recall that $A_{k,h}\in L$ and the $(i,j)$ entry $a^{(k,h)}_{i,j}$ of $A_{k,h}$ is non-zero only if $(i,j)$ satisfies one the conditions recorded in (\ref{Akh}). 

Let $m_{i,j}$ the $(i,j)$-entry of the matrix $\ovm_F$ we are looking for. Condition (\ref{ukh 1}) forces $m_{h,j} = a^{(k,h)}_{k,j}$ and $m_{i,k} = -a^{(k,h)}_{i,h}$. Hence a matrix $\ovm_F$ with the required properties exists if and only if 
\begin{equation}\label{ukh 3}
\begin{array}{l}
a^{(k,h)}_{k,j} \mbox{ does not depend on } k, \mbox{ for every } j > n; \\
a^{(k,h)}_{i,h} \mbox{ does not depend on } h, \mbox{ for every } i > n.
\end{array}
\end{equation} 
We shall prove now that conditions (\ref{ukh 3}) are indeed satisfied. Recall that, for pairwise distinct indices $k, h, p  \in\{1, 2,..., n\}$ and arbtrary elements $t, s\in \FF$ we have $(u_{k,h}(t), u_{h,p}(s)) = u_{k,p}(st)$, where $(.,.)$ is the commutator (explicitly, $(g_1, g_2) := g_1g_2g_1^{-1}g_2^{-1}$). 
With the help of (\ref{delta}), (\ref{ukh 0}) and (\ref{Akh}), by straightforward but laborious computations we can see that the relation $(u_{k,h}(t), u_{h,p}(s)) = u_{k,p}(st)$ forces all the following. 
\begin{equation}\label{ukh 4}
\left\{\begin{array}{lrcl}
a) & a^{(k,p)}_{h,n+k} & = & a^{(k,h)}_{k,n+p};  \\
b) & a^{(k,p)}_{k,i} & = & a^{(h,p)}_{h,i},  ~\mbox{ for } i > n;\\
c) & a^{(k,p)}_{i,p} & = & a^{(k,h)}_{i,h}, ~ \mbox{ for } i > n;\\
d) & a^{(k,p)}_{n+p,h} & = & a^{(h,p)}_{n+k,p}. 
\end{array}\right. 
\end{equation} 
From (\ref{ukh 4})(b) we obtain that $a^{(k,p)}_{k,n+h} = 0$ (because $a^{(h,p)}_{h, n+h} = 0$ since $A_{h,p}\in \cL$).  Hence, permuting $h$ and $p$, we get $a^{(k,h)}_{k,n+p} = 0$. On the other hand $a^{(k,h)}_{k,n+k} = 0$, because $A_{k,h}\in \cL$, while $a^{(k,h)}_{k,n+h} = 0$ by (\ref{Akh}). Therefore $a^{(k,h)}_{k,j} = 0$ for every $j \in \{n+1,..., 2n\}$. Morever (\ref{ukh 4})(b) also implies that $a^{(k,p}_{k,2n+1} = a^{(h,p)}_{h,2n+1}$. From this equality, by permuting $h$ and $p$ we obtain that $a^{(k,h)}_{k,2n+1}$ does not depend on $k$. Therefore $a^{(k,h)}_{k,j}$ does not depend on $k$ as long as $j > n$, as claimed in the first part of (\ref{ukh 3}).  

From (\ref{ukh 4})(c) we obtain that $a^{(k,h)}_{n+p,h} = 0$ (because $a^{(k,p)}_{n+p,p} = 0$ since $A_{k,p}\in \cL$). We know that $a^{(k,h)}_{n+h,h} = 0$ (because $A_{k,h}\in \cL$) while $a^{(k,h)}_{n+k,h} = 0$ by (\ref{Akh}). Moreover $a^{k,h}_{2n+1,h} = a^{k,p}_{2n+1,p}$ by (\ref{ukh 4})(c). Hence $a^{(k,h)}_{i,h}$ does not depend on $i$ as long as $i > n$. Both conditions of (\ref{ukh 3}) are satisfied. Therefore a matrix $\ovm_F$ as required indeed exists.  \hfill $\Box$ 

\begin{result}\label{uni-cobound-1}
With $\ovm_F$ as in Lemma \ref{uni-cobound-0}, let $\KK\in \mathfrak{F}(\FF)$ contain all entries of $\ovm_F$. Let $m_{\KK,B}$ be as in {\rm (\ref{G(K)ter})} and {\rm (\ref{G(K)ter'})}. Then there exists a scalar $\lambda\in \KK$ such that $\ovm_F - m_{\KK,B} = \lambda I_\pm$ where 
\[I_\pm ~ : =  ~ \left(\begin{array}{rrr}
I_n & O_n & {\bf 0}^t \\
O_n & -I_n& {\bf 0}^t\\
{\bf 0} & {\bf 0} & 0
\end{array}\right)\]
\end{result}
\noindent
{\em Proof.} By Lemma \ref{uni-cobound-0}, we have $F(g) = g^{-1}d_F(g) + \ovm_F - g^{-1}\ovm_Fg$ for every $g = u_{k,h}(t)$. Equivalently, $\delta_F(g) = d_F(g) + \ovm_F - \ovm_Fg$. Since the mpping $g \mapsto g^{-1}d_F(g)$ is a cocycle, the same holds for every $g\in G_1$. Explicitly,
\begin{equation}\label{uniform eq1}
\begin{array}{rcll}
F(g) & = &  g^{-1}d_F(g) + \ovm_F - g^{-1}\ovm_Fg, & \forall g\in G_1;\\
\delta_F(g) & = & d_F(g) + g\ovm_F - \ovm_Fg, & \forall g\in G_1.
\end{array}
\end{equation} 
Let $\KK\in\mathfrak{F}(\FF)$ be such that $\ovm_F \in \cL(\KK)$. As noticed in the proof of Result \ref{uniform der1}, the subfield $\KK$ admits an extension $\KK'\in\mathfrak{F}(\FF)$ such that $d_F(t)\in \KK'$ for every $t \in \KK$. Put $G_1(\KK) := G_1\cap G(\KK)$. If $g\in G_1(\KK)$ then all terms in the right side of both equations of (\ref{uniform eq1}) are matrices with entries in $\KK'$. Accordingly, $p_{\KK',B'}(\delta_F(g)) = \delta_F(g)$, $d_{\KK',B'}(g) = d_F(g)$ and
$gm_{\KK',B'}-m_{\KK',B'}g = g\ovm_F - \ovm_Fg$ for every $g\in G_1(\KK)$. 

Let  $B$ be a basis of $\FF$ over $\KK$ containing $B'$. As $g\in G_1(\KK)$ and $\ovm_F\in \cL(\KK)$, the equality $gm_{\KK',B'}-m_{\KK',B'}g = g\ovm_F - \ovm_Fg$ implies that 
\[gm_{\KK',B'}-m_{\KK',B'}g ~ = ~ p_{\KK,B}(gm_{\KK',B'}-m_{\KK',B'}g).\]
However $p_{\KK,B}(gm_{\KK',B'}-m_{\KK',B'}g) = gm_{\KK,B}-m_{\KK,B}g$. Therefore 
\[g\ovm_F- \ovm_Fg ~ = ~ gm_{\KK,B} - m_{\KK, B}g\]
for every $g\in G_1(\KK)$. Equivalently, $\ovm_F - m_{\KK, B}$ is centralized by  every $g\in G_1(\KK)$. Therefore $\ovm_F - m_{\KK, B}$ is diagonal with two scalar blocks $\lambda I_n, \mu I_n$ of order $n$ and one block $\nu$ of order $1$. However $\ovm_F-m_{\KK,B}\in \cL(\KK)$, since both $\ovm_F$ and $m_{\KK,B}$ belong to $\cL(\KK)$. Therefore $\mu = -\lambda$ and $\nu = 0$. \hfill $\Box$\\

Let $\beta \in B^1(\cL,G)$ be the coboundary which maps $g\in G$ onto $\ovm_F-g^{-1}\ovm_Fg$. If we replace $F$ with $F' := F-\beta$ then $F'(g) = g^{-1}d_F(g)$ for every $g\in G_1$. So, we can assume that $F(g) = g^{-1}d_F(g)$ (equivalenly $\delta_F(g) = d_F(g)$) for every $g\in G_1$. In particular
\[F(u_{k,h}(t)) ~ = ~ \delta_F(u_{k,h}(t)) ~ = ~ d_F(t)(e_{k,h}- e_{n+h,n+k})\]
for every choice of $k, h\in \{1, 2,..., n\}$ with $k \neq h$. 

\subsection{End of the proof of Claim \ref{main claim odd}} 

Corollary \ref{uniform der2} and Result \ref{uni-cobound-1} show that in general, with $F$ chosen as above, for every $g\in G$ and for a sufficiently large subfield $\KK \in \mathfrak{F}(\FF)$  owe have
\begin{equation}\label{new1}
p_{\KK,B} (F(g)) ~ = ~  g^{-1}d_F(g) + \lambda_\KK(I_\pm - g^{-1}I_\pm g)
\end{equation} 
for every basis $B$ of $\FF$ over $\KK$ such that $1\in B$ and a suitable scalar $\lambda_\KK\in \KK$ which does not depend on $g$ but might possibly depend on the particular choice of $\KK$. However for every given $g\in G$ we can always choose $\KK$ so large that $p_{\KK,B} (F(g)) = F(g)$ and $p_{\KK,B}(\delta_F(g)) = \delta_F(g)$. With $\KK$ chosen in this way, equation (\ref{new1}) can we rewritten as follows:   
\begin{equation}\label{new2}
F(g) ~ = ~  g^{-1}d_F(g) + \lambda_\KK(I_\pm - g^{-1}I_\pm g). 
\end{equation} 
This equation does not change if we enlarge $\KK$ further and shows that, at this stage, $\lambda := \lambda_\KK$ no more depends on $\KK$. Accordingly, equation (\ref{new2}) can be rewritten as follows, preventing subfields of $\FF$ from interfering in it.
\[F(g) ~ = ~  g^{-1}d_F(g) + \lambda (I_\pm - g^{-1}I_\pm g).\] 
Ler $\beta'$ be the coboundary which maps $g\in G$ onto $\lambda(I_\pm - g^{-1}I_\pm g)$. Note that $\beta'(g) = 0$ if $g\in G_1$. So, if we replace $F$ with $F' := F-\beta'$, the values that $F$ takes on $G_1$ remain unchanged. So, $F'(g) = g^{-1}d_F(g)$ for every $g\in G$. 
The proof of Claim \ref{main claim odd} is complete. The proof of Theorem \ref{Cohom odd} is also complete. 

\section{Proof of Theorem \ref{Cohomology} when $\ch(\FF) = 2$}\label{Cohomology Th even}

Throughout this section we assume that $\ch(\FF) = 2$. In this case $H^1(\cL, G) \cong Z(\cL)\oplus \Der(\FF)$ (V\"{o}lklein \cite{V1}). However this information is useless for us since now $\cL^*\not\cong \cL$. So, in contrast with Section \ref{Cohomology Th odd}, in this section we will make no use of the main result of  \cite{V1}. 

Henceforth when $X_n = B_n$ we put $\chi = h_n$ and when $X_n = D_n$ we put $\chi .= h_{n-1}-h_n$. So, $\chi$ spans the center or the deep center of $\cL$ according to whether $X_n = B_n$ or $X_n = D_n$. We denote by $F_0$ the seed cocycle of $G$ over $\cL_0^*$ (Definition \ref{def F0}).

\begin{lemma}\label{nuovo even lem1}
For every $F\in Z^1(\cL^*,G)$ we have $F(g)(\chi) = 0$ for every $g\in G$.
\end{lemma}
{\em Proof.} Let  $F\in Z^1(\cL^*,G)$. As $\chi\cdot g = \chi$ for every $g\in G$, for any choice of $g_1, g_2\in G$ we have $F(g_1g_2)(\chi) = F(g_1)(\chi) + F(g_2)(\chi)$. Hence $F(g)(\chi) = 0$ for every $g\in G$, since $G$ is perfect.  \hfill $\Box$\\

For $f\in \cL^*$ let $f_{|0} \in \cL_0^*$ be the restriction of $f$ to $\cL_0$ and let $\rho$ be the function that sends every $F\in Z^1(\cL^*,G)$ to the function that maps every $g\in G$ onto $F(g)_{|0}$. 

\begin{corollary}\label{nuovo even lem2}
The mapping $\rho$ is an isomorphism from $Z^1(\cL^*, G)$ to $Z^1(\cL^*_0, G)$. Morever $\rho$ maps $B^1(\cL^*, G)$ onto the sum of $B^1(\cL^*_0,G)$ and the $1$-subspace $\langle F_0\rangle$ of $Z^1(\cL^*_0,G)$ spanned by the seed cocycle $F_0$ of $G$ over $\cL_0^*$.  
\end{corollary} 
{\em Proof.} Given $F\in Z^1(\cL^*,G)$, $g_1, g_2\in G$ and $x\in \cL_0$, we have 
\[F(g_1g_2)(x) ~ = ~ F(g_1)(x\cdot g_2^{-1}) + F(g_)(x).\] 
Recall that $x\cdot g_2^{-1} = g_2xg_2^{-1} + F_0(g_2)(x)\chi$. Hence 
\[F(g_1)(x\cdot g_2^{-1}) ~ = ~ F(g_1)(g_2xg_2^{-1}+ F_0(g_2)(x)\cdot F(g_1)(\chi).\]
However $F(g_1)(\chi) = 0$ by Lemma \ref{nuovo even lem1}. Therefore
\begin{equation}\label{nuovo even eq1}
F(g_1g_2)(x) ~ = ~ F(g_1)(g_2xg_2^{-1}) + F(g_2)(x).
\end{equation}
For every $x\in \cL_0$ and $f\in \cL^*$ we have $f(x) = f_{|0}(x)$ and $f(gxg^{-1}) = (f_{|0}\cdot g)(x)$ for every $g\in G$. So, (\ref{nuovo even eq1}) shows that $\rho(F) \in Z^1(\cL^*_0,G)$. It is now clear that $\rho$ is indeed an isomorphism from $Z^1(\cL^*, G)$ to $Z^1(\cL^*_0,G)$. 

Turning to second claim of the corollary, let $f\in \cL^*$ and $g\in G$. Given $x\in \cL$, let $x_0 = \pi_0(x)$ and $t\chi = x-x_0$. Then $f(x) = tf(\chi) + f(x_0)$. Also, $x\cdot g^{-1} = t\chi + F_0(g)(x_0)\chi + gx_0g^{-1}$, hence $(f\cdot g)(x) = f(x\cdot g^{-1}) = tf(\chi) +  F_0(g)(x_0)f(\chi) + f(gx_0g^{-1}$. Consequently, 
\[(f-f\cdot g)(x) ~ = ~ F_0(g)(x_0)f(\chi) + f(x_0) - f(gx_0g^{-1}) ~ = ~ (F_0(g)f(\chi) + f_{|0}- f_{|0}\cdot g)(x_0).\]
This shows that $\rho$ maps the coboundary $\delta f \in B^1(\cL^*,G)$, which maps $g\in G$ onto $f-f\cdot g$, onto the sum of the coboundary $\delta f_{|0} \in B^1(\cL_0^*,G)$
and a multiple $f(\chi)F_0$ of the seed cocycle. In particular, $F_0 = \rho(\delta\chi^*)$ where $\chi^*\in \cL^*$ maps $\chi$ onto 1 and $\cL_0$ onto $\{0\}$.   \hfill $\Box$ \\

In view of the above, henceforth we take the liberty of regarding $Z^1(\cL^*, G)$ as the same as $Z^1(\cL^*_0,G)$ and $B^1(\cL^*,G)$ as the same as $B^1(\cL^*_0,G)+\langle F_0\rangle$. Accordingly, $H^1(\cL^*,G) = H^1(\cL^*_0,G)/\langle [F_0]\rangle$ where $[F_0] := F_0+ B^1(\cL^*_0,G)$. In this setting  the statement of Theorem \ref{Cohomology} is equivalent to the following:

\begin{theorem}\label{nuovo even theo}
When $|\FF| > 2$ the function $\gamma$ maps $\Der(\FF)$ onto a complement of $B^1(\cL^*_0,G)+\langle F_0\rangle$ in $Z^1(\cL^*_0,G)$.
\end{theorem} 
  
Before going on we highlight some properties of the form $\tau_\cM$. Recall that $\tau_\cM$ is defined on $\cM_\cL\supsetneq\cL_0$. When $X_n = D_n$ the form $\tau_\cM$ is non-degenerate, hence it defines a scalar product on $\cM_\cL$. In contrast, when $X_n = B_n$ the form $\tau_\cM$ admits a $2n$-dimensional radical $\mathrm{Rad}(\tau_\cM)$ contained in $\cL_0$ and formed by the matrices as in (\ref{matrix B})  with $X = Y = Z = O_n$. In either case, let $\perp$ be the orthogonality relation with respect to $\tau_\cM$.  

\begin{lemma}\label{L0perp}
Let $X_n = D_n$. Then $\cL_0^\perp = \langle I\rangle$ and $\cL^*_0\cong \cM_\cL/\langle I\rangle$ (isomorphism of $G$-modules). Moreover, $\cL^*_0\cong \cL_0$ if and only if $n$ is odd.  
\end{lemma} 
\noindent
{\em Proof.} The equality $\cL_0^\perp = \langle I\rangle$ is straightforward. Hence $\cL^*_0\cong L'/\langle I\rangle$ as vector spaces, every $f\in \cL^*_0$ being represented by a matrix $m_f\in \cM_\cL$ (unique modulo $\lambda I \rangle$) such that $\tau_\cM(m_f,x) = f(x)$ for every $x\in \cL_0$. For $g\in G$ we fave $(f\cdot g)(x) = f(gxg^{-1})$ for evey $x\in \cL_0$, hence $\tau_\cM(m_{f\cdot g}, x) = \tau_\cM(m_f, gxg^{-1})$ for every $x\in \cL_0$. However $G$ preserves $\tau_\cM$ (Lemma \ref{GT}).  Hence $\tau_\cM(m_f, gxg^{-1}) = \tau_\cM(g^{-1}m_fg, x)$. Therefore $m_{f\cdot g} \equiv g^{-1}m_fg$ ($\mathrm{mod}~\langle I\rangle$). So, $\cL^*_0$ and $\cM_\cL/\langle I\rangle$ are isomorphic as right $G$-modules, the $G$-module structure of $\cM_\cL/\langle I\rangle$ being the one inherited from the given action of $G$ on $\cM_\cL$.  

The last claim of the lemma follows from the fact that $\cL_0$ contains no $G$-invariant hyperlane and $I\in\cL_0$ if and only if $n$ is even. \hfill $\Box$

\begin{remark}\label{L0perp rem}
\em
Let $X_n = B_n$. Then $\tau_\cM$ is degenerate and $\mathrm{Rad}(\tau_\cM)$ is the $2n$-dimensional subspace of $\cM_\cL$ formed by the matrices with all null columns but possibly the last one.  Moreover $\cL_0^\perp = \langle I_D\rangle+\mathrm{Rad}(\tau_\cM)$ with $I_D$ as in (\ref{ID}) (see Subsection \ref{L*0 sec}). If $n$ is odd then $\cL_0^\perp\cap \cL_0 = \mathrm{Rad}(\tau_\cM)$, otherwise $\cL_0 \supseteq \cL_0^\perp$. So, $\tau_\cM$ is of no help for a nice description of $\cL^*_0$. We shall come back to this issue in Subsection \ref{L*0 sec}.     
\end{remark}    

Turning back to the mapping $\gamma:\Der(\FF)\longrightarrow Z^1(\cL_0^*,G)$, now we can prove the following. 

\begin{lemma}\label{Gamma2}
The image $\gamma(\Der(\FF))$ of $\gamma$ meets $\langle F_0\rangle + B^1(\cL_0^*,G)$ trivially. 
\end{lemma}
\noindent
{\em Proof.} Let $d\in\Der(\FF)$ be such that  $\gamma_d \in B^1(\cL_0^*,G)+\langle F_0\rangle$. So, there exists a linear functional $f\in \cL_0^*$ an a scalar $k\in \FF$ such that $\gamma_d(g) = f - f\cdot g + kF_0(g)$ for every $g\in G$. Explicitly,
\begin{equation}\label{nuovo even 1}
\tau_\cM(g^{-1}d(g), x) ~=~ f(x)-f(gxg^{-1}) + kF_0(g)(x), ~~~ \forall g\in G, ~\forall x\in \cL_0.
\end{equation} 
Let $X_n = D_n$. By Lemma \ref{L0perp} there exists a matrix $m\in\cM_\cL$ such that $f(x) = \tau_\cM(m,x)$ for every $x\in \cL_0$. So, equation (\ref{nuovo even 1}) can be rewritten as follows    
\begin{equation}\label{Gamma2 eq1}
\tau_\cM(g^{-1}d(g), x) ~=~ \tau_\cM(m, x - gxg^{-1}) + kF_0(g)(x), ~~~ \forall g\in G, ~\forall x\in \cL_0.
\end{equation}
Let $g = u_{k,h}(t)$. Then $F_0(g) = 0$ by Lemma \ref{lemma F0 D} and (\ref{Gamma2 eq1}) boils down to the following:
\begin{equation}\label{Gamma2 eq1 bis}
\tau_\cM(g^{-1}d(g), x) ~=~ \tau_\cM(m, x - gxg^{-1}), ~~~ \forall g\in G, ~\forall x\in \cL_0.
\end{equation} 
With $g$ as above, $x = e_{h,k}-e_{n+k,n+h}$ and $g$ and $(m_{i,j})_{i, j = 1}^{2n}=m$, equation (\ref{Gamma2 eq1 bis}) yields 
\begin{equation}\label{Gamma2 eq1 ter}
d(t) ~ = ~ t(m_{h,h}-m_{k,k}) + t^2m_{h,k}, ~~~ \forall t\in \FF.
\end{equation}
With $t = 1$ from (\ref{Gamma2 eq1 ter}) we obtain that $m_{2,2}-m_{1,1}+ m_{2,1} = 0$. Thus (\ref{Gamma2 eq1 ter}) implies 
\begin{equation}\label{Gamma2 eq1 quater}
d(t) ~ = ~ (t-t^2)(m_{h,h}-m_{k,k}), ~~~ \forall t\in \FF.
\end{equation}
If $\FF = \FF_2$ then $\Der(\FF) = \{0\}$ and there is nothing to prove. Otherwise, $\FF^2 \neq \FF_2$. For $t\in \FF^2\setminus \FF_2$ we have $d(t) = 0$ and $t-t^2 \neq 0$. With $t$ chosen in this way (\ref{Gamma2 eq1 quater}) forces $m_{h,h} = m_{k,k}$. Equation (\ref{Gamma2 eq1 quater}) now implies that $d(t) = 0$ for every $t\in\FF$, namely $d = 0$. This completes the proof when $X_n = D_n$. 

Let now $X_n = B_n$. Then $\cL_0$ contains a subalgebra $\cL_{D,0}$ isomorphic to the quotient of the Lie algebra of type $D_n$ over its deep center (just put $x = y = 0$ in the matrices described in (\ref{matrix B}); see also Subsection \ref{DB sec} of these notes). In turn the group $G = \Omega(2n+1,\FF)$ contains a subgroup $G_D$ which stabilizes $\cL_{D,0}$ and acts as $\Omega^+(2n,\FF)$ on it. Let $\gamma_{D,d}$ be the restriction of $\gamma_d$ to $G_D$, with the linear functionals $\gamma_d(g)$ (for $g\in G_D$) replaced by their restrictions to $\cL_{D,0}$, and let $f_D$ be the restricton of $f$ to $\cL_{D,0}$. Thus, we are driven back to the $D_n$ case and we can conclude that $d = 0$ \hfill $\Box$ \\

In view of Lemma \ref{Gamma2}, in order to prove Theorem \ref{nuovo even theo} we only need to prove the following.

\begin{theorem}\label{even main1}
If $\FF \neq \FF_2$ then  $Z^1(\cL^*_0,G) = \gamma(\Der(\FF))+B^1(\cL^*_0,G) + \langle F_0\rangle$.
\end{theorem}  

\subsection{Preliminaries}

In this subsection we shall fix a couple of issues which will be relevant for our proof of Theorem \ref{even main1}. 

\subsubsection{The inclusions $\mathfrak{o}(2n,\FF) \subseteq \mathfrak{o}(2n+1,\FF)$ and $\Omega^+(2n,\FF) \subseteq \Omega(2n+1,\FF)$}\label{DB sec} 

Throughout this subsection it is convenient to have two distinct symbols for the Lie algebra of type $D_n$ and that of type $B_n$ (over the same field $\FF$). We call them $\cL_D$ and $\cL_B$ respectively. Accordingly, we put $G_D = \Omega^+(2n,\FF)$ ad $G_B = \Omega(2n+1,\FF)$. With this notation, $\cL_D$ is a subalgebra of $\cL_B$ and $G_D$ is the commutator subgroup of the stabilizer of $\cL_D$ in $G_B$. 

Explicily, let $\Phi$ and $\Phi'$ be the root systems of type $B_n$ ad $D_n$ respectively. Then $\Phi'\subset\Phi$ and, given a basis $\Pi = \{\alpha_1,..., \alpha_{n-1}, \alpha_n\}$ of $\Pi$, put $\alpha'_n := \alpha_{n-1}+2\alpha_n$. The set $\Pi' = \{\alpha_1,..., \alpha_{n-1}, \alpha'_n\}$ is a basis of $\Phi'$ and $\cL_D$ is the subalgebra of $\cL_B$ generated by $\{x_\alpha\}_{\alpha\in \Pi'\cup -\Pi'}$. As a subspace, $\cL_D$ is the subspace of $\cL_B$ generated by
\[(\bigcup_{1\leq k < h \leq n}\{x_{\alpha^+_{(k,h)}}, x_{\alpha^-_{(k,h)}}, x_{\alpha^+_{[k,h]}}, x_{\alpha^-_{[k,h]}}\}) \cup\{h_1,..., h_{n-1}, h'_n\}\]
where $h'_n := h_{\alpha'_n} = [\alpha'_n,-\alpha'_n] = [x_{\alpha_{n-1}+2\alpha_n}, x_{-\alpha_{n-1}-2\alpha_n}]$. So,
\begin{equation}\label{DB K}
h_{n-1} - h'_n ~ = ~ h_n.
\end{equation}
The subgroup $G_D$ of $G_B$ is generated by $(\cup_{k\neq h}U_{k,h})\cup(\cup_{k<h}(U^+_{k,h}\cup U^-_{k,h}))$ where
\begin{equation}\label{Usbgr}
U_{k,h} := \{u_{k,h}(t)\}_{t\in \FF}, ~ U^+_{k,h}\ :=\{ u^+_{k,h}(t)\}_{t\in \FF}, ~  U^-_{k,h} = \{u^-_{k,h}(t)\}_{t\in \FF}.
\end{equation} 
The above holds independently of $\ch(\FF)$. If $\ch(\FF) = 2$, as we are assuming this Section, then $h_n$ generates the center $Z(\cL_B)$ of $\cL_B$ and $h_{n-1} - h'_n$ generates the deep center $K(\cL_D) \subseteq Z(\cL_D)$ of $\cL_D$. However $h_n = h_{n-1}-h'_n$ by (\ref{DB K}). Therefore:

\begin{lemma}\label{DB K lem}
When $\ch(\FF) = 2$ the deep center of $\cL_D$ is equal to the center of $\cL_B$. 
\end{lemma}

Keeping the assumption that $\ch(\FF) = 2$, we put $\chi := h_n$ anf $K := \langle \chi\rangle =  Z(\cL_B) = K(\cL_D)$.  We call $\chi$ the {\em canonical generator} of $K$. 

Put $h_{\alpha^+_{\{k\}}} = [x_{\alpha^+_{\{k\}}}, x_{\alpha^-_{\{k\}}}]$, $h_{\alpha^+_{(k,h)}} = [x_{\alpha^+_{(k,h)}}, x_{\alpha^-_{(k,h)}}]$ and    
$h_{\alpha^+_{[k,h]}} = [x_{\alpha^+_{[k,h]}}, x_{\alpha^-_{[k,h]}}]$. In particular, $h_{\alpha^+_{(k,k+1)}} = h_k$ for every $k < n$,  $h_{\alpha^+_{[n-1, n]}} = h'_n$ and $h_{\alpha^+_{\{n\}}} = h_n$. Note that for $1\leq k < h \leq n$ we have
\[\begin{array}{rcccccccl}
& h_{\alpha^+_{\{k\}}} & = &  -h_{\alpha^+_{[k,h]}} -  h_{\alpha^+_{(k,h)}} & = & h_{\alpha^+_{\{n\}}} + 2h_{\alpha^+_{(k,n)}} & = \\
 = & h_{\alpha^+_{\{h\}}} & = & h_{\alpha^+_{(k,h)}} -  h_{\alpha^+_{[k,h]}} & = & h_{\alpha^+_{\{k\}}} + 2h_{\alpha^+_{(k,h)}}.& 
\end{array}\]
Therefore, since by assumption $\ch(\FF) = 2$, we obtain
\begin{equation}\label{DB K bis}
\chi ~ = ~ \left\{\begin{array}{cll}
h_{\alpha^+_{\{k\}}} &\mbox{for } k = 1, 2,..., n, & (\mbox{in } \cL_B); \\
h_{\alpha^+_{(k,h)}}+  h_{\alpha^+_{[k,h]}} & \mbox{for } 1\leq k < h\leq n, & (\mbox{both in } \cL_B \mbox{ and in } \cL_D).
\end{array}\right.
\end{equation}
The subspace $\cL_{B,0}$ of $\cL_B$ generated by $\{x_\alpha\}_{\alpha\in\Phi}\cup\{h_1,..., h_{n-1}\}$ (compare (\ref{rootElm B})) will be called the {\em canonical complement} of $K$ in $\cL_B$. This complement  contains the subspace $\cL_{D,0}$ of $\cL_D$ generated by $\{x_\alpha\}_{\alpha\in\Phi'}\cup\{h_1,..., h_{n-1}\}$, which is a complement of $K$ in $\cL_D$. We call it the {\em canonical complement} of $K$ in $\cL_D$. 

Let $F_0$ and $F'_0$ be the seed cocycles of $G_B$ over the dual $\cL^*_{B,0}$ of $\cL_{B,0}$ and respectively of $G_D$ over the dual $\cL^*_{D,0}$ of $\cL_{D,0}$. Lemmas \ref{lemma F0 B} and \ref{lemma F0 D} imply the following. 

\begin{lemma}\label{DB K cor}
The seed cocycle $F_0$, regarded as a mapping from $G_B\times \cL_{B,0}$ to $\FF$, induces the seed cocycle $F'_0$ on the subset $G_D\times \cL_{D,0}$ of $G_B\times \cL_{B,0}$. 
\end{lemma} 

\subsubsection{A description of $\cL^*_0$ as a quotient of $\cM_\cL$ when $X_n = B_n$ and $\ch(\FF) = 2$}\label{L*0 sec} 

Let $\ch(\FF) = 2$. As stated in Lemma \ref{L0perp}, when $X_n  = D_n$ an isomorphism between $\cL^*_0$ and $\cM_\cL/\langle I\rangle$ can be established by means of the form $\tau_\cM$. In contrast, as noticed in Remark \ref{L0perp rem}, when $X_n = B_n$ the form $\tau_\cM$ is not the right tool to get a nice description of $\cL^*_0$. However the following modification of $\tau_\cM$ does the job. For every choice of two matrices 
\[\begin{array}{ccl}
m ~ = ~ \left(\begin{array}{ccl}
A & B & a^t\\
C & A^t & b^t\\
{\bf 0} & {\bf 0} & 0 
\end{array}\right), \hspace{5 mm} 
x ~ = ~ \left(\begin{array}{ccl}
X & Y & x^t \\
Z & X^t  & y^t \\
{\bf 0}  & {\bf 0} & 0
\end{array}\right).
\end{array}\]  
of $\cM_\cL$ we put 
\begin{equation}\label{hat tau}
\widehat{\tau}_\cM(m,x) ~ := ~\tau_\cM(m,x) + ay^t + bx^t.
\end{equation} 
The form $\widehat{\tau}_\cM:\cM_\cL\times\cM_\cL\rightarrow\FF$ defined in this way is bilinear, symmetric and non-degenerate. Hence it yields a scalar product on $\cM_\cL$. If $\perp$ is the orthogonality relation with respect to $\widehat{\tau}_\cM$ then $\cL_0^\perp = \langle I_D \rangle$, where
\begin{equation}\label{ID}
I_D ~ := ~ \left(\begin{array}{ccl}
I_n & O_n & {\bf 0}^t \\
O_n & I_n  & {\bf 0}^t \\ 
{\bf 0} & {\bf 0} & 0
\end{array}\right).
\end{equation} 
Accordingly, $\cL^*_0 \cong \cM_\cL/\langle I_D \rangle$. However $G$, in its action on $\cM_\cL'$, does not centralize $I_D$. Consequenly, the action of $G$ on  $\cM_\cL/\langle I_D \rangle$ that agrees via $\widehat{\tau}_\cM$ with the natural action of $G$ on $\cL^*_0$ does not arise from the action of $G$ on $\cM_\cL$. In other words, $G$ does not preserve $\widehat{\tau}_\cM$. Explicitly, the subgroup $G_D \cong \Omega^+(2n,\FF)$ of $G$ ($= G_B$ in the notation of Section \ref{DB sec}) does preserve $\widehat{\tau}_\cM$ but the elements of $G\setminus G_D$ do not. 

Let ${\ast}$ be the action of $G$ on $\cM_\cL$ that agrees via $\widehat{\tau}_\cM$ with the natural action of $G$ on ${\cM_\cL}^*$. As $G$ stabilizes $\cL_0$ and $\langle I_D\rangle = \cL_0^\perp$ (with respect to $\widehat{\tau}_\cM$), the group $G$, acting on $\cM_\cL$ according to ${\ast}$,  fixes $I_D$. Hence ${\ast}$ induces an action of $G$ on $\cM_\cL/\langle I_D\rangle$. This is indeed the action we need. We shall now describe it to the extent we can.  For every $m\in \cM_\cL$ and $g\in G$ we put 
\begin{equation}\label{Delta}
\Delta(m,g) ~ := ~   m{\ast}g -  g^{-1}mg.
\end{equation} 
So, describing ${\ast}$ is the same as describing the mapping $\Delta:\cM_\cL\times G \rightarrow \cM_\cL$. It is straightforward to check that for every $m\in \cM_\cL$ and every choice of $g_1, g_2\in G$ we have: 
\begin{equation}\label{L*01}
\Delta(m,g_1g_2) ~ = ~  g_2^{-1}\Delta(m,g_1)g_2 + \Delta(g_1^{-1}mg_1, g_2) + \Delta(\Delta(m, g_1), g_2). 
\end{equation}  
Put $U_{k}^+ := \{u^+_k(t)\}_{t\in \FF}$ and $U_{k}^- := \{u^-_{k}(t)\}_{t\in \FF}$ and let $R^+$ and $R^-$ be the subgroups of $G$ generated by $\cup_{k=1}^n U^+_k$ and respectively $\cup_{k=1}^n U^-_k$. We have $R^+\cap R^- = \{1\}$, both $R^+$ and $R^-$ meet $G_D$ trivially and both of them are isomorphic to the additive group of $V(n,\FF)$. The group $G$ is generated by $G_D\cup R^+\cup R^-$. So, if we know what $\Delta(m,g)$ is for $g\in G_D\cup R^+\cup R^-$ then, with the help of (\ref{L*01}) we virtually know $\Delta(m,g)$ for every $g\in G$. 

\begin{lemma}\label{L*0 lem}
For every $m\in\cM_\cL$ we have $\Delta(m,g) = O$ (null matrix) for every $g\in G_D$. For every $g\in R^+\cup R^-$ and every $m\in \cM_\cL$ we have
\begin{equation}\label{L*02}
\Delta(m,g) ~ = ~  (g^{-1}m\delta(g)^t+(g^{-1}m\delta(g)^t)^t)J + g^{-1}m\delta(g),
\end{equation}  
wih $J$ as in {\rm (\ref{J B})} and $\delta(g)$ is the matrix all columns of which but the last one are null, the last column of $\delta(g)$ being equal to the last column of $g-I$. 
\end{lemma}
\noindent
{\em Proof.} If $g\in G_D$ then $\Delta(m,g)= O$. Indeed $G_D$, acting on $\cM_\cL$ in the natural way, preserves the form $\widehat{\tau}_\cM$. Hence $m{\ast}g = g^{-1}mg$ for every $g\in G_D$. 

Turning to $R^+\cup R^-$, let $g\in R^+$, to fix ideas. Then $g$ is represented by the following matrix where $(t_1, t_n,..., t_n) \in V(n,\FF)$ and $\mathrm{diag}(t_1^2,...,t_n^2)$ is the diagonal matrix with diagonal entries equal to $t^2_1, t^2_2,...,t^2_ n$:
\[\left(\begin{array}{ccc}
I_n & O_n & {\bf 0}^t  \\
\mathrm{diag}(t_1^2,...,t_n^2)& I_n & (t_1,t_2,..., t_n)^t \\
{\bf 0} & {\bf 0} & 1 
\end{array}\right).\]
With $g$ as above we can compute $\Delta(m,g)$ by means of (\ref{L*02}) for an arbitrary matrix $m\in \cM_\cL$. Elementary but laborious computations show that, with the value thus obtained for $\Delta(m,g)$ and $m{\ast} g = g^{-1}mg + \Delta(m,g)$, indeed we have $\widehat{\tau}_\cM(m, g^{-1}xg) = \widehat{\tau}_\cM(m{\ast}g, x)$ for every $x\in \cM_\cL$. So, as long as $g \in R^+\cup R^-$, if $\Delta(m,g)$ is as in (\ref{L*02}) then the action $g:m\rightarrow m{\ast}g = \Delta(m,g)+g^{-1}mg$ of $g$ on $\cM_\cL$ agrees with the natural action of $g$ on $\cM_\cL^*$.      \hfill $\Box$ 

\begin{remark}
\em
Formula (\ref{L*02}) is tivially true when $g\in G_D$. Indeed if $g\in G_D$ then $\delta(g) = O = \Delta(m,g)$. I don't know if (\ref{L*02}) also holds true for $g\in G\setminus(G_D\cup R^+\cup R^-)$. I am afraid it doesn't. Anyway, this issue isn't relevant to our proof of Theorem \ref{even main1}. 
\end{remark}     
  
\subsection{Proof of Theorem \ref{even main1}; the $D_n$-case}\label{main D} 

Throughout this subsection $X_n = D_n$ and $\ch(\FF) = 2$. Moreover $K$ is the deep center of $\cL$ and $\cL_0 \cong \cL/K$ is the canonical complement of $K$ in $\cL$ as described in Subsection \ref{L0 D} and regarded as a subalgebra of $\cM_\cL$. Also, $\chi = h_{n-1}-h_n$ is the canonical generator of $K$ and $F_0$ is the seed cocycle of $G$ over $\cL^*_0$. 

We denote by $ G^0, G^+$ and $G^-$ the subgroups of $G$ generated by $\{u_{k,h}(t)~|~ k\neq h, ~ t\in \FF\}$,  $\{u_{k,h}^+(t)~|~ k<h, ~ t\in \FF\}$ and $\{u_{k,h}^-(t)~|~ k<h, ~ t\in \FF\}$ respectively. We have $ G^0 \cong \mathrm{SL}(n,\FF)$ while both $G^+$ and $G^-$ are isomorphic to the additive group of $V({n\choose 2}, \FF)$. The union $G^0\cup G^+\cup G^-$ generates $G$. Therefore, if for a given cocycle $F$ of $G$ over $\cL^*_0$ we have $F(g) = 0$ for every $g\in G^0\cup G^+\cup G^-$, then $F =  0$. We shall prove the following.

\begin{proposition}\label{Main D}
For every cocycle $F\in Z^1(\cL^*_0,G)$ there exist a unique derivation $d_F\in \Der(\FF)$, a unique coboundary $F'\in B^1(\cL^*_0, G)$ and a unique scalar $\lambda_F\in \FF$ such that 
\[(F-\gamma_{d_F}-F'-\lambda_F F_0)(g) = 0, ~~~ \forall g\in G^0\cup G^+\cup G^-.\]
\end{proposition}

Proposition \ref{Main D} yields the conclusion of Theorem \ref{even main1} in the $D_n$-case. 

Recall that, by convention, $n > 3$ in $D_n$. However the Lie algebra $\cL = \mathfrak{o}(2n,\FF)$ and the group $G = \Omega^+(2n,\FF)$ are defined also for $n = 3$. So, in the sequel, initially we allow $n = 3$. The hypothesis $n >3$ will be reintroduced in due course, when we need it.  

\subsubsection{$F_0(g)$ for $g\in G^0\cup G^+\cup G^-$}\label{F0 sec}

As $\ch(\FF) = 2$, all matrices $g\in G^0\cup G^+\cup G^-$ belong to $\cM_\cL$. So, given $g\in G^0\cup G^+\cup G^-$ and $x\in \cL_0$,  we can consider the scalar $\tau_\cM(g, x)$. Note also that $\tau_\cM(I, x) = 0$ for every $x\in \cL_0$. Hence $\tau_\cM(m, x) = \tau_\cM(m-I, x)$ for every $m\in \cM_\cL$ and every $x\in \cL_0$.  

\begin{lemma}\label{F0}  
If $g\in G^0$ then $F_0(g) = 0$ while if $g\in G^+\cup G^-$ then $F_0(g)(x) = \tau_\cM(g, x)$ for every $x\in\cL_0$. 
\end{lemma}
\noindent
{\em Proof.} We know from Lemma \ref{lemma F0 D} that $F_0(u_{k,h}(t)) = 0$ for every choice of $k\neq h$ in $\{1, 2,..., n\}$. Hence $F_0(g) = 0$ for every $g\in G^0$ by the cocycle condition on $F_0$. It is not difficult to see that claim (2) of Lemma \ref{lemma F0 D} implies that if $g = u^\pm_{k,h}(t)$ then
\begin{equation}\label{F0 G+-1}
F_0(g)(x) ~ = ~ \tau_\cM(g, x), ~~~ \forall x\in \cL_0.
\end{equation} 
Let now $g_1, g_2\in G^\pm$ be such that both $g_1$ and $g_2$ satisfy (\ref{F0 G+-1}). Then for every $x\in \cL_0$ we have 
\[F_0(g_1g_2)(x) = F_0(g_1)(g_2xg_2^{-1}) + F_0(g_2)(x) ~ = ~ \tau_\cM(g_1, g_2xg_2^{-1}) + \tau_\cM(g_2, x).\]
However $\tau_\cM(g_1, g_2xg_2^{-1}) = \tau_\cM(g_2^{-1}g_1g_2, x) = \tau_\cM(g_1, x)$ by Lemma \ref{GT} and because $G^\pm$ is abelian. Therefore 
\[F_0(g_1g_2) = \tau_\cM(g_1, x) + \tau_\cM(g_2, x) = \tau_\cM(g_1+g_2, x) = \tau_\cM(g_1g_2 - I, x) = \tau_\cM(g_1g_2, x)\]
because $g_1g_2 = g_1+g_2-I$. So, $g_1g_2$ also satisfies (\ref{F0 G+-1}). The conclusion follows.  \hfill $\Box$ 

\subsubsection{$F(g)$ for $g\in G^0$} 

The vector space $\cL_0$ is the direct sum of three subspaces $\cL_0^0$, $\cL_0^+$ and $\cL^-$ with bases as follows:
\[\begin{array}{rl}
\mbox{subspace} & \mbox{basis} \\
\cL_0^0~:& \{e_{i,j}-e_{n+j,n+i}\}_{i,j \in \{1,..., n\}, ~ i\neq j}^n\cup\{e_{i,i}-e_{n.n}-e_{n+i,n+i}+e_{2n,2n}\}_{i=1}^{n-1}; \\
\cL_0^+~: & \{e_{j+n,i}-e_{i+n,j}\}_{1\leq i<j\leq n}; \\
\cL_0^-~: & \{e_{i,j+n}-e_{j,i+n}\}_{1\leq i<j\leq n}. 
\end{array}\]
Clearly $\dim(\cL_0^0) = n^2-1$ and $\dim(\cL_0^+) = \dim(\cL_0^-) = {n\choose 2}$. Note also that $G^0$ stabilizes each of $\cL_0^0$, $\cL_0^+$ and $\cL_0^-$. On the other hand $G^+$ centralizes $\cL_0^+$ but stabilizes neither $\cL_0^0$ nor $\cL_0^-$; similarly, $G^-$ centralizes $\cL_0^-$ but stabilizes neither $\cL_0^0$ nor $\cL_0^+$. 

Recall that $G^0 \cong \mathrm{SL}(n,\FF)$. As $G^0$ stabilizes $\cL_0^0$, the subspace $\cL_0^0$ is also a $G^0$-module, which is isomorphic with the $\mathrm{SL}(n,\FF)$-module $M_n^0(\FF)$ of traceless square matrices of order $n$. A canonical isomorphism from 
the module $(\cL_0^0,G^0)$ to $(M_n^0(\FF), \mathrm{SL}(n,\FF))$ can be construced as follows. The mappings
\[\begin{array}{rcccl}
\rho_{\cL_0^0}   & : & x ~ = ~  \left(\begin{array}{lc}
X & O_n \\
O_n & -X^t
\end{array}\right) \in \cL_0^0 & \longrightarrow & \rho_{\cL_0^0}(x) ~ = ~  X \in M_n^0(\FF) \\
\end{array}\]
\[\begin{array}{rcrcl}
\rho_{G^0} & : & g ~ = ~  \left(\begin{array}{lc}
A & O_n \\
O_n &  A^{-t}
\end{array}\right) \in G^0 & \longrightarrow &  \rho_{G^0}(g) ~ = ~A \in \mathrm{SL}(n,\FF)  \\
\end{array}\]
are isomorphisms from $\cL_0^0$ to $M_n^0(\FF)$ and from $G^0$ to $\mathrm{SL}(n,\FF)$. The pairing $\rho = (\rho_{\cL_0^0}, \rho_{G^0})$ is an isomorphism of modules from $(\cL_0^0, G^0)$ to $(M_n^0(\FF), \mathrm{SL}(n,\FF))$, the canonical one.  Moreover, 
\begin{equation}\label{D I 1}
\tau_\cM(x,y) ~ = ~ \Tr(\rho_{\cL_0^0}(x)\rho_{\cL_0^0}(y)), \hspace{5 mm} \forall x, y \in \cL_0.
\end{equation} 
Clearly the mapping $\rho_{\cL_0^0}^*$ which maps every linear functional $f$ of $\cL_0^0$ onto the linear functional $f\circ\rho_{\cL_0^0}^{-1}$ of $M_n^0(\FF)$ is an isomorphism from the dual $\cL_0^{0*}$ of $\cL_0^0$ to the dual $M_n^0(\FF)^*$ of $M_n^0(\FF)$ and the pairing $\rho^* := (\rho_{\cL_0^0}^*, \rho_{G^0})$ is an isomorhism of modules from $(\cL_0^{0*}, G^0)$ to $(M_n^0(\FF)^*, \mathrm{SL}(n,\FF))$. 

As recalled in Remark \ref{g**(-1)d(g) BD}, for every derivation $d\in \Der(\FF)$ and every $g\in G$ the matrix $g^{-1}d(g)$ belongs to $\cL_0$. In particular, if $g\in G^0$ then $g^{-1}d(g)\in \cL_0^0$. Formula  (\ref{D I 1}) and \cite[Theorem C.1]{CGP2} now imply the following.  

\begin{lemma}\label{reduction1}
For every cocycle $F\in Z^1(\cL_0^*,G)$ there exists a unique derivation $d_F\in \Der(\FF)$ and a coboundary $F_1\in B^1(\cL_0^*,G)$ such that 
$F(g)(x) = \tau_\cM(g^{-1}d_F(g), x) + F_1(g)(x)$ for every $g\in G^0$ and every $x\in \cL_0^0$.
\end{lemma}

\begin{remark}
\em
The coboundary $F_1$ satisfying the condition required in Lemma \ref{reduction1} is not uniquely determined. Indeed every linear functional of $\cL^0_0$ can be extended in many ways to a linear functional of $\cL_0$. Accordingly, several  coboundaries exist in $B^1(\cL^*_0,G)$ which induce the same coboundary of $G^0$ over $\cL^0_0$.
\end{remark}

Recall that, according to the way $\gamma$ is defined, we have $\tau_\cM(g^{-1}d_F(g), x) = \gamma_{d_F}(g)(x)$. So, if we replace $F$ with $F-\Gamma_{d}-F_1$, we can focus on the cocycles $F$ such that $F(g)(x) = 0$ for every $g\in G^0$ and every $x\in \cL_0^0$. 

\begin{lemma}\label{reduction1bis}
Let $F\in Z^1(\cL_0^*,G)$ be such that $F(g)(x) = 0$ for every $g\in G^0$ and every $x\in \cL_0^0$. Then there exists a coboundary $F_2\in B^1(\cL^*_0,G)$ such that $F(g) = F_2(g)$ for every $g\in G^0$. 
\end{lemma}
{\em Proof.} By Lemma \ref{L0perp}, for every $g\in G$ there exists a matrix $m_g\in\cM_\cL$ (uniquely determined up to multiples of the identity $I$) such that $F(g)(x) = \tau_\cM(m_g, x)$ for every $x\in \cL_0$. By assumption, $\tau_\cM(m_g, x) = 0$ for every $g\in G^0$ and every $x\in \cL_0^0$. This implies that
\[m_g ~ = ~ \left(\begin{array}{cc} 
\lambda_gI_n & B_g\\
C_g & \lambda_gI_n
\end{array}\right)\]
for suitable antisymmetric matrices $B_g$ and $C_g$. However, $m_g$ is determined up to multiples of $I$. Accordingly, we can always assume that $\lambda_g = 0$, namely
\begin{equation}\label{D I 2}
m_g ~ = ~ \left(\begin{array}{cc} 
O_n & B_g\\
C_g & O_n
\end{array}\right).
\end{equation}  
We call this matrix the {\em canonical representative matrix} of the linear functional $F(g)$. Suppose now that $g$ is an involution. Then the cocycle condition on $F$ implies that
\begin{equation}\label{D I 3}
\tau_\cM(m_g, gxg^{-1}) + \tau_\cM(m_g, x) ~ = ~ 0, \hspace{5 mm} \forall x \in \cL_0.
\end{equation}  
We know that $G$ preserves $\tau_\cM$ (Lemma  \ref{GT}). Hence $\tau_\cM(m_g, gxg^{-1}) = \tau_\cM(g^{-1}m_gg, x)$ and (\ref{D I 3}) says that $\tau_\cM(g^{-1}m_gg+m_g, x)  = 0$ for every $x\in\cL_0$. Terefore $g^{-1}m_gg+m_g \in \cL_0^\perp$, namely 
\begin{equation}\label{D I 4}
g^{-1}m_gg+m_g ~ = ~ \lambda I
\end{equation} 
for a suitable scalar $\lambda$. However 
\[g ~ = ~ \left(\begin{array}{lc}
A_g & O_n \\
O_n & A_g^{-t}
\end{array}\right)\]
for a suitable matrix $A_g$ such that $A_g^2 = I_n$ (recall that we have assumed that $g^2 = 1$). Accordingly and in view of (\ref{D I 2}) we obtain that
\[g^{-1}m_gg+m_g ~ = ~ \left(\begin{array}{cc}
O_n & A_g^{-1}B_gA_g^{-t}+B_g \\
A_g^tC_gA_g + C_g & O_n
\end{array}\right). \]
Comparing the above with (\ref{D I 4}) we obtain that $g^{-1}m_gg+m_g = 0$, namely 
\begin{equation}\label{D I 5}
A_g^{-1}B_gA_g^{-t}+B_g ~ = ~ A_g^tC_gA_g + C_g ~ = ~ O_n.
\end{equation} 
Let $g = u_{k,h}(t)$ (which is an involution, since $\ch(\FF) = 2$). Then (\ref{D  I 5}) implies that the $h$-th row of $B_g$ and the $k$-th row of $C_g$ (as well the $h$-th  column of $B_g$ and the $k$-th column of $C_g$, of course) are null. Similarly, let $g' = u_{p,q}(s)$ and let
\[m_{g'} ~ = ~ \left(\begin{array}{cc} 
O_n & B_{g'}\\
C_{g'} & O_n
\end{array}\right)\]
be the canonical representative matrix of $F(g')$. Then the $q$-th row of $B_{g'}$ and the $p$-th row of $C_{g'}$ are null. 

In order to have a less cumbersome notation, in the sequel we write $m, B, C$ and $m', B', C'$ instead of $m_g, B_g, C_g$ and $m_{g'}, B_{g'}, C_{g'}$. Suppose that $g$ and $g'$ commute (as it is the case when $\{k,h\}\cap\{p,q\} = \emptyset$ or $q = h$ or $p = k$). The cocycle condition on $F$ implies that 
\[\tau_\cM({g'}^{-1}mg'+ m', x)  ~ = ~ \tau_\cM(g^{-1}m'g+ m, x), \hspace{5 mm} \forall x\in \cL_0.\] 
Hence $({g'}^{-1}mg'-m')-(g^{-1}m'g + m) = \lambda I$ for a scalar $\lambda$, which is easily seen to be null. Therefore 
\[\begin{array}{rcl}
se_{p,q}B + sBe_{q,p} + s^2e_{p,q}Be_{q,p}  ~ = ~ te_{k,h}B' + tB'e_{h,k} + t^2e_{k,h}B'e_{h,k}; \\
se_{q,p}C + sCe_{p.q} + s^2e_{q,p}Ce_{p,q} ~ = ~ te_{hk}C' + tC'e_{k,h} + t^2e_{h,k}C'e_{k,h}.
\end{array}\]
Equivalently, since $e_{p,q}Be_{q,p} = e_{k,h}B'e_{h,k} = e_{q,p}Ce_{p,q} = e_{h,k}C'e_{k,h} = O_n$, 
\begin{equation}\label{D I 6}
\left\{\begin{array}{rcl}
se_{p,q}B + sBe_{q,p}  ~ = ~ te_{k,h}B' + tB'e_{h,k}; \\
se_{q,p}C + sCe_{p.q}  ~ = ~ te_{hk}C' + tC'e_{k,h}.
\end{array}\right.
\end{equation} 
For $i = 1, 2,..., n$ let $b_i(t)$, $b'_i(s)$, $c_i(t)$ and $c'_i(s)$ be the $i$-th rows of respectively $B$, $B'$, $C$ and $C'$. From (\ref{D I 6}) and recalling that the $h$-th column of $B$ is null, with $p\neq k$ we obtain that the $p$-th entry of $b_q(t)$ is null for every $p\neq k$. Similarly, the $k$-th entry of $b'_h(s)$ is null for every $k\neq p$, the $q$-th entry of $c_p(t)$ is null for every $q\neq h$ and the $h$-th entry of $c'_k(s)$ is null for every $h\neq p$. Moreover, $sb_k(t) = tb'_p(s)$ and $sc_h(t) = tc'_q(s)$. Consequently, there exist vectors $b, b', c, c'\in V(n,\FF)$ such that $b_k(t) = tb$, $b'_p(s) = sb'$, $c_h(t) = tc$ and $c'_q(s) = sc$. From  (\ref{D I 6}) now we also obtain that the $q$-th entry of $b$ is equal to the $h$-th entry of $b'$ and the $p$-th entry of $c$ is equal to the $k$-th entry of $c'$. 

In order to profit from the above remarks, we need to improve our notation. We put $b = (b_i^{(k,h)})_{i = 1}^n$, $b' = (b_i^{(p,q)})_{i = 1}^n$, $c = (c_i^{(k,h)})_{i = 1}^n$ and $c' = (c_i^{(p,q)})_{i = 1}^n$. So, 
\begin{equation}\label{D I 6bis}
\left\{\begin{array}{ll}
B = t\sum_{i=1}^nb^{(k,h)}_i(e_{h,i}+e_{i,h}), & C = t\sum_{i=1}^nc^{(k,h)}_i(e_{k,i}+e_{i,k}),\\
B' = s\sum_{i=1}^nb^{(p,q)}_i(e_{q,i}+e_{i,q}), & C' = s\sum_{i=1}^nc^{(p,q)}_i(e_{p,i}+e_{i,p}), \\
b^{(k,h)}_q = b^{(p,q)}_h, & c^{(k,h)}_p = c^{(p,q)}_k.
\end{array}\right.
\end{equation} 
These conditions hold for every choice of $k,h,p,q$ with $k\neq h$, $p\neq q$ and either $\{k,h\}\cap\{p,q\} = \emptyset$ or $h = q$ or $k = p$. The last two relations of (\ref{D I 6bis}) guarantee the existence of antysimmetric matrices $X_F = (x_{i,j})_{i,j=1}^n$ and $Y_F = (y_{i,j})_{i,j=1}^n$ such that \begin{equation}\label{D I 6ter}
b^{(k,h)}_q = x_{h,q}= x_{q,h} = b^{(p,q)}_h~ \mbox{ and } ~ c^{(k,h)}_p = y_{k,p}= y_{p,k} = b^{(p,q)}_k
\end{equation}
for every choice of $k, h, p, q$ as above. Let $f$ be the linear functional represented by the matrix 
\[\left(\begin{array}{ll}
O_n & X_F \\
Y_F & O_n
\end{array}\right)\]
and let $F_2:g\in G \rightarrow f-f\cdot g$ be the coboundary associated to $f$. As one can check, the equations of (\ref{D I 6ter}) imply that  $F(g) = F_2(g)$ for every $g  \in U^0 := \{u_{k,h}(t) ~ |~ t\in \FF, ~ h\neq k\}$. Since $U^0$ generates $G^0$, by the cocycle condition we obtain that $F(g) = F_2(g)$ for every $g\in G^0$.  \hfill $\Box$ \\

Lemmas \ref{reduction1} and \ref{reduction1bis}  imply the following. 

\begin{proposition}\label{reduction2}
For every cocycle $F\in Z^1(\cL_0^*,G)$ there exists a unique derivation $d_F\in \Der(\FF)$ and a coboundary $F'\in B^1(\cL_0^*,G)$ such that $F(g)(x) = \tau_\cM(g^{-1}d_F(g), x) + F'(g)(x)$ for every $g\in G^0$ and every $x\in \cL_0$.
\end{proposition}
 
\subsubsection{$F(g)$ for $g\in G^\pm$}

In view of Proposition \ref{reduction2} we can assume that $F(g) = 0$ for every $g\in G^0$. In the sequel, given $g\in G$, the symbol $m_g$ stands for a representative matrix of $F(g)$. 

Before turning to $G^+$ and $G^-$ we shall prove that $F$ is null also on the normalizer $N(G^0)$ of $G^0$ in $G$. Recall that $G^0$ consists of the elements of $G$ represented by block-diagonal matrices $\mathrm{diag}(X, X^{-t})$ with $\mathrm{det}(X) = 1$. The  elements of $G$ represented by block-diagonal matrices $\mathrm{diag}(X, X^{-t})$ with $\mathrm{det}(X) \neq 0$ form $N(G^0)$.  

\begin{lemma}\label{D diag}
If $g$ normalizes $G^0$ then $F(g) = 0$.
\end{lemma}
\noindent
{\em Proof.} Let $g\in N(g^0)$. Then for every $g_0\in G^0$ there exists $g'_0 \in G^0$ such that $g'_0g = gg_0$ and the coycyle condition on $F$ implies that $F(g'_0)\cdot g + F(g) = F(g)\cdot g_0 + F(g_0)$. However $F(g'_0) = F(g_0) = 0$. Therefore $F(g) = F(g)\cdot g_0$ for every $g_0\in G^0$. Let  
\[m_g ~ = ~ \left(\begin{array}{cc} 
A_g & B_g \\
C_g & A_g^t
\end{array}\right), \hspace{5 mm} (B_g \mbox{ and } C_g \mbox{ antisymmetric})\]
be a representative matrix of $F(g)$. The equality $F(g) = F(g)\cdot g_0$ implies that $g_0^{-1}m_gg_0 \equiv m_g ~(\mathrm{mod}~\langle I\rangle$) for every $g_0\in G^0$. Therefore $X^{-1}A_gX \equiv A_g ~(\mathrm{mod} ~\langle I_n\rangle)$, $X^{-1}B_gX^{-t} = B_g$ and $X^tC_gX = C_g$ for every $X\in \mathrm{SL}(n,\FF)$. This forces $B_g = C_g = O_n$ and $A_g\in \langle I_n\rangle$, namely $m_g\in \langle I\rangle$. Hence  $F(g) = 0$, as claimed.   \hfill $\Box$\\

We now turn to $G^+$. The group $G^+$ is commutative and all of its elements are involutions. Let $g \in G^+$. Since all elements of $G^+$ are involutions, from the proof of Propositon \ref{reduction2} we know that
\begin{equation}\label{D II 1}
g^{-1}m_gg+m_g ~ = ~ \lambda_g I
\end{equation} 
for a suitable scalar $\lambda_g\in\FF$, where 
\[m_g ~ = ~ \left(\begin{array}{cc} 
A_g & B_g \\
C_g & A_g^t
\end{array}\right)\]
is a representaive matrix of $F(g)$. Recall that 
\[g ~ = ~ \left(\begin{array}{cc} 
I_n & O_n\\
U_g & I_n
\end{array}\right)\] 
for a suitable antisymmetric matrix $U_g$ of order $n$. Condition (\ref{D II 1}) thus implies that $B_gU_g = \lambda_g I_n$ and $U_gA_g+A^t_gU_g = \lambda_gU_g$. The group $G^+$ contains many elements $g$ such that $\mathrm{det}(U_g) = 0$. (For instance, $u_{k,h}^+(t)$ has this property for any choice of $k < h$.) Suppose that $g$ is one of them. Then the relation $B_gU_g = \lambda_gI_n$ implies $\lambda_g = 0$ and the above boils down to the following:
\begin{equation}\label{D II 2} 
B_gU_g ~=~ O_n, \hspace{5 mm} U_gA_g ~=~ A_g^tU_g.
\end{equation}
Let $g'$ be another element of $G^+$. From the proof of Propositon \ref{reduction2} we know that  
\begin{equation}\label{D II 2?}
({g'}^{-1}mg'-m')-(g^{-1}m'g + m) = \lambda_{g,g'} I.
\end{equation} 
Therefore
\begin{equation}\label{D II 2bis}
BU' = B'U+\lambda I_n, \hspace{5 mm} U'A+ U'BU' +A^tU' = UA'+ UB'U + {A'}^tU 
\end{equation}
where we write $U, A, B, U', A', B'$ and $\lambda$ for $U_g, A_g, B_g, U_{g'}, A_{g'}, B_{g'}$ and $\lambda_{g,g'}$ respectively. In view of the first relation of (\ref{D II 2bis}) the second relation can be rewritten as follow:
\[ U'A+ U'B'U + \lambda U' + A^tU' ~=~ UA'+ UB'U + \lambda U + {A'}^tU. \]
However $UB = BU = O_n$ (by (\ref{D II 2}) and because $UB = (BU)^t$). Suppose that we also have $\mathrm{det}(U') = 0$. Then $U'B' = B'U' = O_n$ and the above boils down to
\begin{equation}\label{D II 3}
BU' = B'U+\lambda I_n, \hspace{5 mm} U'A+ A^tU' + \lambda U' = UA'+ {A'}^tU + \lambda U.
\end{equation} 
Let now $g = u^+_{k,h}(t)$ and $g' = u_{p,q}(s)$, with $\{k,h\}\neq \{p,q\}$ and $t \neq 0 \neq s$. Then (\ref{D II 2?}) implies that the $k$-th and $h$-rows of $B$ are null (as well as the $k$-th and $h$-th columns, of course) and all entries of the $k$-th and $h$-th columns of $A$ but possibly the $k$-th and $h$-entries are null. Moreover the $(k,k)$-entry and the $(h,h)$-entry of $A$ are equal. The same features occur for $B'$ and $A'$ up to replacing $k$ and $h$ with $p$ and $q$. In turn the first part of (\ref{D II 3}) implies that $\lambda = 0$ and the $p$-th and $q$-th rows of $B$ are null as well as the $k$-th and $h$-rows of $B'$. As $(p,q)$ is an arbitray pair of distinct indices, we obtain that $B = O_n$. Similarly, $B' = O_n$. 

As $\lambda = 0$, the second part of (\ref{D II 3}) just states that $U'A+ A^tU' = UA'+ {A'}^tU$. This condition implies that all entries of the $p$-th and $q$-th rows of $A$, except possibly the $p$-th, the $q$-th, the $k$-th and the $h$-entry, are null. Moreover the $(p,p)$-entry of $A$ is equal to the $(q,q)$-entry. The same properties hold for $A'$ up to permuting the pairs $(p,q)$ and $(k,h)$. We also obtain that, if  $a_{i,j}$ and $a'_{i,j}$ are the $(i,j)$ entries of $A$ and $A'$ respectively, then   
\[\begin{array}{rl}
sa_{q,k} = ta'_{h,p} & \mbox{if } p\neq k,\\
sa_{q,h} = ta'_{k,p} & \mbox{if } p\neq h,\\
sa_{p,k} = ta'_{h,q} & \mbox{if } q\neq k,\\
sa_{p,h} = ta'_{k,q} & \mbox{if } q\neq h.
\end{array}\]
If $|\FF| > 2$ then, since $s$ ad $t$ are arbitrary non-zero elements of $\FF$, we have 
\[\begin{array}{rl}
a_{q,k} = a'_{h,p} = 0 & \mbox{if } p\neq k,\\
a_{q,h} = a'_{k,p} = 0 & \mbox{if } p\neq h,\\
a_{p,k} = a'_{h,q} = 0 & \mbox{if } q\neq k,\\
a_{p,h} = a'_{k,q} = 0 & \mbox{if } q\neq h.
\end{array}\]
Since $(p,q)$ is an arbitrary pair of indices (with $p < q$) different from $(k,h)$, the properties so far collected on $A$ show that if either $n > 3$ or $\FF\neq \FF_2$ then $A$ is a scalar matrix. Under these hypothes, as the matrix $m_g$ is determined up to multiples of $I$, we can assume that $A = O_n$ and $m_g$ can be given the following canonical form: 
\begin{equation}\label{D II 4} 
m_g ~ = ~ \left(\begin{array}{cc} 
O_n & O_n \\
C_g & O_n
\end{array}\right) \hspace{5 mm} (C_g \mbox{ antisymmetric}).
\end{equation}
This holds for every $g\in U^+_{k,h}$. However $\cup_{k<h} U^+_{k,h}$ generates $G^+$. Therefore, by the cocycle condition on $F$ and Lemma \ref{GT}, for every $g\in G^+$ the linear functional $F(g)$ admits a canonical representative matrix of the form (\ref{D II 4}).

Similarly, under the assumption that $\FF \neq \FF_2$ when $n = 3$, if $g\in G^-$ then $F(g)$ admits a canonical representative matrix as follows: 
\begin{equation}\label{D II 5}
m_g ~ = ~ \left(\begin{array}{cc} 
O_n & B_g \\
O_n & O_n
\end{array}\right) \hspace{5 mm} (B_g \mbox{ antisymmetric}).
\end{equation}

\begin{proposition}\label{reduction3} 
Let $F\in Z^1(\cL_0^*,G)$ be such that $F(g)(x) = 0$ for every $g\in G^0$ and every $x\in \cL_0$. Suppose that $n > 3$. Then $F \in \langle F_0\rangle$. 
\end{proposition}
\noindent
{\em Proof.} Let $F$ be as in the hypotheses of the proposition. Given $g = u^+_{k,h}(t) \in U^+_{k,h}$ and $g = u^-_{pq}(s)\in U^-_{p,q}$  put $m := m_g$, $m' := m_{g'}$, $C_{k,h}(t) := C_g$ and $B_{p,q}(s) = B_{g'}$, with $G_g$ and $B_{g'}$ as in (\ref{D II 4}) and respectively (\ref{D II 5}). Suppose that $\{p,q\}\cap\{k,h\} = \emptyset$. Then $g$ and $g'$ commute. Accordingly, $g'^{-1}mg + m' \equiv g^{-1}m'g + m ~ (\mathrm{mod}~\langle I\rangle$, namely
\begin{equation}\label{D II 6} 
\left\{\begin{array}{rcl}
st(e_{q,p}-e_{p,q})C_{k,h}(t)(e_{p,q}-e_{q,p}) & = & O_n, \\
st(e_{k,h}-e_{h,k})B_{p,q}(s)(e_{h,k}-e_{k,h}) & = & O_n,\\
s(e_{q,p}-e_{p,q})C_{k,h}(t) & = & tB_{p,q}(s)(e_{k,h}-e_{h,k}) + \lambda I_n \mbox{ for some } \lambda\in\FF.  
\end{array}\right.
\end{equation} 
As $\{p,q\}$ is any pair of indices disjoint from the pair $\{k,h\}$ the first relation of (\ref{D II 6}) imply that all rows (and columns) of $C_{k,h}(t)$ but possibly the $k$-th and the $h$-th ones are null. Similarly, all rows (and columns) of $B_{p,q}(s)$ but possibly the $p$-th and the $q$-th ones are null.
Explicitly,
\[\begin{array}{rcl}
C_{k,h}(t) & = & c_{k,h}(t)(e_{h,k} - e_{k,h}) +  \sum_{j \neq k, h}(c_{k,j}(t)(e_{j,k}-e_{k,j}) + c_{h,j}(t)(e_{j,h}- e_{h,j})),\\
B_{p,q}(s) & = & b_{p,q}(s)(e_{p,q} - e_{q,p}) +  \sum_{j \neq p, q}(b_{p,j}(s)(e_{p,j}-e_{j,p}) + b_{q,j}(s)(e_{q,j}- e_{j,q})),
\end{array}\]
for suitable additive functions $c_{k,h}, c_{k,j}, c_{h,j}, b_{p,q}, b_{p,j}, b_{q,j}$ from $\FF$ to $\FF$. (Additivity easily follows from that fact that $u^+_{k,h}(t_1)u_{k,h}^+(t_2) = u^+_{k,h}(t_1+t_2)$ and $u^-_{p,q}(s_1)u_{p,q}^-(s_2) = u^-_{p,q}(s_1+s_2)$.) Taking this result into accont the third equation of (\ref{D II 6}) implies that $\lambda = 0$ and
\[sc_{k,p}(t) = -tb_{h,q}(s), ~ sc_{h,p}(t) = tb_{k,q}(s), ~ sc_{k,q}(t) = tb_{h,p}(s), ~ sc_{h,q}(t) = -tb_{k,p}(s).\]
Hence 
\begin{equation}\label{D II 7}
\left\{\begin{array}{rcl}
C_{k,h}(t)  & = & c_{k,h}(t)(e_{h,k} - e_{k,h}) +  t\sum_{j \neq k, h}(c_{k,j}(e_{j,k}-e_{k,j}) + c_{h,j}(e_{j,h}- e_{h,j})),\\
  & & \mbox{ where } c_{k,j} := c_{k,j}(1) \mbox{ and } c_{h,j} := c_{h,j}(1) \mbox{ for } j \neq k, h; \\
B_{p,q}(s) & = &  b_{p,q}(s)(e_{p,q} - e_{q,p}) +  s\sum_{j \neq p, q}(b_{p,j}(e_{p,j}-e_{j,p}) + b_{q,j}(e_{q,j}- e_{j,q})), \\
 &  & \mbox{ where } b_{p,j} := b_{p,j}(1) \mbox{ and } b_{q,j} := b_{q,j}(1) \mbox{ for } j \neq p, q.  
\end{array}\right.
\end{equation} 
Moreover
\begin{equation}\label{D II 7bis}
c_{k,p} = -b_{h,q}, ~~ c_{h,p} = b_{k,q}, ~~ c_{k,q} = b_{h,p}, ~~ c_{h,q} = -b_{k,p}.
\end{equation} 
With $g = u^+_{k,h}(t)$ and $\{p,q\}\cap\{k,h\} = \emptyset$ as above (but now allowing $p > q$, while above we had $p < q$), let $g_0 = u_{p,q}(s) = I + se_{p,q} - se_{n+q,n+p}$.  We have $gg_0 = g_0g$. Hence $F(g)\cdot g_0 + F(g_0) = F(g_0)\cdot g + F(g)$. However $F(g_0) =  0$ by assumption (since $g_0\in G^0$) while $F(g)$ is represented by a matrix $m_g$ as in (\ref{D II 4}) with $C_g = C_{k,h}(t)$, the latter being described in (\ref{D II 7bis}). Therefore $g_0^{-1}m_gg_0 \equiv m_g$ $(\mathrm{mod}~ \langle I\rangle)$. In view of (\ref{D II 7bis}), this equivalence forces $c_{k,j} = c_{h,j} = 0$ for every $j \neq k,h$. Similarly, $b_{p.j} = b_{q,j} = 0$ for every $j\neq p,q$. Therefore,
\begin{equation}\label{D II 8}
\left\{\begin{array}{rrcl}
\mbox{the matrix} & c_{k,h}(t)x_{\alpha^+_{[k,h]}} & \mbox{represents} & F(u^+_{k,h}(t)); \\
\mbox{the matrix} & b_{p,q}(s)x_{\alpha^-_{[p,q]}} & \mbox{represents} & F(u^-_{p,q}(s)).
\end{array}\right.
\end{equation}
With $g = u^+_{k,h}(t)$ let now $g_0 = u_{k,p}(s)$ with $p\not\in\{k,h\}$. Now $g_0^{-1}gg_0 = gg'$ where $g' = u^+_{p,h}(st)$ or $g' = u^+_{h,p}(st)$ according to whether $p < h$ or $h < p$. We have $F(g_0) = 0$ because $g_0\in G^0$ and $F(gg') = F(g) + F(g')$ (because $g'$ commutes with any matrix as in (\ref{D II 4}), hence with any representative matrix of $F(g)$). The cocycle condition on $F$ now yields $F(g)\cdot g_0 = F(g)+F(g')$. In view of (\ref{D II 8}) this equality is equivalent to the following, where to fix ideas we assume $p < h$: 
\[c_{k,h}(t)(e_{h,k}-e_{k,h}+ s(e_{p,h}+ e_{h,p})) ~ = ~ c_{k,h}(t)(e_{h,k}-e_{k,h}) + c_{p,h}(st)(e_{h,p}- e_{p,h}).\]
Equivalently, $sc_{k,h}(t) = c_{p,h}(st)$ (recall that $\ch(\FF) = 2$).  With $s = 1$ this condition shows that $c_{k,h}(t) = c_{p,h}(t)$ for every $t$ while with $t = 1$ it yieds $c_{p,h}(s) =  sc_{k,h}(1)$ for every $s$. It follows that there exists a scalar $\lambda\in \FF$ such that $c_{k,h}(t) = c_{p,h}(t) = \lambda t$ for every $t\in \FF$. A similar result holds for $g = u^-_{k,h}(t)$. As the indices $k, h, p$ are arbitray, with the only restriction that they are pairwise different (and $k < h$), we obtain the following conclusion: there exist two scalars $\lambda^+, \lambda^- \in \FF$ such that 
\begin{equation}\label{D II 9}
 \mbox{for every } t\in \FF ~\left\{\begin{array}{rrcl}
\mbox{the matrix} & t\lambda^+x_{\alpha^+_{[k,h]}} & \mbox{represents} & F(u^+_{k,h}(t); \\
\mbox{the matrix} & t\lambda^-x_{\alpha^-_{[k,h]}} & \mbox{represents} & F(u^-_{k,h}(t)). 
\end{array}\right. 
\end{equation}
We claim that $\lambda^+ = \lambda^-$. 

Given $t\in\FF\setminus\{0\}$, put $w(t) := u^+_{1,2}(t)u_{1,2}^-(-t^{-1})u_{1,2}^+(t)$ and $h(t) := w(t)w(1)$. The element $h(t)$ is represented by a diagonal matrix. Hence it normalizes $G^0$. Therefore $F(h(t)) = 0$ by Lemma \ref{D diag}. On the other hand, in view of (\ref{D II 9}), the linear functional $F(w(t))$ is represented by the following matrix:
\[m_w(t) ~ = ~ \left(\begin{array}{cc}
-\lambda^-(e_{1,1}+e_{2,2}) & (\lambda^++\lambda^-)t^{-1}(e_{2,1}-e_{1,2}) \\
(\lambda^++\lambda^-)t(e_{2,1}-e_{1,2}) & \lambda^-(e_{1,1}+e_{2,2})
\end{array}\right).\]
So $F(h(t))$ is represented by
\[\begin{array}{l}
w(1)^{-1}m_w(t)w(1) + m_w(1) ~ = \\
= ~ \left(\begin{array}{cc}
 O_n & (\lambda^++\lambda^-)(1+t)(e_{2,1}-e_{1,2}) \\
(\lambda^++\lambda^-)(1+\frac{1}{t})(e_{2,1}-e_{1,2}) & O_n
\end{array}\right).
\end{array}\]
As $F(h(t)) = 0$, the above matrix is null for every $t\neq 0$. When $\FF\neq \FF_2$ we can choose $t\not\in\{0,1\}$. With this choice of $t$, the above matrix is null only if $\lambda^++\lambda^- = 0$, namely $\lambda^+ = \lambda^-$, as claimed. 

Let $\FF = \FF_2$ and suppose for a contradiction  that $\lambda^+ \neq \lambda^-$, namely $\{\lambda^+,\lambda^-\} = \{0,1\}$. Let $\lambda^+ = 1$ and $\lambda^- = 0$, to fix ideas. So, $F(g) = 0$ for every $g\in G^0\cup G^-$ while $F(g)$ is represented by both $g$ and $g +I$ fo every $g\in G^+$. Let $g_0$ be the element of $G$ represented by the matrix
\[\left(\begin{array}{cc}
O_n & I_n \\
I_n & O_n
\end{array}\right)\]
and let
 \[m_0 ~ = ~ \left(\begin{array}{cc}
A & B \\
C & A^T
\end{array}\right)\]
be a given representative matrix of $F(g_0)$. For every choice of $k< h$ we have $g_0u^-_{k,h}(1)g_0 = u^+_{k,h}(1)$. Therefore, since $F(u^-_{k,h}(1)) = 0$ and $F(u^+_{k,h}(1)) = \{u^+_{k,h}(1), u^+_{k,h}(1)+I\}$, we have 
\[g_0u^-_{k,h}(1)m_0u^-_{k,h}(1)g_0 + m_0 ~ \in ~ \{u^+_{k,h}(1), u^+_{k,h}(1)+I\}.\]
Recalling that $\{k,h\}$ is an arbitrary pair of indices, from the above by straightforward computations we obtain that $B = C = O_n$, $A$ is symmetric, say $A = \sum_{i < j}a_{i,j}(e_{i,j}+e_{j,i}) + \sum_{i=1}^na_{i,i}e_{i,i}$, and all entries of the $k$-th and $h$-th rows but possibly the $k$-th and $h$-ones are null. Moreover $a_{k,k}+a_{h,h} = 1$. However this holds for every choice of $k < h$. If follows that $A$ is diagonal and for every triplet of indices $i < j < k$ we have $a_{i,i}+a_{i,j} = a_{j,j}+a_{k,k} = a_{i,i}+a_{k,k} = 1$. This triple of conditions forces $0 = 1$. We have reached a contradiction, which forces us to conclude that $\lambda^+ = \lambda^-$, as claimed.  

Put $\lambda := \lambda^+ = \lambda^-$. In view of (\ref{D II 9}) and Lemma \ref{F0}, we have $F(g) = \lambda F_0(g)$ for every $g\in G^0\cup(\cup_{k<h}(U^+_{k,h}\cup U^-_{k,h}))$. However the set $G^0\cup(\cup_{k<h}(U^+_{k,h}\cup U^-_{k,h}))$ generates $G$. Consequently $F = \lambda F_0$. The proposition is proven.  \hfill $\Box$\\ 

Propositions \ref{reduction2} and \ref{reduction3} imply Proposition \ref{Main D}. The proof of Proposition \ref{Main D} is complete.  

\subsection{The long-root geometry of type $D_3$} 

By convention, the diagram $D_n$ is defined for $n > 3$. However nothing prevents to consider the Lie algebra $\mathfrak{o}(6,\FF)$ of type $D_3$ associated to the group $\Omega^+(6,\FF)$. Actually, we will need to consider both $\mathfrak{o}(6,\FF)$ and $\Omega^+(6,\FF)$ in Subsection \ref{main B} when dealing with the case $X_3 = B_3$, where $\mathfrak{o}(6,\FF)$ occurs as a subalgebra of $\mathfrak{o}(7,\FF)$ and $\Omega^+(6,\FF)$ as a subgroups of $\Omega(7,\FF)$. In that context we will need to know that Theorem \ref{even main1} also applies to $D_3$, but Proposition \ref{Main D} does not provide this information. 

However $D_3$ is nothing but $A_3$ with the nodes numbered in a different way (the types $1$ and $2$ are switched). Accordingly, $\mathfrak{o}(6,\FF) \cong \mathfrak{sl}(4,\FF)$. The long-root geometry of type $D_3$ is obviously the same as the long-root geometry $A_{3,\{1,3\}}(\FF)$ of type $A_3$, which in the $D_3$-setting should be named $D_{3,\{2,3\}}(\FF)$    $D_{3,\{2,3\}}(\FF) = A_{3, \{1, 3\}}(\FF)$ is the long-root geometry of type $D_3$. So, we might believe that Lemma \ref{A gamma1234} on $A_{3,\{1,3\}}(\FF)$, specifically regarding property $(\gamma1)$, is sufficient to conclude that Theorem \ref{Cohomology} applies to $D_3$ and, consequently, the statement of Theorem \ref{even main1} holds true for $D_3$ as well. However this conclusion, although correct, is not so obvious as it looks. Indeed it is not so obvious that the mapping $\gamma$ defined in Subsection \ref{gamma A} is basically the same as the $\gamma$ of Subsection \ref{gamma BD}. This issue needs a deeper analysis, which is what we shall do in this subsection. 

We are mainly interested in the characteristic 2 case here. However, having preliminarily looked at the case where $\ch(\FF) \neq 2$ will help to better understand the characteristic 2 case. 

Put $\cL = \mathfrak{o}(6,\FF)$, $G = \Omega^+(6,\FF)$, $\overline{\cL} = \mathfrak{sl}(4,\FF)$ and $\overline{G} := \mathrm{SL}(4,\FF)$, with $\overline{\cL}$ realized as Lie algebra of the traceless square matrices of order 4 wih entries in $\FF$. The order of the center $\overline{Z}$ of $\overline{G}$ is a divisor of $4$ and $\overline{Z}$ admits a subgroup $\overline{Z}_0$ of index at most 2 such that $\overline{G}/\overline{Z}_0 \cong G = \Omega^+(6,\FF)$. As $\overline{G}$ acts as $\overline{G}/\overline{Z}$ over $\overline{L}$, the group $G$ also acts on $\overline{L}$. On the other hand $G$ acts on the subalgebra $\cM_\cL$ of $M_6(\FF)$, which consists of the matrices described in (\ref{matrix D}).  When $\ch(\FF) \neq 2$ then $\cM_\cL = \cL$ and $\overline{\cL}$ are isomorphic both as Lie algebras and as $G$-modules. In contrast when $\ch(\FF) = 2$ these isomorphisms are missing. Indeed, as we know, when $\ch(\FF) = 2$ the group $G$ acts on a subalgera $\cL_0$ of $\cM_\cL$ which, both as a Lie algebra and as a $G$-module, is  isomorphic to the quotient $\cL/K$ of $\cL$ over its center $K$, hence to the quotient of $\overline{\cL}$ over its center $\overline{K} := \langle I_4\rangle$ (where $I_4$ is the identity matrix of order $4$). 

In the sequel we denote by $\ove_{i,j}$ the square matrix of oder $4$ with all null entries but the $(i,j)$ one, which is equal to $1$. So, $\{\ove_{i,j}\}_{i\neq j}\cup\{\ove_{i+1,i+1}-\ove_{i,i}\}_{i=1}^3$ is a Chevalley basis of $\overline{\cL}$, with $\ove _{2,1}, \ove_{3,2}$ and $\ove_{4,3}$ corresponding to the roots $\bar{\alpha}_1, \bar{\alpha}_2$ and $\bar{\alpha}_3$ of a basis of the root system of type $A_3$ while $\ove _{1,2}, \ove_{2,3}$ and $\ove_{3,4}$ correspond to their opposites. The elements $\bar{h}_1 := \ove_{2,2}-\ove_{1,1}, \bar{h}_2 := \ove_{3,3}-\ove_{2,2}$ and $\bar{h}_3 := \ove_{4,4}-\ove_{3,3}$ span the Cartan subalgebra of $\overline{\cL}$. If for $i\neq j$ and $t\in\FF$ we put $\ovu_{i,j}(t) = I_4 + t\ove_{i,j}$, then the set $\overline{U} := \{\ovu_{i,j}(t) ~|~ i\neq j, ~ t\in \FF\}$ generates $\overline{G}$. 

\subsubsection{The case that $\ch(\FF)\neq 2$}  

Let  $\ch(\FF) \neq 2$. We define an isomorphism of Lie algebras $\phi_\cL$ from $\overline{\cL}$ to $\cL = \cM_\cL$ and a homomorphism $\phi_G$ from $\overline{G}$ to $G$ as follows.
\begin{equation}\label{A3D3oddL}
\phi_\cL ~ : ~ \left\{\begin{array}{l}
\begin{array}{rcl}
\ove_{i,j} ~ (\mbox{with } i,j > 1, ~ i\neq j) & \longrightarrow & e_{i-1,j-1}- e_{2+j, 2+i};
\end{array}\\
\begin{array}{rcl}
\ove_{i,1} ~ (\mbox{with } i > 1) & \longrightarrow & e_{h+3,k}-e_{k+3,h},\\
 & & \mbox{for }  1\leq k < h\leq 3 \mbox{ such that } k+h = 7-i; \\
\ove_{1,j}  ~ (\mbox{with } j >1) & \longrightarrow & e_{k,h+3}-e_{h,k+3},\\
& &  \mbox{for }  1\leq k < h\leq 3 \mbox{ such that } k+h = 7-j; \\
\bar{h}_1 ~ = ~ \ove_{2,2}-\ove_{1,1} & \longrightarrow & -e_{2,2}-e_{3,3} + e_{5,5} + e_{6,6} ~ = ~ h_3; \\
\bar{h}_2 ~ = ~ \ove_{3,3}-\ove_{2,2} & \longrightarrow & -e_{1,1}+e_{2,2} + e_{3,3} - e_{4,4} ~ = ~ h_1; \\
\bar{h}_3 ~ = ~\ove_{4,4}-\ove_{3,3} & \longrightarrow & -e_{2,2}+e_{3,3} + e_{4,4} - e_{5,5} ~ = ~ h_2.
\end{array}
\end{array}\right.
\end{equation}
\begin{equation}\label{A3D3oddG}
\phi_G ~ : ~ \left\{\begin{array}{l}
\begin{array}{rcl}
\ovu_{i,j}(t) ~ (\mbox{with } i,j > 1) & \longrightarrow & u_{i-1,j-1}(t)
\end{array}\\
\begin{array}{rcl}
\ovu_{i,1}(t)  & \longrightarrow & u^+_{k,h}(t), \mbox{ for }  1\leq k < h\leq 3 \mbox{ such that } k+h = 7-i; \\
\ovu_{1,j}(t)  & \longrightarrow &  u^-_{k,h}(t), \mbox{ for } 1\leq k < h\leq 3 \mbox{ such that } k+h = 7-j; 
\end{array}
\end{array}\right.
\end{equation}
It is straightforward to check that $\phi_\cL$ and $\phi_G$ are indeed an isomorphism from $\overline{L}$ to $\cL$ and respectively a surjective homomorphism from $\overline{G}$ to $G$. (Note that the roots called $\bar{\alpha}_1, \bar{\alpha}_2$ and $\bar{\alpha}_3$ in $A_3$ are renamed $\alpha_3, \alpha_1$ and $\alpha_2$ in $D_3$.) Moreover, $\ker(\phi_G) = \overline{Z}_0$. If $|\overline{Z}| = 2$ then $\overline{Z}_0 = 1$ and $\phi_G$ is an isomorphism from $\overline{G}$ to $G$; otherwise $|\overline{Z}| = 4$, $|\overline{Z}_0| = 2$ and $\phi_G$ induces an isomorphism from $\overline{G}/\overline{Z}_0$ to $G$, which by a little abuse we still denote by $\phi_G$. In any case, the pair $\phi = (\phi_L, \phi_G)$ is an isomorphism of modules from $(\overline{\cL}, \overline{G}/\overline{Z}_0)$ to $(\cL,G)$.

Clearly, $\phi$ induces and isomorphism $\phi_Z$ from $Z^1(\cL^*,G)$ to $Z^1(\overline{\cL}^*, \overline{G}/\overline{Z}_0)$) which maps $B^1(\cL^*,G)$ onto $B^1(\overline{\cL}^*, \overline{G}/\overline{Z}_0)$. Explicitly,
\[\phi_Z~ : ~ F\in Z^1(\cL^*,G) ~ \longrightarrow ~ \overline{F} ~: ~ \left(\begin{array}{rcl}
{} & \bar{g} \in \overline{G}/\overline{Z}_0 & \\
& \downarrow&  \\
\ovx\in \overline{\cL} & \longrightarrow & F(\phi_G(\bar{g}))(\phi_L(\ovx))
\end{array}\right).\]
We can rephrase the above as follows. As every coycle of $\overline{G}$ over $\overline{\cL}^*$ is null over $\overline{Z}$, there is no substantial difference between $Z^1(\overline{\cL}^*, \overline{G}/\overline{Z}_0)$ and $Z^1(\overline{\cL}^*, \overline{G})$. So we can freely regard $Z^1(\overline{\cL}^*, \overline{G}/\overline{Z}_0)$ and $Z^1(\overline{\cL}^*, \overline{G})$ as the same structure, thus regarding $\phi_Z$ as an isomorphism from $Z^1(\cL^*, G)$ to $Z^1(\overline{\cL}^*, \overline{G})$ (which maps $B^1(\cL^*, G)$ onto $B^1(\overline{\cL}^*, \overline{G})$). In this way we can avoid unnecessary complications in what we are going to say in the next lines.        

For every $d\in \Der(\FF)$ let $\overline{\gamma}_d:\overline{G}\rightarrow \overline{\cL}^*$ be the mapping defined as follows:
\begin{equation}\label{ovgamma}
\overline{\gamma}_d ~:~ \bar{g}\in \overline{G} ~\longrightarrow ~ \overline{\gamma}_d(\bar{g}) ~ : ~ \left(\begin{array}{c}
\ovx \in \overline{\cL} \\
\downarrow \\
\Tr(\bar{g}^{-1}d(\bar{g})\ovx) \in \FF
\end{array}\right).
\end{equation}   
Then $\overline{\gamma}_d\in Z^1(\overline{\cL}^*,\overline{G})$ and, as proved in \cite[Appendix C]{CGP2}, the linear mapping 
\[\begin{array}{rcccc}
\overline{\gamma} & : & \Der(\FF) & \longrightarrow & Z^1(\overline{\cL}^*,\overline{G}) \\
& & d\in\Der(\FF) & \longrightarrow & \overline{\gamma}_d 
\end{array}\] 
is injective and its image is a complement of  $B^1(\overline{\cL}^*,\overline{G})$ in  $Z^1(\overline{\cL}^*,\overline{G})$. Accordingly, $\overline{\gamma}$ induces an isomorphism from $\Der(\FF)$ to $Z^1(\overline{\cL}^*,\overline{G})$. Clearly, the mapping $\overline{\gamma}$ can be defined also when $\ch(\FF) = 2$ and, according to \cite{CGP2}, what we have just said on it holds true regardless of $\ch(\FF)$. 

With $\gamma:\Der(\FF)\rightarrow Z^1(\cL^*,G)$ as in Subsection \ref{gamma BD}, let $\frac{1}{2}\gamma$ be the function that maps every $d\in \Der(\FF)$ onto the cocycle $\frac{1}{2}\gamma_d$. In short, $\frac{1}{2}\gamma$ is defined just as $\gamma$ but with $T$ replaced by $\frac{1}{2}T$. 

\begin{lemma}\label{A3D3odd lem}
Let $\ch(\FF) \neq 2$. Then $\overline{\gamma} = \phi_Z\circ \frac{1}{2}\gamma$.
\end{lemma}
\noindent
{\em Proof.} Straightforword computations show that for every $d\in \Der(\FF)$ we have 
\begin{equation}\label{A3D3oddT}
\Tr(\bar{g}^{-1}d(\bar{g})\ovx) ~ = ~  \frac{1}{2}\Tr(\phi_G(\bar{g})^{-1}d(\phi_G(\bar{g}))\phi_\cL(\ovx)). 
\end{equation} 
for every $\bar{g}\in \overline{U} =  \{\ovu_{i,j}(t) ~|~ i\neq j, ~ t\in \FF\}$ and every $\ovx\in\overline{\cL}$. However $\overline{U}$ generates $\overline{G}$ and $\phi_G(\overline{U})$ generates $G$. Moreover, both $\overline{\gamma}_d$ and $\frac{1}{2}\gamma_d$ are cocycles (of $Z^1(\overline{\cL}^*,\overline{G})$ and  $Z^1(\cL^*,G)$ respectively). Therefore 
equation (\ref{A3D3oddT}) holds for every $\bar{g}\in\overline{G}$ and every $\ovx\in\overline{\cL}$.   \hfill  $\Box$ 

\begin{corollary}\label{A3D3odd cor}
Let $\ch(\FF) \neq 2$. Then $\gamma$ isomorphically maps $\Der(\FF)$ onto a complement of $B^1(\cL^*,G)$ in $Z^1(\cL^*,G)$.
\end{corollary} 
\noindent
{\em Proof.} This follows from Lemma \ref{A3D3odd lem} and the fact that, as recalled above, $\overline{\gamma}$ isomorphically maps $\Der(\FF)$ onto a complement of $B^1(\overline{\cL}^*,\overline{G})$ in  $Z^1(\overline{\cL}^*,\overline{G})$. \hfill $\Box$ 

\subsubsection{The characteristic $2$ case} 

Let now $\ch(\FF) = 2$. Then $\overline{Z} = 1$, hence $\overline{G} = G$. We can still define $\phi_\cL$ and $\phi_G$ as in (\ref{A3D3oddL}) and (\ref{A3D3oddG}). The mapping $\phi_G$ is still an isomorphism from $\overline{G}$ to $G$. In contrast, $\phi_\cL$ is a surjective morphism from $\overline{\cL}$ to $\cL_0$ with $\ker(\phi_\cL) = \overline{K} := Z(\overline{\cL}) = \langle I_4\rangle$. Hence $\phi_\cL$ induces an isomorphism from $\overline{\cL}_0 := \overline{\cL}/\overline{K}$ to $\cL_0$. By a minor abuse, we also denote this isomorphism by $\phi_\cL$. With this notation, $\phi := (\phi_\cL, \phi_G)$ is an isomorphism of modules from $(\overline{\cL}_0,\overline{G})$ to $(\cL_0, G)$ and induces an ismorphism $\phi_Z$ from $Z^1(\cL^*_0,G)$ to $Z^1(\overline{\cL}^*_0,\overline{G})$.

We have $f(I_4) = 0$ for every linear functional $f\in\overline{\cL}^*$. Accordingly, $Z^1(\overline{\cL}^*_0,\overline{G})$ is essentially the same as $Z^1(\overline{\cL}^*,\overline{G})$. 

Let $\overline{\cL}_{\overline{K}}$ be the subspace of $\overline{\cL}$ spanned by $\{\ove_{i,j}\}_{i\neq j}\cup\{\ove_{3,3}-\ove_{2,2}, \ove_{4,4}-\ove_{3,3}\}$. Then $\overline{\cL}_{\overline{K}}$ is a complement of $\overline{K}$ in the vector space $\overline{\cL}$. A right action of $\overline{G}$ on $\overline{\cL}_{\overline{K}}$ can be defined, isomorphic to the natural right action of $\overline{G}$ on $\overline{\cL}_0$; explicitly, for every $\bar{g}\in\overline{G}$ and every $\ovx\in\overline{\cL}_{\overline{K}}$ we put $\ovx\cdot\bar{g}$ equal to the projection of $\bar{g}^{-1}\ovx\bar{g}$ (computed in $\overline{\cL}$) onto $\overline{\cL}_{\overline{K}}$ along $\overline{K}$. Thus we can regard $\overline{L}$ as a central extension of $\overline{L}_{\overline{K}}$ by the $1$-dimensional space $\overline{K}$. According to \cite[Theorem 2.6]{CGP2}, this extension can be constructed by means of a cocycle $\overline{F}_0$ of $\overline{G}$ over the dual $\overline{\cL}^*_{\overline{K}}$ of $\overline{\cL}_{\overline{K}}$.  

Explicitly, $\overline{F}_0$ can be described as follows. For every $\ovx\in\overline{\cL}$ let $p(\ovx)I_4$ be the projection of $\ovx$ onto $\overline{K}$ along $\overline{\cL}_{\overline{K}}$. Then the mapping $p:\ovx\rightarrow p(\ovx)$ is a linear functional of $\overline{\cL}$. The cocycle $\overline{F}_0$ maps every $\bar{g}\in\overline{G}$ onto the linear functional of $\overline{\cL}_{\overline{K}}$ which maps every $\ovx\in\overline{\cL}_{\overline{K}}$ onto $p(\bar{g}\ovx\bar{g}^{-1})$, the element $\bar{g}\ovx\bar{g}^{-1}$ being computed in $\overline{\cL}$. Note that in $\overline{\cL}$ we have $p(\bar{g}\ovx\bar{g}^{-1}) = p(\ovx) - p(\bar{g}\ovx\bar{g}^{-1}) = (p-p\cdot\bar{g})(\ovx)$. Therefore $\overline{F}_0$, regarded a cocycle of $\overline{G}$ over $\overline{\cL}^*$, is a coboundary. However $\overline{F}_0$ is not a coboundary of $\overline{G}$ over $\overline{\cL}^*_{\overline{K}}$, since the extension $\overline{\cL} = \overline{K}\cdot\overline{\cL}_{\overline{K}}$ ($\cong \overline{K}\cdot\overline{\cL}_0$) is non-split. We call $\overline{F}_0$ the {\em seed cocycle} of $\overline{G}$ over $\overline{\cL}^*_0$ (where we freely identify $\overline{\cL}^*_0$ with $\overline{\cL}^*_{\overline{K}}$). 

The mapping $\overline{\gamma}$, still defined as in (\ref{ovgamma}), maps $\Der(\FF)$ isomorphically onto a complement of $B^1(\overline{\cL}^*,\overline{G})$ in $Z^1(\overline{\cL}^*,\overline{G})$ (see \cite[Appendix C]{CGP2}). On the other hand,  $Z^1(\overline{\cL}^*,\overline{G}) = Z^1(\overline{\cL}^*_0,\overline{G})$ but $B^1(\overline{\cL}^*,\overline{G}) = B^1(\overline{\cL}^*_0,\overline{G})\oplus\langle\overline{F}_0\rangle$.  Therefore, 

\begin{lemma}\label{A3D3even lem1}
Let $\ch(\FF) =  2$. Then the injective linear mapping $\overline{\gamma}:\Der(\FF)\rightarrow Z^1(\overline{\cL}^*_0,\overline{G})$ maps $\Der(\FF)$ onto a complement of $B^1(\overline{\cL}^*_0,\overline{G})\oplus\langle\overline{F}_0\rangle$ in $Z^1(\overline{\cL}^*_0,\overline{G})$. 
\end{lemma}

As in the case where $\ch(\FF)\neq 2$, the isomorphism $\phi = (\phi_\cL,\phi_G)$ from $(\overline{\cL}_0,\overline{G})$ to $(\cL_0,G)$ induces an isomorphism $\phi_Z$ from $Z^1(\cL^*_0, G)$ to $Z^1(\overline{\cL}^*_0,\overline{G})$.

\begin{lemma}\label{A3D3even lem2}
Let $\ch(\FF) = 2$. Then $\overline{\gamma} = \phi_Z\circ \gamma$ and $\phi_Z$ maps the seed cocycle $F_0$ of $Z^1(\cL^*_0, G)$ onto the seed cocycle $\overline{F}_0$ of $Z^1(\overline{\cL}^*_0, \overline{G})$.   
\end{lemma}
\noindent
{\em Proof.}  Formula (\ref{A3D3oddT}) holds true even if $\ch(\FF) = 2$ but for replacing $\frac{1}{2}\Tr(\phi_G(\bar{g})^{-1}d(\phi_G(\bar{g}))\phi_\cL(\ovx))$ with $\tau_\cM(\phi_G(\bar{g})^{-1}d(\phi_G(\bar{g})), \phi_\cL(\ovx))$, where $\tau_\cM$ is  defined as in Subsection \ref{tau 1}. This implies that $\overline{\gamma} = \phi_Z\circ \gamma$. Turning to the equality $\overline{F}_0 = \phi_Z(F_0)$, by the description of $\overline{F}_0$ given above we obtain that for every $\ovx\in \overline{\cL}_{\overline{K}}$ we have $\overline{F}_0(\bar{g}) = 0$ for every $\bar{g}\in \overline{U} =  \{\ovu_{i,j}(t) ~|~ i\neq j, ~ t\in \FF\}$ and $\overline{F}_0(\bar{g}) = \Tr(\bar{g}\ovx)$ for every $\bar{g}$ equal to $\ovu_{i,1}(t)$ or $\ovu_{1i}(t)$, wih $i > 1$. Comparing these results with those reported for $F_0$ in Lemma \ref{F0} (to be stated in Section \ref{F0 sec}) with the help of formula (\ref{A3D3oddT}) we see that $\overline{F}_0(\bar{g})(\ovx) = F_0(\phi_G(\bar{g}))(\phi_\cL(\ovx))$ for every $\bar{g}\in\overline{U}$ and every $\ovx\in\overline{\cL}_{\overline{K}}$. However $\overline{U}$ generates $\overline{G}$ and $\phi_G(\overline{U})$ generates $G$. Hence $\overline{F}_0(\bar{g})(\ovx) = F_0(\phi_G(\bar{g}))(\phi_\cL(\ovx))$ for every $\bar{g}\in\overline{G}$ and every $\ovx\in\overline{\cL}_{\overline{K}}$, namely $\phi_Z(F_0) = \overline{F}_0$, as claimed.   \hfill $\Box$ \\   

Lemmas \ref{A3D3even lem1} and \ref{A3D3even lem2} imply the following: 

\begin{corollary}\label{A3D3even cor}
Let $\ch(\FF) = 2$. Then $\gamma$ isomorphically maps $\Der(\FF)$ onto a complement of $B^1(\cL^*_0,G)\oplus\langle F_0\rangle$ in $Z^1(\cL^*_0,G)$.
\end{corollary} 

So, the conclusions of Proposition \ref{Main D} hold true when $n = 3$ too. 

\subsection{Proof of Theorem \ref{even main1}; the $B_n$ case}\label{main B}  

Keeping the assumption that $\ch(\FF) = 2$, throughout this section $G = \Omega(2n+1,\FF)$ ($n \geq 3$) and $\cL = \mathfrak{o}(2n+1,\FF)$. Moreover $K = Z(\cL)$ and $\cL_0 \cong \cL/K$ is the canonical complement of $K$ in $\cL$ as described in Subsection \ref{L0 B} and regarded as a subalgebra of $\cM_\cL$. Also $\chi = h_n$ is the canonical generator of $K$ and $F_0$ is the seed cocycle of $G$ over $\cL^*_0$. 

We denote by $ G^0, G^+, G^-, R^+$ and $R^-$ the subgroups of $G$ generated by the sets $\cup_{k \neq h}U_{k,h}$, $\cup_{k < h}U^+_{k,h}$, $\cup_{k < h}U^-_{k,h}$, $\cup_{k=1}^nU^+_k$ and respectively $\cup_{k=1}^nU^-_k$ where $U_{k,h}$ and  $U^\pm_{k,h}$ are as in (\ref{Usbgr}) and $U^\pm_k := \{u^\pm_k(t)\}_{t\in \FF}$. So, the set $G^0\cup G^+\cup G^-$ generates the subgroup $G_D = \Omega^+(2n,\FF)$ of $G = G_B$ (notation as in Section \ref{DB sec}). 

As in Section \ref{DB sec}, we denote by $\cL_{D,0}$ the subspace of $\cL_0$ spanned by the elements of $\cL$ corresponding to the elements $x_\alpha$ of $L$ for $\alpha$ a long root and the elements $h_1, h_2,..., h_{n-1}$ of the Cartan subalgebra of $\cL$. So, $G_D$ stabilizes $\cL_{D,0}$. 

By Lemma \ref{DB K cor}, the seed cocycle $F_0$ of $(\cL_0^*,G)$ induces on $G_D$ the seed cocycle of $(\cL_{D,0}^*, G_D)$. Morover, every coboundary of $G_D$ over $\cL^*_{D,0}$ extends (in more than one way) to a coboundary of $G$ over $\cL^*_0$. According to Proposition \ref{Main D}, there exists a unique derivation $d_F\in \Der(\FF)$, a coboundary $F_1\in B^1(\cL^*_0, G)$ and a unique scalar $\lambda\in \FF$ such that 
\[(F-\gamma_{d_F}-F_1-\lambda F_0)(g)(x) ~ = ~ 0, \hspace{5mm}\forall  g\in G_D, ~ \forall x\in \cL_{D,0}.\]
So, up to replacing $F$ with $F-\gamma_{d_F}-F_1-\lambda F_0$, we can focus on a cocycle $F$ such that $F(g)(x) = 0$ for every $g\in G_D$ and every $x\in \cL_{D,0}$; in short, $F$ is {\em null over $(\cL_{D,0}, G_D)$}. Theorem \ref{even main1} is equivalent to the following. 

\begin{proposition}\label{Main B}
Let $F$ be null over $(\cL_{D,0},G_D)$. Then $F$ is a coboundary, except perhaps when $n = 3$ and $\FF = \FF_2$. 
\end{proposition}

This subsection is devoted to the proof of Proposition \ref{Main B}.  

The isomorphism $\cL^*_0\cong \cM_\cL/\langle I_D\rangle$ established via the form $\widehat{\tau}_\cM$ introduced in Section \ref{L*0 sec} provides a representative matrix $m_{F,g}$ of $F(g)$ for every $g\in G$ and every $F\in Z^1(\cL^*_0,G)$, namely $m_{F,g}\in\cM_\cL$ is such that $F(g)(x) = \widehat{\tau}_\cM(m_{F,g}, x)$ for every $x\in \cL_0$. The matrix $m_{F,g}$ is uniquely determined by $F$ and $g$ modulo $\langle I_D\rangle$. 

In general $\tau_\cM(m_{F,g}, x) \neq \widehat{\tau}_\cM(m_{F,g}, x)$, however equality holds when the last column of $m_{F,g}$ is null.  In particular, if $d\in \Der(\FF)$ and $F = \gamma_d$ then $m_{F,g} = g^{-1}d(g)$ for every $g\in G_D$, but when $g\in G\setminus G_D$ then in general $m_{F,g} \neq g^{-1}d(g)$. For istance, let $g\in R^+\cup R^-$. Then $m_{F,g} = O$ (null matrix) but if $\Der(\FF)\neq\{0\}$ then the matrix $g^{-1}d(g)$ need not be null. 

Moreover, let $g_1, g_2\in G$. Then $F(g_1)\cdot g_2$ is represented by the matrix $m_{F,g_1}{\ast} g_2$ with ${\ast}$ defined as in Section \ref{L*0 sec}.  If $g_2\in G_D$ then $m_{F,g_1}{\ast} g_2 = m_{F,g_1}\cdot g_2 = g_2^{-1}m_{F,g_1}g_2$ but if $g_2\in G\setminus G_D$ then in general $m_{F,g_1}{\ast} g_2 \neq  g_2^{-1}m_{F,g_1}g_2$. In general computing the difference $\Delta(m_{F,g_1}, g_2) = m_{F,g_1}{\ast} g_2- g_2^{-1}m_{F,g_1}g_2$ can be rather laborious.  However in all cases to be considered here this computation is an easy task. 

Henceforth $F$ is a given cocycle null over $(\cL_{D,0},G_D)$. Since in the sequel we shall always deal with that single given cocycle, there is no need of keeping a record of $F$ in our notation for representative matrices. So, we will always write $m_g$ instead of $m_{F,g}$. Recall that $m_g$ is uniquely determined up to multiples of $I_D$. As $F$ is null over $(\cL_{D,0},G_D)$, we can always assume that when $g\in G_D$ the matrix $m_g$ is given the following {\em canonical form}, for suitable vectors $\ba_g, \bfb_g \in V(n,\FF)$ (throughout this subsection we use bold characters to denote vectors of $V(n,\FF)$):  
\begin{equation}\label{B eq1}
m_g ~ = ~ \left(\begin{array}{ccc}
O_n & O_n & \ba_g^t \\
O_n & O_n & \bfb_g^t \\
{\bf 0} & {\bf 0} & 0
\end{array}\right).
\end{equation} 
Let $g_1 =  \left(\begin{array}{cc}
g_{1,D} & {\bf 0}^t_{2n}\\
{\bf 0}_{2n} & 1
\end{array}\right)$ 
be another element of $G_D$. Then  
\begin{equation}\label{B eq2}
g_1^{-1}m_gg_1 ~ = ~ \left(\begin{array}{cc}
O_{2n} & g_{1,D}^{-1}\left(\begin{array}{c}
\ba_g^t\\
\bfb_g^t 
\end{array}\right) \\
{\bf 0}_{2n} & 1
\end{array}\right).
\end{equation} 
By (\ref{B eq2}) and the cocycle condition, if $g$ is an involution then 
\begin{equation}\label{B eq2 inv}
g_{D}^{-1}\left(\begin{array}{c}
\ba_g^t\\
\bfb_g^t 
\end{array}\right) ~ = ~\left(\begin{array}{c}
\ba_g^t\\
\bfb_g^t 
\end{array}\right).
\end{equation}
If $g_1g_2 = g_2g_1$ then 
\begin{equation}\label{B eq2 comm}
(g_{2,D}^{-1}+I_D)\left(\begin{array}{c}
\ba_{g_1}^t\\
\bfb_{g_1}^t 
\end{array}\right) ~ = ~ (g_{1,D}^{-1}+I_D)\left(\begin{array}{c}
\ba_{g_2}^t\\
\bfb_{g_2}^t 
\end{array}\right). 
\end{equation}

\subsubsection{$F(g)$ for $g \in G^0$}

Let $g = u_{k,h}(t)$ and, with the notation stated above, let $\ba_g := (a_1(t),..., a_n(t))$ and $\bfb_g = (b_1(t),..., b_n(t))$. The element $g$ is an involution. Hence (\ref{B eq2 inv}) yields 
\begin{equation}\label{B eq3 1}
a_h(t) ~ = ~ b_k(t) ~ = ~  0.
\end{equation}
For $p\not\in\{k,h\}$, let $g' = u_{p,h}(s)$ and $g'' = u_{k,p}(s)$, with 
\[\ba_{g'} = (a'_1(s),..., a'_n(s)), ~\bfb_{g'} = (b'_1(s),...,b'_n(s)), ~ \ba_{g''} = (a''_1,..., a''_n), ~\bfb_{g''} = (b''_1(s),...,b''_n(s)).\]
Then $g$ commutes with both $g'$ and $g''$. By (\ref{B eq2 comm}) we get that $sa_p(t) = ta''_h(s)$ and $sb_p(t) = tb'_k(s)$. Therefore $a_p(t) = ta''_h(1) ) = ta_p(1)$ and $b_p(t) = tb'_k(1) = tb_p(1)$. As $p$ is an arbitrary index different from $k$ and $h$, we obtain that there exists scalars $\alpha_k, \beta_h\in \FF$ such that 
\begin{equation}\label{B eq3 2}
\ba_g ~=~ a_k(t)\be_k + t\alpha_k\sum_{i\neq k,h}\be_i,~~~ \bfb_g ~=~ b_h(t)\be_h + t\beta_h\sum_{i\neq k,h}\be_i.
\end{equation}
(Here $\{\be_1,..., \be_n\}$ is the natural basis of $V(n,\FF)$.)  Note that $\alpha_k$ does not depend on $h$ and $\beta_h$ does not depend on $k$.  Of course, (\ref{B eq3 2}) holds for any choice of distinct indices $k$ and $h$.  

With $g = u_{k,h}(t)$ as above, let $g' = u_{h,p}(s)$ and $g'' = u_{k,p}(st)$. According to (\ref{B eq3 2}) we have
\[\begin{array}{ll}
\ba_{g'} ~=~ a'_h(s)\be_h + s\alpha_h\sum_{i\neq h,p}\be_i, & \bfb_{g'} ~=~ b_p(s)\be_p + s\beta_p\sum_{i\neq h,p}\be_i,\\
\ba_{g''} ~=~ a''_k(st)\be_k + st\alpha_k\sum_{i\neq k,p}\be_i, & \bfb_{g''} ~=~ b_p(st)\be_p + st\beta_p\sum_{i\neq k,p}\be_i.
\end{array}\]
Recall that $gg'gg' = gg'g^{-1}g'^{-1} = g''$. We can compute $F(g'') = F(gg'gg')$ by the cocycle condition on $F$. By the above descriptions of $\ba_g, \ba_{g'}$ and $\ba_{g''}$ and formula (\ref{B eq2}) we obtain that 
\begin{equation}\label{B eq3 3}
(st^2\alpha_k+ta'_h(s))\be_k + st\alpha_k\be_h ~ = ~ a''_k(st)\be_k + st\alpha_k\sum_{i\neq k, p}\be_i.
\end{equation}
Suppose that $n > 3$. Then at the right side of (\ref{B eq3 3}) occur summands of the form $st\alpha_ke_i$ with $i \not\in\{k,h\}$. Nothing in the left side of (\ref{B eq3 3}) corresponds to them. It follows that $\alpha_k = 0$. Consequently $ta'_h(s) = a''_k(st)$. From the latter with $t = 1$ we obtain that $a'_h(s) = a''_k(s)$ while with $s = 1$ we obtain that $a''_k(t) = ta'_h(1)$. Finally, $a'_h(s) = sa'_h(1)$, $a''_k(t) = ta''_k(1)$ and $a'_h(1) = a''_k(1)$. Similarly, $\beta_p = 0$, $b''_p(s) = sb''_p(1)$, $b_h(t) = tb_h(1)$ and $b''_p(1) = b_h(1)$. 

Let now $n = 3$. We shall prove that $\alpha_k = 0$ in this case too, but the proof is not so easy as in previous case. Formula (\ref{B eq3 3}) implies that 
\begin{equation}\label{B eq3 3bis}
st^2\alpha_k + ta'_h(s) ~ = ~ a''_k(st).
\end{equation}
However, if we replace the triplet $(k,h,p)$ with $(h,k,p)$ and put $g''' = u_{h,k}(t)$ and $\ba_{g'''} = a'''_h(r)\be_h + r\alpha_h\sum_{i\neq h,k}\be_i$, then we get $\alpha'_h(st) = st^2\alpha_h + ta''_k(s)$. Therefore $\alpha'_h(t) = t^2\alpha_h + ta''_k(1)$. Equivalently, $\alpha'_h(s) = s^2\alpha_h + sa''_k(1)$. Replacing $a'_h(s)$ in (\ref{B eq3 3bis}) with $s^2\alpha_h + sa''_k(1)$ we obtain
\begin{equation}\label{B eq3 3ter}
a''_k(st) ~ = ~ st^2\alpha_k + ts^2\alpha_h + tsa''_k(1).
\end{equation}
However from (\ref{B eq3 3bis}) with $s = 1$ we also have $a''_k(t) = t^2\alpha_k + ta'_h(1)$. Therefore
\[a''_k(st) ~ =  ~ s^2t^2\alpha_k + sta'_h(1).\]
Comparing this relation with (\ref{B eq3 3ter}) we get $st^2\alpha_k + ts^2\alpha_h + tsa''_k(1) =  s^2t^2\alpha_k + sta'_h(1)$, which, divided by $st$, yields
\[t\alpha_k + s\alpha_h  + st\alpha_k ~ = ~ a''_k(1)+ a'_h(1).\]
This equality holds for every choice of $s, t \in \FF\setminus\{0\}$. If $\FF \neq \FF_2$ this can be only if $\alpha_k = \alpha_h = 0$ and $a''_k(1) = a'_h(1)$. Having proved this, we can continue as in the case of $n > 3$. 

Thus we have proved that, if either $|\FF| > 2$ or $n > 3$ then there exist scalars $a_1,..., a_n, b_1,...., b_n$ such that if $g = u_{k,h}(t)$ then $\ba_g = ta_h\be_k$ and $\bfb_g = tb_k\be_h$. This shows that $F(g)$ admits the same representative matrix as $f-f\cdot g$, where $f$ is the linear functional represented be the matrix 
\begin{equation}\label{B eq3 4}
m_f ~ = ~ \left(\begin{array}{ccc}
O_n & O_n & \ba^t \\
O_n & O_n & \bfb^t \\
{\bf 0} & {\bf 0} & 0
\end{array}\right).
\end{equation} 
where $\ba =(a_1, a_2,..., a_n)$ and $\bfb = (b_1, b_2,..., b_n)$. As the elements $u_{k,h}(t)$ generate $G^0$, the same conclusion holds for every $g\in G$. So far, we have proved the following. 

\begin{lemma}\label{B lem1}
Let $F$ be null over $(\cL_{D,0},G_D)$ and assume that either $n > 3$ or $|\FF| > 2$. Then there exists a linear functional $f\in \cL_0^*$ with canonical representative matrix as in {\rm (\ref{B eq3 4})} such that $F(g) = f-f\cdot g$ for every $g\in G^0$.
\end{lemma}   

Let $F'$ be the coboundary that maps every $g\in G$ onto the linear functional $f - f\cdot g$, with $f$ as in the conclusion of Proposition \ref{B lem1}. By (\ref{B eq2}) and the fact that all elements of $G_D$ preserve the form $\widehat{\tau}_\cM$, we have $F'(g)(x) = 0$ for every $g\in G_D$ ed every $x\in \cL{D,0}$. Therefore, up to replacing $F$ with $F-F'$ we can assume that $F$ is still null over $(\cL_{D,0},G_D)$ and, morever, $F(g) = 0$ for every $g\in G^0$; say $F$ is {\em null over} $G^0$. Proving the following is our next step.

\begin{proposition}\label{Main B bis}
Let $F$ be null over both $(\cL_{D,0},G_D)$ and $G^0$. Then $F = 0$ except perhaps when $n = 3$ ad $\FF= \FF_2$. 
\end{proposition}

Proposition \ref{Main B} immediately follows from Proposition \ref{Main B bis} and Lemma \ref{B lem1}.    

\subsubsection{$F(g)$ for $g\in G^\pm$}

Througout this subsection $F$ is assumed to be null over both $(\cL_{D,0}, G_D)$ and $G^0$. We shall prove that $F$ is {\em null over} $G_D$, namely $F(g) = 0$ for every $g\in G_D$. 

Before turning to $G^+$ and $G^-$ we shall prove that $F$ is null also on the normalizer $N(G^0)$ of $G^0$ in $G_D$ (which coincides with the normalizer of $G^0$ in $G$). 

\begin{lemma}\label{B diag}
If $g \in G_D$ normalizes $G^0$ then $F(g) = 0$.
\end{lemma}
\noindent
{\em Proof.} Let $g\in N(g^0)$. Then for every $g_0\in G^0$ there exist $g'_0 \in G^0$ such that $g'_0g = gg_0$ and the coycyle condition on $F$ implies that $F(g'_0)\cdot g + F(g) = F(g)\cdot g_0 + F(g_0)$. However $F(g'_0) = F(g_0) = 0$. Therefore $F(g) = F(g)\cdot g_0$ for every $g_0\in G^0$. Combined with (\ref{B eq2}) this equality implies that $X^{-1}\ba_g^t = X^t\bfb_g^t = {\bf 0}^t$ for every $X\in \mathrm {SL}(n,\FF)$.  This forces $\ba_g =\bfb_g = {\bf 0}$.  \hfill $\Box$\\

Let $g = u_{[k,h]}^+(t)$ and let $m_g$ be the canonical representative matrix of $F(g)$. With $\ba_g$ and $\bfb_g$ as in (\ref{B eq1}), we put 
\[\ba^+_{k,h}(t) = (a^+_{k,h;1}(t),..., a^+_{k,h;n}(t)) := \ba_g, \hspace{5 mm} \bfb^+_{k,h}(t) = (b^+_{k,h;1}(t),..., b^+_{k,h;n}(t)) := \bfb_g.\]
We adopt a similar notation for $u^-_{k,h}(t)$.  

\begin{lemma}\label{B lem2}
The cocycle $F$ is null over $G_D$ except perhaps when $n = 3$ and $\FF = \FF_2$.
\end{lemma}
\noindent
{\em Proof.} As $u^+_{k,h}(t)$ is an involution, (\ref{B eq2 inv}) implies that 
\begin{equation}\label{B eq4}
a^+_{k,h;k}(t) ~= ~a^+_{k,h;h}(t)~ = ~0. 
\end{equation} 
Suppose $n > 3$. Let $1\leq p < q\leq n$ be such that $\{p,q\}\cap\{k,h\} = \emptyset$ (such a pair $\{p,q\}$ exists because $n > 3$). As $u^+_{k,h}(t)$ commutes with $u^+_{p,q}(s)$ formula (\ref{B eq2 comm}) implies that $a^+_{k,h;p}(t) = a^+_{k,h;q}(t) = 0$. However $p$ and $q$ are arbitrary indices both different from either of $k$ and $h$ (and $p <  q$). Therefore $\ba^+_{k,h}(t) = {\bf 0}$. Similarly, if $g = u^-_{k,h}(t)$ then $\bfb^-_{k,h}(t) = {\bf 0}$. 

Let now $k, h, p$ be pairwise distinct indices with $k < h$. For arbitrary scalars $s, t\in \FF$ put $g_0 = u_{p,k}(t)$, $g' = u_{p,h}^+(s)$ or $g' = u_{h,p}^+(s)$ according to whether $p < h$ or $p > h$ and  $g = u^+_{k,h}(st)$. We have $g_0g'g_0g' = g_0g'g_0^{-1}g'^{-1} = g$. We can compute $F(g) = F(g_0g'g_0g')$ by the cocyle condition on $F$ and formula (\ref{B eq2}). Recalling $F(g_0) = 0$ by assumption (since $g_0\in G^0$) we obtain that 
\begin{equation}\label{B eq4bis}
\bfb^+_{k,h}(st) ~ = ~  tb^+_{p',h';p}(s)e_k, ~~ \mbox{ where } p' = \min(p,h) \mbox{ and } h' = \max(p,h). 
\end{equation}
So, $b^+_{k,h;i}(s) = 0$ for every $i\neq k$ and $b^+_{k,h;k}(st) = tb^+_{p',h';,p}(s)$. The latter implies that $b^+_{k,h;k}(1) = b^+_{p',h';p}(1)$, $b^+_{k,h;k}(t) = tb^+_{k,h;k}(1)$ and $b^+_{p',h';p}(s) = sb^+_{k,h;k}(1) = sb^+_{p',h';p}(1)$. Thus, all entries of $\bfb^+_{k,h}(st)$ are null but possibly the $k$-th one, which is equal to $stb^+_{p',h';p}(1)$. Similarly, all entries of $\bfb^+_{p',h'}(s)$ are null but possibly the $p'$-th. However if $h < n$ then, given $k < h$, nothing forbids to choose $p > h$ in the above argument. With this choice of $p$ we have $p' = h$ and $h' = p$. Consequently $b^+_{p',h';p}(s) = 0$. Therefore $b^+_{k,h;k}(st) = 0$. So, $\bfb^+_{k,h}(st) = {\bf 0}$.  

So far we have proved that $\bfb^+_{k,h}(t) = {\bf 0}$ for every choice of $k < h < n$ and $t\in \FF$. Let $h = n$. Then in (\ref{B eq4bis}) we have $p < n$. By the above, $b^+_{p,n;p}(s) = 0$ and (\ref{B eq4bis}) yields $\bfb^+_{k,n}(st) = {\bf 0}$. So, when $n > 3$ we have $F(u^+_{k,h}(t)) = 0$ for every choice of $k < h$ (and every $t\in \FF$). In the same way we can prove that $F(u^+_{k,h}(t)) = 0$ for every choice of $k < h$. However, $G^+$ and $G^-$ are generated by elements as $u^+_{k,h}(t)$ and $u^-_{k,h}(t)$ respectively. Moreover $G^0\cup C^+\cup G^-$ generates $G_D$. Therefore $F(g) = 0$ for every $g\in G_D$, namely $F$ is null over $G_D$. 

Let now $n = 3$. For $\{k,h,p\} = \{1,2,3\}$ with $k < h$, if $g = u^+_{kh}(t)$ then $\ba_g = a_{k,h}(t)\be_p$ for a scalar $a_{k,h}(t)$. By the commutativity of $G^+$ we get that there exists a scalar $a\in \FF$ such that $\ba^+_{k,h}(t) = ta\be_p$ for every choice of $k < h$ and $t\in \FF$. The argument we have used in the case $n > 3$ to show that $\bfb_g = {\bf 0}$ works when $n = 3$ as well. So, we also have $\bfb^+_{k,h}(t) = {\bf 0}$ for every choice of $k< h$ and $t\in \FF$. Similarly, there exists a scalar $b\in \FF$ such that $\bfb^-_{k,h}(t) = tb\be_p$ for every choice of $k< h$ and $t\in \FF$ and $\ba^+_{k,h}(t) = {\bf 0}$ for every choice of $k < h$ and $t\in \FF$. 

We shall now prove that if $\FF \neq \FF_2$ then $a = b = 0$ (hence $F(g) = 0$ for every $g\in G^+\cup G^-$). Given $t\in \FF\setminus\{0\}$, put $w^+_{1,2}(t) := u^+_{1,2}(t)u^-_{1,2}(t^{-1})u^+_{1,2}(t)$ and $h^+_{1,2}(t) := w^+_{1,2}(t)w^+_{1,2}(1)$. The matrix $h^+_{1,2}(t)$ is diagonal. Hence it belongs to $N(G^0)$. By Lemma \ref{B diag} we have $F(h^+_{1,2}(t)) = 0$. However the cocycle condition on $F$ allows to compute 
\[F(h^+_{1,2}(t)) ~=~  F(u^+_{1,2}(t)u^-_{1,2}(t^{-1})u^+_{1,2}(t)u^+_{1,2}(1)u^-_{1,2}(1)u^+_{1,2}(1))\] 
as a function of $a, b$ and $t$. We obtain that $F(h^+_{1,2}(t)) = 0$ only if $(t^{-1}+1)b = 0$. However if $\FF\neq \FF_2$ we can choose $t\in \FF\setminus\{0,1\}$. With $t$ chosen in this way we have $(t^{-1}+1)b = 0$ only if $b = 0$. Similarly, by considering $h^-_{1,2}(t) = w^-_{1,2}(t)w^-_{1,2}(1)$ with $w^-_{1,2}(t) := u^-_{1,2}(t)u^-_{1,2}(t^{-1})u^-_{1,2}(t)$, we obtain that $a = 0$.   \hfill $\Box$ 

\subsubsection{$F(g)$ with $g \in R^\pm$}  

In view of Lemma \ref{B lem2} we can assume that $F$ is null over $G_D$. For $\eta\in\{+.-\}$, let $m^\eta_k(t)$ be a representative matrix of $F(u^\eta_k(t))$, say  
\[m^\eta_k(t) ~ = ~ \left(\begin{array}{ccc}
A^\eta_k(t) & B^\eta_k(t) & \ba^{\eta t}_k(t) \\
C^\eta_k(t) & A^{\eta t}_k(t) & \bfb^{\eta t}_k(t) \\
{\bf 0} & {\bf 0} & 0
\end{array}\right).\]
The element $u^\eta_k(t)$ commutes with every element of the subgroup $G_{D,k}$ of $G_D$ generated by $(\cup_{p,q\neq k, ~ p\neq q}U_{p,q})\cup(\cup_{p,q\neq k, ~  p< q}(U^+_{p,q}\cup U^-_{p,q}))$. By assumption $F(g) =0$ for every $g\in G$. Therefore 
$g^{-1}m^\eta_k(t)g \equiv  m^\eta_k(t) ~ (\mathrm{mod} \langle I_D\rangle)$ for every $g\in G_{D,k}$. This relation implies that 
\begin{equation}\label{R eq1}
\left\{\begin{array}{l}
A^\eta_k(t) = \sum_{i = 1}^na^\eta_{k; i,k}(t)e_{i,k}+\sum_{j\neq k}a^\eta_{k; k,j}(t)e_{k,j} + \lambda^\eta_k(t)I_n, \\
B^\eta_k(t) = \sum_{i \neq k}b^\eta_{k,i,k}(t)(e_{i,k}+e_{k,i}), ~~ C^\eta_k(t) = \sum_{i \neq k}c^\eta_{k,i,k}(t)(e_{i,k}+e_{k,i}),\\
\ba^{\eta}_k(t) = \alpha^\eta_k(t)\be_k, ~~  \bfb^{\eta}_k(t) = \beta^\eta_k(t)\be_k,
\end{array}\right.
\end{equation} 
for suitable scalars $a^\eta_{k; i,k}(t), a^\eta_{k; k,j}(t), b^\eta_{k,i,k}(t), c^\eta_{k,i,k}(t), \alpha^\eta_k(t), \beta^\eta_k(t)$ and $\lambda^\eta_k(t)$. However $m^\eta_k(t)$ is determined up to multiples of $I_D$. So, we can safely assume that $\lambda^\eta_k(t) =0$, thus obtaining a {\em canonical} representative matrix for $F(h^\eta_k(t))$. 

Since $u^\eta_k(t)$ is an involution, we have $m^\eta_k(t){\ast}u^\eta_k(t) = m^\eta_k(t)$. The matrix $m^\eta_k(t){\ast}u^\eta_k(t) = u^\eta_k(t)m^\eta_k(t) + \Delta(m^\eta_k(t), u^\eta_k(t))$ can be computed with the help of formula (\ref{L*02}). In this way from the equality $m^\eta_k(t){\ast}u^\eta_k(t) = m^\eta_k(t)$ and in view of (\ref{R eq1}) (with $\lambda^\eta_k(t) = 0$) we obtain the following.
\begin{equation}\label{R eq2}
\left\{\begin{array}{lll}
B^+_k(t) = O_n, & A^+_k(t) = \sum_{i=1}^na^+_{k;i,k}(t)e_{i,k}, & C^+_k(t), ~ \ba^+_k(t) \mbox{ and } \bfb^+_k(t) \mbox{ as in (\ref{R eq1})};\\
C^-_k(t) = O_n;  & A^-_k(t) = \sum_{j=1}^na^+_{k;k,j}(t)e_{k,j}, & B^-_k(t), ~ \ba^-_k(t) \mbox{ and } \bfb^-_k(t) \mbox{ as in (\ref{R eq1})}.\\
\end{array}\right.
\end{equation}
The elements $u^\eta_k(t)$ and $u^\eta_h(s)$ commute. If $k\neq h$ then $u^+_k(t)$ commutes with $u^-_h(s)$. Hence $m^\eta_k(t){\ast}u^\eta_h(t) + m^\eta_h(t) = m^\eta_k(t){\ast}u^\eta_k(t) + m^\eta_k(t)$ and $m^+_k(t){\ast}u^-_h(t) + m^-_h(t) = m^-_k(t){\ast}u^+_k(t) + m^+_k(t)$. From these equalities and (\ref{R eq2}) we get the following improvement of (\ref{R eq2}). 
 
\begin{result}\label{R1} 
As stated in {\rm (\ref{R eq2})}, for every $k = 1, 2, ..., n$ and every $t\in \FF$ we have $B^+_k(t) = C^-_k(t) = O_n$, $A^+_k(t) = \sum_{i=1}^na^+_{i,k}(t)e_{i,k}$ and $A^-_k(t) = \sum_{i=1}^na^-_{k,i}(t)e_{k,i}$ for suitable scalars $a^+_{i,k}(t)$ and $a^-_{k,j}(t) \in \FF$. Moreover $\ba^+_k(t) = \bfb^-_k(t) = {\bf 0}$.

If $|\FF| > 2$ then $\bfb^+(t) = \ba^-_k(t) = {\bf 0}$ and $C^+_k(t) = B^-_k(t) = O_n$. 

Let $\FF = \FF_2$ and put $J_k := \sum_{i\neq k}(e_{k,i}+ e_{i,k})$. Then there exist scalars $c^+_k, c^-_k \in \FF$ such that $C^+_k(1) = c^+_kJ_k$, $\bfb^+_k(1) = c^+\be_k$, $B^-_k(1) = c^-_kJ_k$ and $\ba^-_k(1) = c^-_k\be_k$ 
\end{result} 

\begin{proposition}\label{final}
We have $F(u_k^\eta(t)) = 0$ for every $k = 1, 2,..., n$, every $t\in \FF$ and $\eta \in \{+,-\}$.
\end{proposition}
\noindent
{\em Proof.} For $k \neq h$ we have
\[u_{k,h}(s)u^+_k(t)u_{k,h}(s)u^+_k(t) ~ = ~ \left\{\begin{array}{ll}
  u^+_h(st)u^+_{h,k}(st^2) & \mbox{if } h < k, \\
u^+_h(st)u^+_{k,h}(st^2) & \mbox{if } k < h.
\end{array}\right.\]
Therefore, recalling that $F(u_{k,h}(t)) = F(u^+_{k,h}(st^2)) = 0$, we have
\begin{equation}\label{R eq3}
(u_{k,h}(s)m^+_k(t)u_{k,h}(s)){\ast}u^+_k(t) + m^+_k(t) ~ = ~ \left\{\begin{array}{ll}
u^+_{k,h}(st^2)m^+_h(st)u^+_{h,k}(st^2) & \mbox{if } h < k, \\
u^+_{k,h}(st^2)m^+_h(st)u^+_{k,h}(st^2) & \mbox{if } k < h. 
\end{array}\right. 
\end{equation} 
We can compute both sides of  (\ref{R eq3}) with he help of Lemma\ref{L*0 lem} and, taking Result \ref{R1} into account, we see that $\beta_k(t) = a_{h,k}(t) = 0$ and
\begin{equation}\label{R eq4}
s^2a_{h,k}(t) + sa_{k,k}(t) ~ = ~ a_{k,h}(st).
\end{equation}
As $k$ and $h$ are abitrary, the same holds with $k$ and $h$ permuted. In particular $a_{k,h}(r) = 0$ for every $r\in \FF$. Therefore $a_{h,k}(t) = a_{k,h}(st) = 0$. Form  equation (\ref{R eq4}) with $a_{h,k}(t) = a_{k,h}(st) = 0$ we obtain that $a_{k,k}(t) = 0$. So $a_{i,k}(t) = \beta_k(t) = 0$ for every $t$, every $k$ and every $i$. Therefore $m^+_k(t) = O$, namely $F(u^+_k(t)) = 0$. Similarly, $F(u^+_k(t)) = 0$.  \hfill $\Box$\\

As the elements $u^\eta_k(t)$ generate $R^\eta$ and $G$ is generated by $G_D\cup R^+\cup R^-$ and $F$ is null over $G_D$ by assumption, Proposition \ref{final} implies that $F  = 0$. Proposition \ref{Main B bis} is proven. The proof of Theorem \ref{even main1} is complete.  

\appendix 

\section{Extensions of modules}\label{teoria}

In this appendix we fix some terminology and recall some basics on extensions of modules, focusing on extensions $U = W\cdot V$ where $\dim(W) = 1$. 

Throughout this appendix $G$ and $\FF$ are a given group and a given field. A (right) $G$-module is a pair $(V,\rho)$ where $V$ is an $\FF$-vector space and $\rho$ is a homomorphism from $G$ to the group of invertible linear transformations of $V$ such that $\rho(g_1g_2)(v) = \rho(g_2)(\rho(g_1)(v))$ for any two elements $g_1, g_2\in G$ and every $v\in V$. By a little abuse, if a homomorphism $\rho$ as above is given we say that $V$ is $G$-module, also that the pair $(V,G)$ a module.  

If $V$ is a $G$-module we write $v\cdot g$ for $\rho(g)(v)$, for every $g\in G$ and $v\in V$. We extend this notation to subspaces of $V$ and subgroups of $G$ in a natural way. Explicitly, let $X$ and $v$ a subset and a vector of $V$ respectively and $F$ and $g$ a subset and an element of $G$. Then we put $X\cdot g : = \{v\cdot g\}_{v\in X}$, $vF := \{vg\}_{g\in F}$ and $X\cdot F = \cup_{v\in X}v\cdot F = \cup_{g\in F}X\cdot g$. In particular, $v\cdot G$ is the  \emph{orbit} of $v$ under $G$ and if $X\cdot G =  X$ we say that $G$ \emph{stabilizes} $X$. We also take the liberty of writing $g(v)$, $g(X)$, $F(v)$ and $F(X)$ instead of  $v\cdot g$, $X\cdot g$, $v\cdot F$ and $X\cdot F$ respectively.      

Given a $G$-module $V$, the action of $G$ on the dual $V^*$ of $V$ is defined as follows: 
\begin{equation}\label{GsuV*}
\xi\cdot g ~:~ x\in V ~ \longrightarrow ~ \xi(x\cdot g^{-1}), \hspace{5 mm} \forall g\in G, ~\xi \in V^*.
\end{equation}
In short, $\xi\cdot g = \xi\circ g^{-1}$, where $\circ$ stands for composition of mappings. According to this definition, we have $\xi\cdot(g_1g_2) = (\xi\cdot g_1)\cdot g_2$ for every choice of $\xi\in V^*$ and $g_1, g_2\in G$.   

Two $G$-modules $V$ and $V'$ defined over the same field $\FF$ are  {\em isomorphic} if there is a bijective linear mapping $f:V\rightarrow V'$ such that $f(v\cdot g) = f(v)\cdot g$ for every $v\in V$ and every $g\in G$.  

\subsection{Generalities on extensions} 

Given a $G$-module $U$, a  \emph{submodule} of $U$ is a subspace of $U$ stabilized by $G$ in its action on $U$. If $W$ is a submodule of $U$ then an action of $G$ is naturally defined on the quotient $U/W$ as follows: $(u+ W)\cdot g := u\cdot g + W$ for every coset $u+W$ of $W$ and every $g\in G$.  

Given another $G$-module $V$, the module $U$ is an \emph{extension} of $V$ if $U$ admits a submodule $W$, called the  \emph{kernel} of the extension, such that $V\cong U/W$ (isomorphism of $G$-modules). We write $U = W\cdot V$ to mean that $U$ is an extension of $V$ with kernel $W$.

If $\dim(W)=d$ we say that $U$ is a $d$-\emph{extension} of $V$. In particular, if $d = 1$ then $U$ is a $1$-\emph{extension} of $V$. 
In case $W=\{0\}$ the extension $U$ is \emph{trivial}.

We say that a non-trivial extension $U = W\cdot V$  \emph{splits} if there is an injective linear mapping $\phi:V\to U$ such that $\phi(V)$ is stabilized by $G$ and $\pi\circ\phi= \mathrm{id}_V$, where $\pi$ is the canonical projection of $U$ onto $U/W \cong V$ as vector spaces. If this is the case then $U$ is the direct sum of $W$ and $\phi(V)$ (direct sum of $G$-modules, of course); with a little abuse, $U = W\oplus V$ (or $U = W\times V$, if we like products more than sums). A {\em non-split} extension is an extension that does not split. 

Given a subspace $W_0$ of $W$ stabilized by $G$, the quotient $U/W_0$ is again an extension of $V$ called a  \emph{quotient} of the extension $U$ (a  \emph{proper quotient} if $W_0 \subsetneq W$). 

\begin{definition}\label{def1}  
\em
Let $U=W\cdot V$ be an extension of $V.$ We say that the extension $U$ is
\begin{itemize}
\item \emph{central} if $G$ acts trivially on the kernel $W$ of $U$;
\item \emph{totally non-split} if every proper quotient of $U$ is non-split as an extension of $V$;  
\item \emph{$1$-universal} if every non-split $1$-extension of $V$ is isomorphic to a quotient of $U$ over a hyperplane of $W$ stabilized by $G$;
\item $1$-\emph{non-split} if $U$ is central and $U/H$ is non-split for every hyperplane $H$ of $W$. 
  \end{itemize}
\end{definition}

\begin{remark}\label{Gperf-1-et}
\em
If a group acts on a $1$-dimensional $\FF$-vector space, then it
necessarily acts on it as a subgroup of the multiplicative group of $\FF$. Consequently, the trivial action is the unique action of a perfect group on a $1$-dimensional vector space. So, if $G$ is perfect then all $1$-extensions of
a $G$-module $V$ are central.
\end{remark} 

Let $U=W\cdot V$ be an extension of $V$ and let $X$ is a complement of $W$ in the vector space $U$. The elements of $V$ can be regarded as cosets of $W$. Thus, for every $v\in V$ the intersection $v_X:=v\cap X$ is an element of $X.$
The map $\iota_X \colon V \rightarrow X,\, \iota_X(v)= v_X$ is a vector space  isomorphism. For every $g\in G$, we define an action $g_X$ of $g$ on $X$ as follows:  
\begin{equation}\label{gX}
g_X ~\colon ~ x\in X ~\longrightarrow ~((x+W)\cdot g)_X.
\end{equation}
Since $(x+W)\cdot g  = x\cdot g + W$ (with $x\cdot g$ computed in $U$), this action is well defined. According to \eqref{gX}, the vector space $X$ can be regarded as a $G$-module and $\iota_X$ is a module isomorphism. If $X$ is not stabilized by $G$ in the action of $G$ on $U$, the subspace $X$ is not a submodule of $U.$ However $v_X\cdot g -  (v\cdot g)_X \in W$ for every $v\in V$ and $g\in G$. We put 
\begin{equation}\label{wX}
w_X(v,g) ~ :=  ~ v_X\cdot g-(v\cdot g)_X 
\end{equation}
and  define $W_X$ as the subspace of $W$ spanned by the elements $w_X(v,g)$'s for $v\in V$ and $g\in G$. The group $G$ acts as follows on $W_X:$
\begin{equation}\label{cociclo}
	w_X(v,g_1)\cdot g_2 ~ = ~ w_X(v,g_1g_2)-w_X(vg_1,g_2).
\end{equation}
Also, according to \eqref{wX}
\begin{equation}\label{pre cociclo}
w_X(v,1) ~ = ~  v_X - v_X = 0
\end{equation}
for every $v\in V$. As a consequence of equation~\eqref{cociclo}, we have the following. 

\begin{lemma}\label{WX}
The group $G$ stabilizes the subspace $W_X$ of $W$.  
\end{lemma} 

The subspace $X$ is stabilized by $G$ if and only if $W_X=\{0\}.$ So, $U$ is non-split  if and only if $W_X\not=\{0\}$ for every complement $X$ of $W$. Combining these remarks with Lemma \ref{WX} we immediately obtain the frst claim of the following proposition. The second claim  is not difficult to prove (anyway, see \cite[Lemma 2.5]{CGP2}). 

\begin{proposition}
The extension $U=W\cdot V$ is totally non-split if and only if $W_X=W$ for every complement $X$ of $W$ in $U.$ Moreover, if $U$ is central then $U$ is totally non-split if and only if it is $1$-non-split. 
\end{proposition}  

\subsection{Central $1$-extensions} 

It is well known that the elements of the first cohomology group $H^1(V^*,G)$ of a group $G$ acting over the dual $V^* $ of a $G$-module $V$ correspond in a standard way to the non-split central $1$-extensions of the $G$-module $V$. In this subsection we shall discuss this claim. We shall also recall the main lines of its proof, since in this case the proof is more enlighting than the statement ifself. We not going to recall basics on cohomology of groups. We refer to Hall \cite[chapter 15]{Hall} for them. Several presentations are available for this matter, more comprehensive than \cite{Hall}, but \cite{Hall} is enough for our needs. 

Let $U = W\cdot V$ be a central extension of $V$. As $G$ acts trivially on $W$ equation~\eqref{cociclo} yields
\begin{equation}\label{wXgv} 
	w_X(v,g_1)=w_X(v,g_1g_2)-w_X(vg_1,g_2), ~~~ \forall g_1, g_2\in G, ~\forall v\in V. 
\end{equation}
For every $v\in V$ define $w^*_X(v, . ):G\rightarrow W$ as follows: 
 \begin{equation}\label{cociclo bis}
w^*_X(v,g)~ := ~ w_X(v,g^{-1}), ~~~ \forall g\in G. 
\end{equation}
With this notation, we can rewrite \eqref{cociclo} as follows: 
\[w^*_X(v, g_2^{-1}g_1^{-1}) ~ = ~  w^*_X(vg_1, g_2^{-1}) + w^*_X(v, g_1^{-1}), ~~ \forall g_1, g_2\in G, ~ \forall v\in V.\]
Replacing $g_1^{-1}$ with $g_2$ and $g_2^{-1}$ with $g_1$, the above yields
\begin{equation}\label{cociclo ter}
w^*_X(v, g_1g_2) ~ = ~ w^*_X(vg_2^{-1}, g_1) + w^*_X(v, g_2).
\end{equation}  
Given $g\in G$, we can regard $w^*_X( . ,g)$ as a linear mapping $f_X(g):v\rightarrow w^*_X(v,g)$ from $V$ to $W$ (in fact to $W_X$) and $w^*_X( . , . )$ as the mapping $f_X$ from $G$ to the space $L(V, W)$ of linear mappings from $V$ to $W$ which maps every $g\in G$ onto $f_X(g)$. The mapping $w^*_X(( . )g_2^{-1}, g_1) : v \rightarrow w^*_X(vg_2^{-1}, g_1)$ is the same as $f_X(g_1)\cdot g_2$ and we can rewrite \eqref{cociclo ter} as follows
\begin{equation}\label{cociclo quater}
f_X(g_1g_2) ~ = ~ f_X(g_1)\cdot g_2 + f_X(g_2).
\end{equation}
Formula \eqref{pre cociclo} implies that $f_X(1)(v) = 0$ for every $v\in V$, namely $f_X(1) = 0$. Formula \eqref{cociclo quater} now makes it clear that $f_X$ is a $1$-cocycle of $G$ over $L(V,W)$. 

Let now $\alpha\in W^*$ be a linear functional of $W$ and, with $f_X$ as above, for every $g\in G$ put 
\begin{equation}\label{alphaX}
\alpha_X(g) ~ :=~   \alpha\circ f_X(g) =  \alpha\circ w_X( . , g^{-1}).
\end{equation}
So, $\alpha_X(g)$ is the linear functional of $V^*$ which maps every $v\in V$ onto
$\alpha(w_X(v,g^{-1}))$. As $f_X\in Z^1(L(V,W), G)$, the map $\alpha_X$ is a $1$-cocycle of $G$ over the dual $V^*$ of $V$, i.e. $\alpha_X\in Z^1(V^*, G)$.

We are now ready to prove the main theorem of this subsection. In order to state it properly, we need some conventions and a definition.
Given a $G$-module $V$, every central $1$-extension of $V$ can be realized in $\FF\times V$ by choosing a function $\phi: V\times G\rightarrow \FF$ such that
\begin{equation}\label{Phi(v,g)}
\phi(v, g_1g_2) ~ = ~ \phi(v, g_1) + \phi(v\cdot g_1, g_2)  ~~~ \forall v\in V, ~ \forall g_1, g_2\in G 
\end{equation} 
(compare (\ref{wXgv})) and defining the action of $G$ on $\FF\times V$ as follows:
\begin{equation}\label{1-ext-1}
g ~ : ~ (t,v)\in \FF\times V ~ \rightarrow ~ (t+\phi(v,g), v\cdot g), ~~~ \forall g\in G, ~ \forall v\in V.
\end{equation} 
Note that $\phi(v,1) = 0$ for every $v\in V$, as one can see but putting $g_1 = 1$ in \eqref{Phi(v,g)}. We denote by $U_\phi$ the extension of $V$ realized in this way.        

Given two extensions $U = W\cdot V$ and $U' = W'\cdot V$ of the same $G$-module $V$, an  \emph{isomorphism} from the extension $U$ to the extension $U'$ is just an isomorphism of $G$-modules from $U$ to $U'$ which maps the kernel $W$ of $U$ onto the kernel $W'$ of $U'$.  However, this notion of isomorphism is unsuitable for the next theorem. The notion we need in it is the following one. 

Let $U = W\cdot V$ and $U' =  W'\cdot V$ be two extensions of $V$ with the same underlying vector-space structure, namely $U = U'$ and $W = W'$ as vector spaces. We say that an isomorphism of extensions from $U$ to $U'$ is  \emph{rigid} (with respect to the given underlying vector-space structure) if it induces the identity mapping on both $W$ and $U/W$. Explicitly, if $U = W\cdot V$ and $U' = W'\cdot V$ have the same vector space structure, say  $U = U' = W\times X$ as vector spaces for a given copy $X$ of the $G$-module $V$, and $\psi:U\rightarrow U'$ is a rigid isomorphism from $U$ to $U'$, then there exists a linear mapping $\lambda:X\rightarrow W$ such that $\psi$ maps every vector $(w,v)\in W\times X$ onto $(w+\lambda(v), v)$. Of course, for $\psi$ to be an isomorphism of $G$-modules the mapping $\lambda$ must satisfy the following condition:
\begin{equation}\label{rigid eq}
\lambda(v\cdot g) - \lambda(v)\cdot g ~ = ~ w'_X(v, g) - w_X(v, g),~~~ \forall g\in G, ~ \forall v\in V,
\end{equation} 
with $w_X(v,g)$ and $w'_X(v,g)$ defined according to \eqref{wX} in $U$ and $U'$ respectively.  

Keeping the assumption that $U$ and $U'$ have the same vector-space structure,  \emph{semi-rigid} isomorphisms from $U$ to $U'$ can also be considered. We define them by dropping the requirement to induce the identity on $W$ but keeping the hypothesis that the identity is induced on $U/W$. 

\begin{theorem}\label{thm link}
  The rigid isomorphism classes of the central $1$-extensions of $V$ bijectively correspond to the elements of $H^1(V^*,G)$. In particular, all split $1$-extensions of $V$ are pairwise rigidly isomorphic and their
rigid isomorphism class corresponds to the null element of $H^1(V^*,G)$.  
\end{theorem}
{\em Proof.}  Let $U_\phi$ be the (central) $1$-extension of $V$ defined by a mapping $\phi$ as previously explained and let $f_\phi:G\rightarrow V^*$ be the function which maps every $g\in G$ onto the linear functional $f_\phi(g) : v\in V \rightarrow \phi(v, g^{-1})$. Property~\eqref{Phi(v,g)} on $\phi$ implies that $\phi(v,1) = 0$ for every $v\in V$. Hence $f_\phi(1) = 0$ and $f_\phi$ is a $1$-cocycle of $G$ over $V^*$, namely $f_\phi(1) = 0$ and
\begin{equation}\label{basic formula} 
f_\phi(g_1g_2) ~ = ~ f_\phi(g_1)\cdot g_2 + f_\phi(g_2), ~~~ \forall g_1, g_2\in G.
\end{equation} 
Conversely, for every cocycle $f\in Z^1(V^*, G)$ the mapping $\phi_f:V\times G\rightarrow \FF$ defined by the clause $\phi_f(v, g) := f(g^{-1})(v)$ yields a central $1$-extension $U_f$ of $V$. Of course, $U_{f_\phi} = U_\phi$ and $U_{\phi_f} = U_f$. 

The extension $U_\phi$ splits if and only if there exists a linear functional $\lambda\in V^*$ such that the subspace $V_\lambda = \{(\lambda(v), v)\}_{v\in V}$ is stabilized by $G$ in its action on $\FF\times V$ defined as in \eqref{1-ext-1}. This amounts to $\lambda(v) + \phi(v,g) = \lambda(v\cdot g)$, which in turn is equivalent to $f_\phi(g) = \lambda\cdot g - \lambda$. So, $U_\phi$ splits only if $f_\phi$ is a $1$-coboundary. The above also shows that, conversely, if $f\in B^1(V^*, G)$ is a $1$-coboundary then $U_f$ splits. 

So far we have proved that the central $1$-extensions defined on $\FF\times V$ by a clause like~\eqref{1-ext-1} bijectively correspond to the elements of $Z^1(V^*,G)$, the split ones corresponding to the elements of $B^1(V^*,G)$. In order to finish the proof it remains to prove that, with $U = U_\phi$, $U' = U_{\phi'}$, $f = f_\phi$ and $f' = f_{\phi'}$, the extensions $U$ and $U'$ are rigidly isomorphic if and only if $f'-f \in B^1(V^*,G)$.  

Let $\psi$ be a rigid isomorphism from $U$ to $U'$. Then there exists a linear functional $\lambda\in V^*$ such that $\psi$ maps $(t,v)\in U$ onto $(t+\lambda(v), v)\in U'$. In the present setting, Condition~\eqref{rigid eq} amounts to the following
\begin{equation}\label{iso2}
\phi(v,g) + \lambda(v\cdot g)  ~ = ~ \lambda(v) + \phi'(v,g), ~~~ \forall v\in V, ~ \forall g\in G. 
\end{equation}
Hence $\phi'(v,g) - \phi(v,g) = \lambda(v\cdot g) - \lambda(v)$. Equivalently, $f'(g^{-1})(v) - f(g^{-1})(v) = (\lambda\cdot g^{-1})(v) - \lambda(v)$ for every $v\in V$ and every $g\in G$. In short, $f' - f$ is the same as the coboundary $g \rightarrow \lambda\cdot g - \lambda$. Hence $f$ and $f'$ represent the same element of $H^1(V^*,G)$.  

Conversely let $f' = f + \beta$ for $\beta \in B^1(V^*,G)$, say $\beta(g) = \lambda - \lambda\cdot g$ for every $g\in G$ and a suitable $\lambda\in V^*$. Then $U$ and $U'$ are rigidly isomorphic, a rigid isomorphism $\psi:U\rightarrow U'$ being defined as follows: $\psi(t,v) = (t- \lambda(v), v)$.  \hfill $\Box$\\

Recall that $H^1(V^*,G)$ is a vector space over the same field $\FF$ as $V$.  So, we can consider the projective geometry $\PG(H^1(V^*, G))$ of the linear subspaces of $H^1(V^*, G)$. Keeping the setting implicit in Theorem \ref{thm link} for central $1$-extensions of $V$, consider semi-rigid isomorphisms of central $1$-extensions of $V$. In this setting, a semi-rigid isomorphism is the composition of a rigid isomorphism with a rescaling $t \rightarrow kt$ of the $1$-dimensional vector space $\FF$. Up to a few minor modifications, the proof of Theorem \ref{thm link} also yields the following.

\begin{corollary}\label{cor link}
The semi-rigid isomorphism classes of the non-split central $1$-extensions of $V$ bijectively correspond to the points of $\PG(H^1(V^*,G))$. 
\end{corollary} 

\subsection{Central extensions of rigid $G$-modules}\label{M*sec} 

In the sequel $V$ is a given rigid $G$-module (Definition \ref{rigid}). The following is Lemma 2.8 of \cite{CGP2}. 

\begin{theorem}\label{M* lemma} 
Let $U = W\cdot V$ be a $1$-non-split central extension of $V$. Then $U/H_1\not\cong U/H_2$ (as $G$-modules) for any choice of distinct hyperplanes $H_1, H_2$ of $W$.
\end{theorem}
\noindent
We refer to \cite{CGP2} for the proof. As a by-product of that proof we also obtain that two central $1$-extensions of a given rigid $G$-module are isomorphic (if and) only if they are semi-rigidly isomorphic. Theorem \ref{M* lemma} implies the following (see \cite[Theorem 2.9]{CGP2}.

\begin{corollary}\label{M*}
Let $U = W\cdot V$ be a central extension of $V$. If $U$ is both $1$-non-split and $1$-universal then $H^1(V^*,G)\cong W^*$.
\end{corollary}

\section{The relatively universal cover of a given embedding}\label{AppEmb}

Given a projective embedding $e:\Gamma\rightarrow V$, for every point $p$ (line $\ell$) of $\Gamma = (P,L)$ let $V_p$ (respectively $V_\ell$) be the $1$-dimensional ($2$-dimensional) subspace of $V$ corresponding to $e(p)$ (respectively $e(\ell)$). For every $p\in P$ and every $\ell \in L$ let ${\bf  V}_p$ and ${\bf V}_\ell$ be copies of $V_p$ and $V_\ell$, respectively. All of these copies are assumed to be pairwise disjoint but for sharing the same null vector. It is understood that for every $p\in P$ (every $\ell \in L$) an isomorphism $\iota_p:{\bf V}_p\rightarrow V_p$ (respectively $\iota_\ell:{\bf V}_\ell\rightarrow V_\ell$) is given. For every $p\in P$ and every $\ell\in L$ containing $p$, we put $\iota_{p,\ell} = \iota_\ell^{-1}\circ \iota_p$. Consider now the direct sum 
\[{\bf V} ~ := ~ (\bigoplus_{p\in P}{\bf V}_p)\oplus(\bigoplus_{\ell \in L}{\bf V}_\ell)\]
and let $W$ be the subspace of $\bf V$ spanned by the set 
\[J ~ = ~ \{{\bf v} - \iota_{p,\ell}({\bf v}) ~:~ {\bf v}\in {\bf V}_p, ~ p\in P, ~ \ell\in L, ~ p\in \ell\}.\]
Put $\widehat{V} := {\bf V}/W$. Then $\widehat{V}$ hosts a projective embedding $\hat{e}$ of $\Gamma$ and $\hat{e}$ covers $e$.  Explicitly, $\hat{e}(p) = \widehat{V}_p := ({\bf V}_p + W)/W$ and $\hat{e}(\ell) = \widehat{V}_\ell := ({\bf V}_\ell + W)/W$ for every $p\in P$ and every $\ell\in L$. The natural projection $\pi_V$ of $\bf V$ onto $V$, which maps ${\bf V}_p$ onto $V_p$ according to $\iota_p$ and ${\bf V}_\ell$ onto $V_\ell$ according to $\iota_\ell$, induces a projection $\widehat{\pi}:\widehat{V}\rightarrow V$ of $\hat{e}$ onto $e$. Obviously, $\ker(\widehat{\pi}) = \ker(\pi_V)/W$. 

\begin{proposition}[{\rm Ronan \cite{R}}]
The embedding $\hat{e}$ is the universal cover of $e$. 
\end{proposition}

\section{Derivations}\label{AppDer}  

Given a field $\FF$, a derivation of a field $\FF$ is an additive maping $d:\FF\rightarrow \FF$ such that
\begin{equation}\label{der def}
d(ts) ~ = ~ d(t)s+td(s), \hspace{5 mm} \forall t, s\in \FF. 
\end{equation} 
As $1 = 1\cdot 1$, property (\ref{der def}) implies that $d(1) = 0$. Therefore if $\FF_0$ is the minimum subfield of $\FF$ then $d(t) = 0$ for every $t\in \FF_0$ and it is not hard to prove that if $t\in\FF$ is algebraic over $\FF_0$ then $d(t) = 0$. Note also that $d(t^{-1}) = -t^{-2}d(t)$ for every $t\in \FF\setminus\{0\}$ and $d(t^p) = pt^{p-1}d(t)$ for every positive integer $p$. In particular, if $\ch(\FF) = p > 0$ then $d(t^p) = 0$ for every $t\in \FF$.

The derivations of $\FF$ form an $\FF$-vector space $\Der(\FF)$, which we are now going to describe. The field $\FF$ admits a maximum subfield $N_{\mathrm{der}}(\FF)$ such that every derivation of $\FF$ induces the null map on it. When $\ch(\FF) = p > 0$ then $N_{\mathrm{der}}(\FF) = \FF^p := \{t^p\}_{t\in\FF}$ while when $\ch(\FF) = 0$ then $N_{\mathrm{der}}(\FF)$ is the algebraic closure $\overline{\FF}_0$ in $\FF$ of the minimum subfield $\FF_0$ of $\FF$. (Note that if $\ch(\FF) = p > 0$ then $\FF^p\supseteq \overline{\FF}_0$.) A {\em derivation basis} of $\FF$ is subset $\Omega$ of $\FF \setminus N_{\mathrm{der}}(\FF)$ such that every map from $\Omega$ to $\FF$ is induced by a unique derivation of $\FF$. So, $\Der(\FF)$ is isomorphic to the space $\FF^\Omega$ of all mappings from $\Omega$ to $\FF$.

Every field admits a derivation basis (possibly the empty one) and all derivation bases of $\FF$ have the same cardinality, which we call the {\em derivation rank} of $\FF$, denoting it by $\drk(\FF)$. Explicitly, if $\ch(\FF) = 0$ then the derivation bases of $\FF$ are its trascendence bases (in this case $\drk(\FF)$ is the trascendence degree of $\FF$) while when $\ch(\FF) = p > 0$ then the derivation bases are the subsets $X\subset \FF\setminus\FF^p$ such that $\FF^p\cup X$ generates $\FF$ as a field and $X$ is minimal with respect to this property.  In particular, $\drk(\FF) = 0$ (namely $\Der(\FF) = \{0\}$) if and only if $\FF = N_{\mathrm{der}}(\FF)$, namely $\FF$ is either perfect of positive characteristic or algebraic over its minimum subfield.

Given a derivation basis $\Omega$ of $\FF$, fo every $\omega\in \Omega$ we denote by $d_\omega$ the derivation which maps $\omega$ onto $1\in\FF$ and all elements of $\Omega\setminus\{\omega\}$ onto $0\in\FF$. We denote by $\Der_\Omega(\FF)$ the subspace of $\Der(\FF)$ spanned by $\{d_\omega\}_{\omega\in\Omega}$. The space $\Der(\FF)$ is isomorphic to the dual of $\Der_\Omega(\FF)$ in a natural way and we have $\Der_\Omega(\FF)\cong \Der(\FF)$ if and only if $\Omega$ is finite.


\begin{thebibliography}{99}
\bibitem{B} R. J. Blok. The generating rank of the symplectiv line-grassmannian, {\em Beitr. Algebra Geom.}, {\bf 44} (2003), 575--580. 
\bibitem{BB} R. J. Blok and A. E. Brouwer. Spanning point-line geometries in buildings of spherical type,  {\em J. Geom.} \textbf{62} (1998), 26--35.
\bibitem{BC} R. J. Blok and B. N. Cooperstein. The generating rank of the unitary and symplectic grassmannians. {\em J. Combin. Theory Ser. A}  {\bf 119} (2012), 1--13.
\bibitem{BP1} R. J. Blok and A. Pasini. Point-line geometries with a generating set that depends on the underlying field,
  in: {\em Finite Geometries} (eds. A. Blokhuis et al.), Kluwer, Dordrecht (2001), 1--25.
\bibitem{BP2} R. J. Blok and  A. Pasini. On absolutely universal embeddings,
  {\em Discr. Math.} \textbf{267} (2003), 45--62.
\bibitem{CGP1} I. Cardinali, L. Giuzzi and A. Pasini. On the gneration of polar grassmannians, in {\em Algebraic Combinatorics and the Moster Group} (ed. A. A. Ivanov), London Math. Soc. L.N.  {\bf 487}, Cambridge Univ. Press, Cambridge (2024), 327--368.  
\bibitem{CGP2} I. Cardinali, L. Giuzzi and A. Pasini. The relatively universal cover of the natural embedding of the long root geometry for the group $\mathrm{SL}(n+1,\KK)$, {\em Alg. Combin.} {\bf 9} (2026), 231-259. 
\bibitem{Coop1}  B. N. Cooperstein.  Generating long root subgroup geometries of classical groups over finite prime fields,
  {\em Bull. Belg. Math. Soc. } \textbf{5} (1998), 531--548. 
\bibitem{Coop2} B. N. Cooperstein. On the generation of dual olar spaces of symplectic type over finite fields, {J. Combin. Th. A} {\bf 83} (1998), 221--232. 
\bibitem{Coop3} B. N. Cooperstein. Generation of embeddable incidence geometries: a survey, in {\em Topics in Diagram Geometry} (ed. A. Pasini), Vol. 12 of the series {\em Quaderni di Matematica}, II Univ. Naples  (2003), 29-57.
\bibitem{CS1} B. N. Cooperstein and E. E. Shult. Frames and bases of lie incidence geometries, {\em J. Geom.} \textbf{60} (1997), 17-46.
\bibitem{CS2} B. N. Cooperstein and E. E.Shult. A note on embedding and  generating dual polar spaces, {\em Adv. Geom.} {\bf 1} (2001), 37--48.
\bibitem{DSSVM} A. De Schepper, J. Schillewaert and H. van Maldeghem. On the generating rank and embeding rank of the hexagonic Lie incidence geometries, {\em Combinatorica} {\bf 44} (2024), 355--392. 
\bibitem{GP} L. Giuzzi and A. Pasini. Non existence of absolutely universal embeddings of certain Lie incidence geometries, in preparation.  
\bibitem{Hall} M. Hall Jr. \textit{The Theory of Groups}, Macmillan, New York, 1959.
\bibitem{H-Lie} J. E. Humphreys, {\em Introduction to Lie Algebras and Representation Theory}, Springer, 1972. 
\bibitem{H-mod}  J. E. Humphreys, {\em Modular Representations of Finite Groups of Lie Type}, Cambridge Univ. Press, 2006.
\bibitem{KS} A. Kasikova and E. E. Shult, Absolute embeddings of point-line geometries,
 {J. Algebra} \textbf{238} (2001), 265--291.
\bibitem{Li1} P. Li. On the universal embedding of the $Sp_{2n}(2)$ dual polar space, {\em J. Combin. Th. A} {\bf 94} (2001), 100-117.
\bibitem{Li2} P. Li. On the Universal Embedding of the $U_{2n}(2)$ dual polar space, {\em J. Combin. Th. A} {\bf 98} (2002), 235-252.
 \bibitem{Pas}  A. Pasini.  Embeddings and hyperplanes of the Lie incidence geometry $A_{n,\{1,n\}}(\FF)$,
 {Comb. Theory}, \textbf{4}, No. 2 (2024), paper No. 5.
\bibitem{R}  M. A. Ronan. Embeddings and hyperplanes of discrete geometries
 {\em European J. Combin.} \textbf{8} (1987), 179--185.
\bibitem{S93} E. E. Shult. Embeddings and hyperplanes of Lie incidence geometries,  in  \emph{Groups of Lie Type and Their Geometries} (ed. W. M. Kantor and L. Di Martino), London Math. Soc. Lect. Notes \textbf{207}, Cambridge Univ. Press, Cambridge 1995, 215-232.
\bibitem{S} E. E. Shult.  \textit{Points and Lines}, Springer-Verlag, Berlin 2011.
\bibitem{St} R. Steinberg, {\em Lectures on Chevalley Groups}, A.M.S., 2016. 
\bibitem{SV} S.D. Smith and H.~V\"{o}lklein. A geometric presentation for the adjoint module of $\mathrm{SL}_3(\KK)$,
{\em J. Algebra} \textbf{127} (1989), 127--138.
\bibitem{TZ} O. Taussky and H. Zassenhaus. On the 1-cohomology of the general and special linear groups, {\em Aequationes Math.} \textbf{5} (1970), 129-201.
\bibitem{TVM} J. A. Thas and H. Van Maldeghem. Classification of embeddings of the flag geometries of projective planes in finite projective planes, Parts I, II and III, {J. Combin. Th. A}, {\bf 90} (2000), 159--172 (Part I), 241--256 (Part II) and 173--196 (Part III). 
\bibitem{T} J. Tits. {\em Buildings of Spherical Type and Finite $BN$ pairs}, Springer L. N. in Maths. \textbf{386}, Springer-Verlag, Berlin 1974. 
\bibitem{V1} H.~V\"{o}lklein. The $1$-Cohomology group of the Adjoint Module of a Chevalley Group, 
{\em Forum Math.} \textbf{1} (1989), 1-13.
\bibitem{V2} H.~V\"{o}lklein. On the geometry of the adjoint representation of a Chevalley group, {\em J. Algebra} \textbf{127} (1989), 139--154.
\end{thebibliography}
\end{document}